\documentclass{amsart}

\usepackage{amsmath,amsfonts,amssymb}

\def\trans{\xi}
\def\addfn{Z}
\def\cendep{\tilde{D}}
\def\integspace{{\cal C}_{\chatm}}

\def\cad{{c\`{a}dl\`{a}g }}
\def\are{\Phi}
\def\opint{\Theta}
\def\cb{\C_b}

\def\testspace{{\cal S}}
\def\dualspace{{\cal S}'}

\def\cg{\overline{G}}

\def\hatm{\widehat{M}}
\def\hatmn{\widehat{M}^{N}}

\def\meas{\nu}

\def\L{{\mathbb L}}
\def\H{{\mathbb H}}
\def\R{{\mathbb R}}
\def\N{{\mathbb N}}
\def\C{{\mathbb C}}
\def\Z{{\mathbb Z}}
\def\ind{{1\!\!1}}
\def\E{\mathbb{E}}
\def\P{\mathbb{P}}

\def\D{{\mathbb D}}

\def\cal{\mathcal}
\def\dom{\rm{dom}\,}

\def\supp{{\rm{supp}}}
\def\endsup{L}
\def\altendsup{\ell'}
\def\hconst{C}
\def\hexp{\gamma}

\def\del{\nabla}
\def\ops{{\cal S}}

\def\theequation{\arabic{section}.\arabic{equation}}
\def\thetheorem{\arabic{section}.\arabic{theorem}}

\newcounter{bean}
\newcommand{\benuma}{\setlength{\labelwidth}{.25in}
\begin{list}%
{(\alph{bean})}{\usecounter{bean}}}
\newcommand{\eenuma}{\end{list}}

\newcommand{\beginsec}{\setcounter{equation}{0}}

\newtheorem{theorem}{Theorem}[section]
\newtheorem{remark}[theorem]{Remark}
\newtheorem{ex}[theorem]{Example}
\newtheorem{lemma}[theorem]{Lemma}
\newtheorem{cor}[theorem]{Corollary}
\newtheorem{defn}[theorem]{Definition}

\newtheorem{ass}{Assumption}
\newtheorem{prop}[theorem]{Proposition}

\newcommand{\noi}{\noindent }
\newcommand{\ba}{\begin{array}}
\newcommand{\ea}{\end{array}}
\newcommand{\bea}{\begin{eqnarray}}
\newcommand{\eea}{\end{eqnarray}}
\newcommand{\beas}{\begin{eqnarray*}}
\newcommand{\eeas}{\end{eqnarray*}}
\newcommand{\be}{\begin{equation}}
\newcommand{\ee}{\end{equation}}
\newcommand{\bt}{\begin{theorem}}
\newcommand{\et}{\end{theorem}}
\newcommand{\bc}{\begin{center}}
\newcommand{\ec}{\end{center}}
\newcommand{\ben}{\begin{enumerate}}
\newcommand{\een}{\end{enumerate}}
\newcommand{\lan}{\langle}
\newcommand{\ran}{\rangle}
\newcommand{\ei}{\end{itemize}}

\newcommand{\nrm}[1]{\left\Vert #1 \right\Vert}

\newcommand{\ds}{\displaystyle}

\newcommand{\ve}{\varepsilon}

\def\dataspace{\widehat{{\cal I}}_0}

\def\newspace{{\cal I}_0}

\newcommand{\Ra}{\Rightarrow}
\newcommand{\ra}{\rightarrow}

\def\vconv{\stackrel{v}{\rightarrow}}
\def\wconv{\stackrel{w}{\rightarrow}}
\def\swconv{\Rightarrow}

\def\met{{\cal E}}
\def\eem{\mm(\met)}
\def\eemf{\mm_F(\met)}

\def\mm{{\mathbb M}}
\def\mmf{{\mathbb M}_F[0,\endsup)}
\def\mmsp{{\mathbb M}_{\leq 1} [0,\endsup)}

\def\mrangespace{{\cal S}'}
 \def\cinfspace{\C^\infty_c([0,\endsup) \times \R_+)}

\def\ddspace{{\mathbb D}_{\mrangespace}[0,\infty)}

\def\dspaceh{{\mathbb D}_{\cal H}[0,\infty)}
\def\dTspaceh{{\mathbb D}_{\cal H}[0,T]}

\def\dspacerp{{\mathbb D}_{\R_+}[0,\infty)}
\def\dspacer{{\mathbb D}_{\R}[0,\infty)}
\def\incspace{{\mathbb I}_{\R_+}[0,\infty)}

\def\ecb{{\cal C}_b(\met)}
\def\ecc{{\cal C}_c (\met)}
\def\eocd{{\cal C}^1(\met)}
\def\eocdc{{\cal C}^1_c(\met)}
\def\eocdb{{\cal C}^1_b(\met)}

\def\cb{{\cal C}_b(\R_+)}

\def\acb{{{\cal A}{\cal C}}_b({\cal E})}
\def\acbl{{{\cal A}{\cal C}}_b[0,\endsup)}

\def\newocdcpm{{\cal C}^{1,1}_c([0,\endsup)\times\R_+)}

\def\newcbm{{\cal C}_b([0,\endsup)\times\R)}

\def\tnewf{\tilde{\varphi}}

\def\newf{\varphi}
\def\newft{\varphi(\cdot,t)}

\def\newfs{\varphi(\cdot,s)}
\def\dxnewfs{\newf_x(\cdot,s)}
\def\dtnewfs{\newf_s(\cdot,s)}

\def\f1{{\bf 1}}
\def\zerof{{\bf 0}}
\def\measzero{{\bf \tilde{0}}}

\def\ubound{L}

\def\lc{\Lambda}

\def\fren{\overline{R}_E^{(N)}}
\def\ren{R_E^{(N)}}

\def\qn{Q^{(N)}}
\def\dn{D^{(N)}}
\def\kn{K^{(N)}}
\def\kninv{\theta^{(N)}}
\def\kinv{\theta}
\def\fkn{\overline{K}^{(N)}}

\def\en{E^{(N)}}

\def\agen{a^{(N)}}

\def\measn{\nu^{(N)}}

\def\fe{\overline{E}}
\def\fen{\overline{E}^{(N)}}
\def\fmeasn{\overline{\nu}^{(N)}}
\def\fmeas{\overline{\nu}}

\def\flam{\overline{\lambda}}
\def\fdn{\overline{D}^{(N)}}

\def\xn{X^{(N)}}

\def\fxn{\overline{X}^{(N)}}
\def\fx{\overline{X}}

\def\fdcomp{\overline{A}}

\def\ecompn{A_E^{(N)}}

\def\fecompn{\overline{A}_E^{(N)}}

\def\hmeasn{\widehat{\nu}^{(N)}}
\def\hxn{\widehat{X}^{(N)}}
\def\hmeas{\widehat{\nu}}

\def\fk{\overline{K}}
\def\hkn{\widehat{K}^{(N)}}

\def\gene{E}
\def\geninitx{x}
\def\initxn{x_0^{(N)}}
\def\finitxn{\overline{x}_0^{(N)}}
\def\finitx{\overline{x}_0}
\def\hinitxn{\widehat{x}_0^{(N)}}
\def\hinitx{\widehat{x}_0}
\def\genzeta{\meas}
\def\genk{K}
\def\genx{X}

\def\genflx{\overline{X}}
\def\genh{{\mathcal{H}}}

\def\fsmallk{\overline{\kappa}}

\def\hatxn{\widehat{X}^{(N)}}
\def\hatx{\widehat{X}}
\def\filtn{\{{\cal F}^{(N)}_t\}}

\def\haten{\widehat{E}^{(N)}}

\def\hatmeasn{\widehat{\nu}^{(N)}}
\def\hatkn{\widehat{K}^{(N)}}
\def\bv{C}
\def\hate{\widehat{E}}
\def\hath{\widehat{{\cal H}}}
\def\hathn{\widehat{{\cal H}}^{(N)}}
\def\calk{{\cal K}}
\def\hcalkn{\widehat{{\cal K}}^{(N)}}
\def\hcalk{\widehat{{\cal K}}}

\def\hatk{\widehat{K}}

\def\hatmn{\widehat{M}^{(N)}}
\def\cmn{{\cal M}^{(N)}}

\def\cm{{\cal M}}

\def\chatmn{\widehat{{\cal M}}^{(N)}}
\def\chatm{\widehat{{\cal M}}}
\def\covfnal{{\cal Q}}

\def\hatm{\widehat{M}}
\def\hatmeas{\widehat{\nu}}

\def\chn{{\cal H}^{(N)}}
\def\bare{Q}
\def\baren{Q^{(N)}}
\def\fbaren{\overline{Q}^{(N)}}

\def\fmartn{\overline{M}^{(N)}}
\def\martn{M^{(N)}}
\def\dcompn{A^{(N)}}
\def\fdcompn{\overline{A}^{(N)}}

\def\mn{M^{(N)}}

\def\refmeas{\zeta}
\def\trefmeas{\tilde{\refmeas}}

\def\refx{X}

\def\refe{E}

\def\suli{\sum\limits}

\title[Limits of Many-Server Queues]{SPDE Limits of Many-Server Queues}

\author{Haya Kaspi}
\address{Department of Industrial Engineering and Management \\
Technion, Haifa, Israel} \email{iehaya@techunix.technion.ac.il}

\author{Kavita Ramanan }
\address{Division of Applied Mathematics \\
Brown University \\
Providence, RI 02912, USA }
\email{Kavita\_Ramanan@brown.edu }
\thanks{Both authors were partially supported by the US-Israel Binational
Science Foundation under Grant BSF-2006379. The first author was
partially supported by the Milford Bohm Chair Grant. The second author was
partially supported by the National Science Foundation under Grants
CMMI-0728064 (CMMI-1059967) and CMMI-0928154 (CMMI-1052750).}


\begin{document}

\begin{abstract} 
A many-server queueing system is considered in which 
customers with independent and identically distributed service
times enter service in the order of arrival. 
The state of the system is represented by a process that describes 
the total number of customers in the system, as well as 
a measure-valued process that keeps track of the ages of 
customers in service,  leading to a Markovian description of the dynamics.  
Under suitable assumptions, a functional central limit theorem is 
established for the sequence of (centered and scaled) 
state processes as the number of servers goes 
to infinity. 
The limit process describing the total number 
in system is shown to be an It\^{o} diffusion with a constant diffusion 
coefficient that is insensitive to  the service distribution.  
The  limit of the sequence of (centered and scaled) age processes is shown 
to be  a Hilbert space valued diffusion that can also be characterized as 
 the unique solution of a stochastic partial differential equation 
that is coupled with the It\^{o} diffusion. 
Furthermore, the limit processes are  shown to be semimartingales and to 
possess a strong Markov property.  
\end{abstract}

\date{September 29, 2010}

\maketitle


\vspace{-0.5cm} 

\bigskip
\hrule
 \vspace{-0.25cm}
\tableofcontents
\vspace{-1.0cm}
\hrule

\beginsec
\section{Introduction}
\label{sec-intro}

\subsection{Background, Motivation and Results.}
\label{subs-back}

Many-server queues constitute a fundamental model in 
queueing theory and are typically harder to analyze 
than single-server queues. 
The main objective of this paper is to establish useful 
functional central limit theorems for many-server 
queues in the  asymptotic regime in which 
the number $N$ of servers tends to infinity and the mean arrival rate 
scales as $\lambda^{(N)} = \flam N - \beta \sqrt{N}$ for some $\flam > 0$ 
and $\beta \in (-\infty, \infty)$.  For many-server queues with Poisson
arrivals, this scaling was considered more than half a century ago by 
Erlang \cite{Erlang48} and thereafter by Jagerman \cite{Jag74} for a loss system with 
exponential service times, but it was not until the 
influential work of Halfin and Whitt \cite{HalWhi81} that 
   a general heavy traffic limit theorem was 
established for queues with  renewal arrivals, 
 exponential service times, normalized to have unit mean, 
and $\flam = 1$.  As a result, this asymptotic regime is often 
referred to as the Halfin-Whitt regime.    
In contrast to conventional heavy traffic scalings, 
in the Halfin-Whitt regime 
 the limiting stationary probability of a positive wait is non-trivial  
(i.e., it lies strictly between zero and one), 
which better models the behavior of 
many systems found in applications.  
 Halfin and Whitt \cite{HalWhi81} showed that 
the limit of the sequence of 
processes representing the (appropriately centered and scaled) 
number of customers in the system is a diffusion process that behaves like 
an Ornstein-Uhlenbeck process below zero and like 
a Brownian motion with drift above zero.  
When  $\beta > 0$, which ensures that 
each of the $N$-server queues is stable,  
 this characterization of the 
limit process was  used to establish approximations 
to the stationary probability of positive wait in 
a queue with $N$ servers. 
For exponential service distributions, the work 
of Halfin and Whitt was subsequently generalized by Mandelbaum, Massey and Reiman
\cite{manmasrei} to the network 
setting and the case of inhomogeneous Poisson arrivals.

However, in many applications,  
statistical evidence suggests that 
 it may be more appropriate to model the service times 
as being non-exponential (see, for example, the study 
of real call center data in Brown et al.\ \cite{brownetal} that 
suggests that the service times are lognormally distributed). 
A natural goal is then to understand the behavior of 
many-server queues in this scaling regime when the service distribution 
is not exponentially distributed.   
  Specifically, in addition to 
establishing a limit theorem, the aim is 
to obtain a tractable representation of the 
limit process that makes 
 it  amenable to computation, so that 
 the limit  could be used to shed insight 
into performance measures of interest for an $N$-server 
queue.  

In this work, we represent the state  of the $N$-server queue  
by a nonnegative, integer-valued process $\xn$ that
records the  total number of customers in system, as well as a
measure-valued process $\measn$ that keeps track of the ages of customers
in service.  This representation was first introduced by Kaspi and  
Ramanan in \cite{KasRam07}, where it was
 used to identify the functional strong law of 
large numbers limits or equivalently, fluid limits  
for these queues and was  
 subsequently shown  to provide 
 a Markovian description  of the dynamics 
(see  Kang and Ramanan \cite{KanRam08}). 
Under suitable assumptions, 
in each of the cases when the fluid limit is subcritical,
critical or supercritical (which, roughly speaking, corresponds to the cases 
$\flam < 1$, $\flam = 1$ and $\flam > 1$), we show  (in Theorems 
\ref{th-main} and \ref{th-fclt}) 
that the diffusion-scaled state sequence,  
$\{(\hatxn, \hmeasn)\}_{N \in  \N}$ 
 obtained by centering the state around the fluid limit 
and multiplying the centered state by $\sqrt{N}$, converges weakly 
to a limit  process $(\hatx, \hmeas)$.   
Moreover, the component  
$\hatx$ is characterized as  a real-valued \cad process that is the solution
to an It\^{o} diffusion  with a constant 
 diffusion coefficient that is {\em  insensitive to} the service distribution, 
and whose drift is an adapted process that is a functional of 
$\hmeas_t$ (see Corollary \ref{cor-ito}). 
As for the age process, although the $\hmeasn$ are 
(signed) Radon measure valued processes, the limit $\hmeas$ lies outside 
this space. A key challenge was to identify a suitable space in 
which to establish convergence without imposing restrictive 
assumptions on the service distribution $G$.  Under 
conditions that include a large class of service distributions 
relevant in applications such as phase-type, Weibull, lognormal, logistic  
and (for a large class of parameters) Erlang and Pareto distributions, 
we show that the convergence of $\hmeasn$ to $\hmeas$ holds in the space 
of $\H_{-2}$-valued c\`{a}dl\`{a}g processes, where $\H_{-2}$ is the 
dual of the Hilbert space $\H_2$.  
In particular, this immediately implies convergence of a large class
 of functionals of the many-server queue. 
In addition,  we  show that both processes 
are semimartingales with an explicit decomposition 
(see Theorem \ref{th-main1})  and 
we  characterize $\hmeas$ as the unique solution to 
a stochastic partial differential equation that is coupled with 
the It\^{o} diffusion $\hatx$ (see Theorem \ref{th-main2}(a)).  
Furthermore (in Theorem \ref{th-main2}(b)), we also  show that the pair, along 
with an appended state, forms a strong Markov process.  

\subsection{Relation to Prior Work.}

To date, the most general results on process level convergence in the 
Halfin-Whitt regime were obtained in a nice pair of papers 
by Reed \cite{reed07} and Puhalskii and Reed \cite{Puhree09}. 
Under the assumptions that $\flam = 1$, the residual service times of customers 
in service at time $0$ are independent and identically distributed 
(i.i.d.)\ and taken from the equilibrium 
fluid distribution, and the total (fluid scaled) number in system converges to $1$, 
a heavy traffic limit theorem for the sequence  of processes 
$\{\widehat{X}^{(N)}\}_{N \in \N}$ was 
established by Reed \cite{reed07}  with only a finite mean condition on the
service distribution.   
 This result was extended by Puhalskii and Reed \cite{Puhree09} to allow  
for more general, possibly inhomogeneous arrival 
processes and residual service times of customers in 
service at time zero  that, while still  i.i.d.\, could be chosen from 
an arbitrary distribution.  In this setting, 
  convergence of finite-dimensional distributions  was established in 
\cite{Puhree09}, 
and strengthened to process level convergence established when the 
service distribution is continuous.  
 The general approach used in both these 
papers is to represent the many-server queue as a perturbation 
of an infinite-server queue and to  establish tightness and 
convergence  using a continuous mapping representation and  
estimates analogous to those obtained by Krichagina and Puhalskii
\cite{KriPuh97} for the infinite-server queue under  
similar assumptions on the initial conditions. 
In both papers, the limit is characterized as the unique 
 solution to a certain implicit stochastic convolution equation.  
Several previous works had also extended 
the Halfin-Whitt process level result for specific classes of 
service distributions.  
Noteworthy amongst them is the paper by 
Puhalskii and Reiman \cite{PuhRei00}, which considered 
phase-type service distributions and characterized the  
heavy traffic limit theorem as  a multidimensional diffusion, 
where each dimension corresponds to a different phase of the service 
distribution.  
Whitt \cite{whi05} also established a process level result 
for a many-server queue with finite waiting room and 
a service distribution that is a mixture of an exponential 
random variable and a point mass at zero. 
Moreover, for service distributions 
with finite support, 
Mandelbaum and  Mom\c{c}ilovi\'{c} \cite{ManMom05} used a combination 
of combinatorial and probabilistic methods to study 
the limit of the virtual waiting time process. 
 In addition to the process level results described above, interesting 
results on the asymptotics of steady state distributions in the Halfin-Whitt regime 
have been obtained by Jelenkovic, Mandelbaum and 
Mom\c{c}ilovi\'{c} \cite{JelManMom04} 
for deterministic service times and by Gamarnik and Mom\c{c}ilovi\'{c} 
\cite{GamMom08} for service times that are 
lattice-valued with finite support.

  Our work serves 
to complement the above mentioned results, with the focus being on 
establishing tractability of the limit process under 
reasonably general 
assumptions on the service distribution that includes 
a large class of service distributions of interest. 
  Whereas in all the above 
papers only the number in system is considered,
 we establish convergence for a more general state process, which 
implies the convergence of a large class of functionals of the 
process and not just the number in system.    
In addition, our approach  leads to a new characterization 
for the limiting number in system $\hatx$ as 
an It\^{o} diffusion, which 
relies on an asymptotic independence 
result for the centered arrival and departure processes 
(see Proposition \ref{prop-martconv}) that 
may be of independent interest.   We also establish 
an insensitivity result showing that the diffusion coefficient depends only 
on the mean and variance of the interarrival times and is independent 
of the service distribution. 
As a special case, we can  recover the results of 
Halfin and Whitt \cite{HalWhi81} and  Puhalskii and Reiman 
\cite{PuhRei00}  and (for the smaller class of 
service distributions that we consider) 
Reed \cite{reed07}.  
Moreover, we allow in a sense more general initial conditions  
than those considered  in Reed \cite{reed07} and Puhalskii and Reed 
\cite{Puhree09}, both of which assume that the residual times 
of customers in service at the initial time are i.i.d. 
This property is not typically preserved at positive times.  
In contrast, we establish a consistency property (see Lemma 
\ref{lem-consistency}) that shows that the assumptions we impose at the initial time 
are also satisfied at any positive time and our assumptions are 
trivially satisfied by a system that starts with zero initial 
conditions (i.e., at the fluid initial condition). 
 As shown in Theorem 3.7 and Section 6 of Kaspi and Ramanan 
\cite{KasRam07},  starting 
from an empty system, the fluid limit does not reach 
the fluid equilibrium state in finite time. 
Therefore, the consideration of general initial conditions 
is useful for both capturing the 
transient behavior of the system as well as 
for establishing the (strong) Markov property for 
the limit process. 
The latter can  be potentially useful as this enables the 
application of a wide array of tools available for Markov 
processes in general state spaces. 

The Markovian representation 
of the state, though infinite-dimensional,
 leads to an intuitive characterization of the dynamics, 
which allows the framework  to be extended to incorporate more general features 
into the model (see, for example, the extension of this framework 
to include abandonments by Kang and Ramanan in \cite{KanRam08} and \cite{KanRam10a})). 
In the subcritical case our results  provide a characterization  
of the diffusion limit of the well studied 
 infinite-server queue, which 
is easier to analyze  due to the absence of a queue and, hence, 
of an interaction between those in service and those waiting in queue.  
A few representative works on diffusion limits of the number in system 
in the infinite-server queue include 
Iglehart \cite{Igl65}, Borovkov \cite{Bor67}, Whitt \cite{Whi82} and 
Glynn and Whitt \cite{GlyWhi91}, where the limit process 
 is characterized as an Ornstein-Uhlenbeck process, and 
Krichagina and Puhalskii   
\cite{KriPuh97}, who provided an alternative  representation 
of the limit in terms of the so-called Kiefer process. 
More recently,  a functional central limit theorem in the 
space of distribution-valued processes was established 
for the $M/G/\infty$ queue by 
Decreusefond and Moyal \cite{DecMoy08}. 
In contrast to the infinite-dimensional 
Markovian representation in terms of 
residual service times used in  Decreusefond and 
Moyal \cite{DecMoy08}, 
the Markovian representation in terms of the age process that we use  
allows us to associate some natural martingales that 
facilitate the analysis.  This perspective may be 
useful in the analysis of other queueing networks as well  
and has, for example, been recently adopted by
 Reed and Talreja \cite{ReedTal09} 
 in their extension of the work 
of Decreusefond and Moyal \cite{DecMoy08} 
 to establish infinite-dimensional functional central limit theorems 
for the $GI/G/\infty$ queue.  The work  \cite{ReedTal09} adopts a semi-group 
approach that seems to require much stronger assumptions on the 
service distribution (namely that the hazard rate function $h$ 
of the service distribution is infinitely differentiable and $h$ 
and its derivatives are all uniformly bounded)  than is imposed in our paper.

\subsection{Outline of the Paper} 
\label{subs-outline}

Section \ref{sec-mode} contains a precise mathematical description of the
model and the state descriptor used, as well as  the defining dynamical
equations. A deterministic analog of the model,  described by dynamical equations that 
are referred to as the fluid equations, is introduced in Section \ref{sec-res}. 
Section \ref{sec-res} also recapitulates the  result of Kaspi and Ramanan 
 \cite{KasRam07} that shows that (under fairly general conditions stated as 
Assumptions \ref{as-flinit} and \ref{as-h}) the functional strong law of 
large numbers limit of the normalized (divided by $N$) state 
of the $N$-server system is the unique solution to the fluid equations. 
In Section \ref{subs-mmeas} a sequence of martingales obtained as 
compensated departure processes, which play an important role in the analysis, 
is introduced and the associated scaled martingale measures $\chatmn$, $N \in \N$, 
are shown to be orthogonal, which allows one to define
 certain associated stochastic convolution integrals $\hathn$.  
The main results and their corollaries are stated in Section
\ref{sec-mainres}, and their proofs are presented in  
Section \ref{sec-proofs}.  The proofs rely on results  
obtained in Sections \ref{subs-prelim}, \ref{subs-cont} and
\ref{sec-martconv}. 
Section \ref{subs-prelim} contains a 
succinct characterization of the dynamics and establishes a representation 
(see Proposition \ref{cor-sae1})
for $\hmeasn$, the diffusion-scaled age process  in the $N$-server 
system, in terms of certain stochastic convolution integrals $\hathn$, 
$\hcalkn$ and the initial data. 
In Section \ref{subs-cont}, it is shown that the processes $\hcalkn$, $\hxn$ and 
$\hmeasn$ can be obtained as a continuous mapping of 
the initial data sequence and the process $\hathn$.    
 Section \ref{sec-martconv} is devoted to establishing  convergence 
of the martingale measure sequence $\{\chatmn\}_{N \in \N}$ 
and the associated sequence $\{\hathn\}_{N \in \N}$ of stochastic convolution
integrals, jointly with the sequence of centered arrival processes and 
initial conditions (see Corollary
\ref{cor-hreg}). In particular, the asymptotic independence 
property is established. 
Section \ref{sec-martconv} is the most technically demanding part of the paper. 
To maintain the flow of the exposition, some supporting results 
are relegated to the Appendix.   Appendix \ref{sec-consistency} also contains 
the proof of a consistency result, which shows that the assumptions 
on the initial conditions are reasonable.  
 First, in Section \ref{subs-not} we  introduce some common 
notation and terminology used in the paper.

\subsection{Notation and Terminology}   
\label{subs-not}

The following notation will be used throughout the paper.
 $\Z_+$ is the set of non-negative integers, $\N$ is the set of 
natural numbers or, equivalently,  strictly positive integers,
$\R$ is the set of real numbers and
$\R_+$ the set of non-negative real numbers.
For $a, b \in \R$, $a \vee b$ and $a \wedge b$ denote, respectively, 
  the maximum and minimum of $a$ and $b$,
 and the short-hand notation $a^+$ will also be used for $a \vee 0$.
Given $B \subset \R$, 
 $\ind_B$ denotes the indicator function of the set $B$
(that is, $\ind_B (x) = 1$ if $x \in B$ and $\ind_B(x) = 0$ otherwise).

\subsubsection{Function Spaces}
\label{subsub-fun}

Given any metric space $\met$, we denote by ${\cal B}(\met)$ the Borel sets
of $\met$ (with topology compatible with  the metric on $\met$), and  
let $\ecb$, $\acb$ and $\ecc$, respectively, denote 
 the space of bounded continuous functions, bounded absolutely continuous 
functions and 
 the space of continuous  functions with compact support defined on
$\met$ and taking values in the reals.  We also let 
 $\eocd$ and ${\cal C}^{\infty}(\met)$, respectively, represent the space of real-valued,
once  continuously differentiable and infinitely differentiable 
functions on $\met$,   $\eocdc$
the subspace of functions in $\eocd$ that have compact support and  
$\eocdb$  the subspace of functions in $\eocd$ that, together with its 
first derivatives, are bounded. 
 We let
 $\D_{{\cal E}}[0,\infty)$ denote the space of ${\cal E}$-valued 
c\`{a}dl\`{a}g functions defined on $[0,\infty)$ and let  $\supp(\newf)$ 
 denote the support of a function $\newf$.

We will mostly be interested in the case when
 $\met = [0,\endsup)$ and  $\met = [0,\endsup) \times \R_+$, for some $\endsup \in (0,\infty]$.
To distinguish these cases, we will usually use $f$ to denote generic  functions
on $[0,\endsup)$  and $\newf$ to denote generic
 functions on $[0,\endsup) \times \R_+$.  By some abuse of notation,
given $f$ on $[0,\endsup)$, we will sometimes
also treat it as a function on $[0,\endsup) \times \R_+$ 
that is constant in the second variable.  
Recall that given $T < \infty$ and 
a continuous function $f \in {\cal C}[0,T]$, the modulus 
of continuity $w_f(\cdot)$ of $f$ is defined by 
\be
\label{def-modcon}
 w_f(\delta) \doteq \sup_{s,t \in [0,T]:|t-s| < \delta}
|f(t)-f(s)|, \qquad \delta > 0. 
\ee
When $\met = [0,\endsup) \times \R_+$, 
for some $\endsup \leq \infty$, we let 
${\cal C}^{1,1}([0,\endsup) \times \R_+)$ 
denote the 
space of absolutely continuous functions $\newf$
 on $[0,\endsup) \times \R_+$ for which the directional 
derivative  $\newf_x + \newf_s$ in the $(1,1)$ direction exists 
and is continuous and let 
${\cal C}^{1,1}_c([0,\endsup) \times \R_+)$ (respectively, 
${\cal C}^{1,1}_b([0,\endsup) \times \R_+)$) 
denote the subset of functions $\varphi$ in 
${\cal C}^{1,1}([0,\endsup) \times \R_+)$ such that $\varphi$, along with 
its directional derivative $\varphi_x + \varphi_s$, 
has  compact support (respectively, is bounded).  We let  
$\incspace$ denote the space 
of non-decreasing functions $f \in \D_{\R}[0,\infty)$ with $f(0) = 0$. 
For $\endsup \in [0,\infty]$, 
 $\L^\alpha[0,\endsup)$, $\alpha \geq 1$, and $\L^\infty[0,\endsup)$ 
represent, respectively, the spaces of measurable functions $f$ 
such that $\int_{[0,\endsup)} |f|^\alpha < \infty$    
and the space of essentially bounded functions on $[0,\endsup)$. 
Also,   
$L^{i}_{loc}[0,\endsup)$, $i = 1, 2, \infty$, represents the corresponding 
space in which the associated property holds only locally, that is, on every 
compact subset of $[0,\endsup)$. 
The constant functions 
$f \equiv 1$ and $f \equiv 0$ on $[0,\endsup)$ 
will be represented by the symbols
$\f1$ and $\zerof$, respectively. Given any c\`{a}dl\`{a}g,
real-valued function $f$ defined on $E$, we define $\nrm{f
}_T \doteq \sup_{s \in [0,T]} |f (s)|$ for every $T < \infty$,
and let $\nrm{f}_\infty \doteq \sup_{s \in [0,\infty)}
|f(s)|$, which could possibly take the value $\infty$.  
Also, for $f \in \D_{\R}[0,\infty)$, we use  
$\Delta f (t) = f(t) - f(t-)$ to denote the jump of 
$f$ at $t$.

For any  $f\in {\cal C}^\infty[0,\endsup)$,  
let $f^{(n)}$ denote the $n$th derivative of $f$.  
 Also, let $\nrm{f}_{\H_0}$ be the usual $\L^2$-norm: 
$\nrm{f}_{\H_0}^2 \doteq \nrm{f}_{\L^2}^2 \doteq \left( \int_0^\endsup f^2(x) \, dx\right)$, 
and set 
\[ \nrm{f}_{\H_n}^2 \doteq \nrm{f}_{\H_0}^2 + \sum_{i=1}^n
\nrm{f^{(i)}}_{\H_0}^2. 
\] 
For $n = 1, 2$, and $f$ for which the corresponding first or 
second (weak) derivatives are well defined, we 
will sometimes also use the notation $f^\prime  = f^{(1)}$ and 
$f^{\prime \prime} = f^{(2)}$.  
Note that if $f \in \L^2[0,\infty)$, then there exists a real-valued 
sequence $\{x_n\}$ with $x_n \ra \infty$ and $f(x_n) \ra 0$ 
as $n \ra \infty$.  Moreover, $f^2(x_n) - f^2(0) = 2 \int_0^{x_n} f(u) f^\prime (u)
\, du$.  Applying the  Cauchy-Schwarz inequality and taking 
limits as $n \ra \infty$, this implies 
$|f(0)|^2 \leq 2 \nrm{f}_{\H_0} \nrm{f^\prime}_{\H_0} \leq 2 \nrm{f}_{\H_1}^2$.   
When combined with the relation  
 $f^2(x) = f^2(0) + 2\int_0^x f(u) f^\prime(u) \, du$
and  another application of 
the Cauchy-Schwarz inequality, this yields the norm inequalities 
\be
\label{norm-ineq}
 |f(0)| \leq  \sqrt{2} \nrm{f}_{\H_1}, \qquad 
\nrm{f}_{\infty} \leq 2 \nrm{f}_{\H_1}, 
\ee
which will be used in the sequel.

For a fixed $[0,\endsup)$, we
define $\testspace = \testspace[0,\endsup)$ (respectively, $\testspace_c = \testspace_c[0,\endsup)$) 
to be the vector space of ${\cal C}^{\infty}$ functions 
(respectively, ${\cal C}^{\infty}$ functions with compact support)  
on $[0,\endsup)$),  
 equipped with 
the sequence of norms $\nrm{\cdot}_{\H_n}$, $n = 0, 1, 2, \ldots$, 
and let $\H_n = \H_n[0,\endsup)$ be the completion of $\testspace$ relative to the norm 
$\nrm{\cdot}_n$.  
Moreover, let $\dualspace$ be the dual of $\testspace$ 
(i.e., the space of continuous linear functionals on $\testspace$) 
equipped with the strong topology and likewise, let 
$\dualspace_c$ be the dual of $\testspace_c$.   For $n \in \N$, 
let $\H_{-n} = \H_{-n}[0,\endsup)$ be the dual of $\H_n[0,\endsup)$, 
with the dual norm $\nrm{\cdot}_{-n}$ defined by 
\[ \nrm{f}_{-n}^2  = \sum_{k=1}^\infty f(e_{nk})^2, \qquad f \in
\H_{-n}, 
\]
where $\{e_{nk}, k = 1, \ldots,\}$ is a complete orthonormal 
system in $(\testspace, \nrm{\cdot}_n)$.  
Each $\H_n$ is a Sobolev space and also a 
Hilbert space and it follows from Maurin's theorem (see, for example, 
Theorem 6.53 of \cite{adamsbook}) that 
$\nrm{\cdot}_{\H_1}\overset{\mbox{\tiny{HS}}}{<} \nrm{\cdot}_{\H_2}$ 
 and it follows from Lemma 5 and Assertion 11 of \cite{aga89} 
that $\testspace$ is a separable,  Fr\'{e}chet nuclear space and 
consequently (see Corollary 2 and
Assertion 11 of \cite{aga89}),   its dual  $\dualspace$ is also a separable
Fr\'{e}chet nuclear space.
For $\meas \in \dualspace$ and $f \in \testspace$ and likewise, 
for $\meas \in \H_{-n}$ and $f \in \H_n$,  we let 
$\meas(f)$ denote the duality pairing.

\subsubsection{Measure Spaces}
\label{subsub-meas}

 The space of  Radon measures on a metric
space $\met$, endowed with the Borel $\sigma$-algebra,
 is denoted by  $\eem$,  
  $\eemf$  is the subspace of finite 
measures in $\eem$ and $\mm_{\leq 1}(\met)$ is the subspace  
of sub-probability  measures (i.e., positive measures 
with total mass less than on equal to $1$) on $\met$. 
For any 
Borel measurable function $f: \met \ra \R$ that is integrable
with respect to $\xi \in  \mm(\met)$, we  often use the short-hand notation
$\lan f, \xi \ran \doteq \int_{\met} f(x) \, \xi(dx).$ 
Recall that a Radon measure on $\met$ is one that assigns a finite measure to every relatively
compact subset of $\met$.
By identifying a Radon measure $\mu \in \eem$ with the mapping on $\ecc$ defined by
$f \mapsto \lan f, \mu\ran,$ 
one can equivalently define a Radon measure on $\met$ as a linear mapping
from $\ecc$ into $\R$ such that for every compact set ${\cal K} \subset \met$, there exists
$L_{{\cal K}} < \infty$ such that
\[ \left| \lan f, \mu\ran \right| \leq L_{{\cal K}} \nrm{f}_\infty \quad \quad \forall
f \in \ecc  \mbox{ with } \supp (f) \subset {\cal K}. \] 
We will equip   $\eemf$ with the weak topology, i.e., a sequence 
$\{\mu_n\}_{n \in \N}$ in $\eemf$ is said to converge to $\mu$ in the weak
 topology (denoted $\mu_n \wconv \mu$) 
if and only if for every $f \in \ecb$,
$\lan f, \mu_n\ran \ra \lan f, \mu \ran$ as $n \ra \infty$. 
The symbol $\delta_x$ will be used to denote the measure with unit
mass at the point $x$ and we will use
$\measzero$ to denote the identically zero Radon measure. 
When $\met$ is an interval, say $[0,\endsup)$, for notational conciseness,
we will often write $\mm[0,\endsup)$ instead of $\mm([0,\endsup))$.
Also, for ease of notation, given $\xi \in \mm[0,\endsup)$ and an interval $(a,b) \subset [0,\endsup)$, we will
use $\xi(a,b)$ and $\xi(a)$ to denote $\xi((a,b))$ and $\xi(\{a\})$, respectively.

\subsubsection{Stochastic Processes}
\label{subs-notsp}

Given a Polish space ${\cal H}$, we denote by $\dTspaceh$
(respectively, $\dspaceh$) the space of ${\cal H}$-valued,
c\`{a}dl\`{a}g functions on $[0,T]$ (respectively, $[0,\infty)$),
 endowed  with the usual Skorokhod $J_1$-topology (see 
\cite{bilbook} for details on this topology). 
Then $\dTspaceh$ and $\dspaceh$ are also Polish
spaces. In this work, we will be interested
in ${\cal H}$-valued stochastic processes, especially the 
cases when ${\cal H} = \R$, 
${\cal H} = \mmf$ for some $\endsup \le\infty$, 
 ${\cal H} = \dualspace[0,\endsup)$ and ${\cal H} = \H_{-n}[0,\endsup)$ for  
$n = 1, 2$, and products of these spaces. 
These are random elements that are defined on a probability space 
$(\Omega, {\cal F}, \P)$ and  take values in $\dspaceh$, equipped 
with the Borel $\sigma$-algebra (generated by open sets under the
Skorokhod $J_1$-topology). A sequence $\{Z^{(N)}\}_{N \in \N}$ of c\`{a}dl\`{a}g,
${\cal H}$-valued  processes, with $Z^{(N)}$ defined on the
probability space $(\Omega^{(N)}, {\cal F}^{(N)}, \P^{(N)})$,
 is said to converge in distribution
to a c\`{a}dl\`{a}g ${\cal H}$-valued process $Z$ defined on  $(\Omega, {\cal
  F}, \P)$ if and only if for every bounded, continuous functional
$F:\dspaceh \ra \R$, 
\[ \lim_{n \ra \infty}  \E^{(N)}\left[ F(Z^{(N)}) \right] = \E \left[ F(Z)\right],
\]
where $\E^{(N)}$ and $\E$ are the expectation operators with respect
to the probability measures $\P^{(N)}$ and $\P$, respectively.
Convergence in distribution of $Z^{(N)}$ to $Z$ will be denoted by
$Z^{(N)} \swconv Z$.

\beginsec

\section{Description of the Model}
\label{sec-mode}

In Section \ref{subs-modyn} we describe the many-server model under 
consideration.  In Section \ref{subs-aux} we introduce the state descriptor
and the dynamical equations that describe the evolution of the state.

\subsection{The $N$-server model}
\label{subs-modyn}

Consider a system with $N$ servers, where arriving
customers are served in a non-idling, First-Come-First-Serve (FCFS) manner,
i.e., a newly arriving customer immediately enters service if there are any
idle servers or, if all servers are busy, then the customer joins the back of
the queue and the customer at the head of the queue (if one is present)
enters  service as soon as a server becomes free.
Our results are not sensitive to the exact mechanism used to assign an
arriving customer to an idle server as long as the non-idling condition
is satisfied.
Customers are assumed to be infinitely patient, i.e., they wait in queue till they
receive service.  Servers are non-preemptive and serve a customer 
to completion before starting service of a new customer. 
 Let $\en$ denote the
cumulative arrival process, with $\en (t)$ representing the total number
of customers that arrive into the system in the time
interval $[0,t]$, and let the service requirements be given by
the i.i.d.\ sequence $\{v_i, i = -N+1, -N+2, \ldots, 0, 1, \ldots\}$, with
common cumulative distribution function $G$.
Let $\xn(0)$ represent the number of customers in the system at time $0$.
Due to the non-idling condition, the number of customers in service at time
$0$ is then $\xn(0) \wedge N$.
The sequence $\{v_i, i = -\xn(0) \wedge N + 1, \ldots, 0\}$ represents
the service requirements of customers already in service at time zero,
ordered according to the amount of time they have spent in service at time
zero, whereas for $i \in \N$, $v_i$ represents the service requirement of the
$i$th customer to enter service after time $0$.  

Consider the c\`{a}dl\`{a}g process $\ren$   defined by 
\be \label{def-ren}
 \ren (s) \doteq
\inf \left\{u > s:E^{(N)}(u) > E^{(N)}(s)\right\}-s, \qquad s \in [0,\infty). 
\ee 
Note that $\ren(s)$ represents the time to 
the next arrival.  
The following mild assumptions will be imposed throughout, without explicit mention.
\begin{itemize}
\item $\en$ is a \cad non-decreasing pure jump process with $\en(0) = 0$ and almost
  surely, for
$t \in [0,\infty)$, $\en(t) < \infty$ and $\en(t) - \en(t-) \in \{0,1\}$;
\item
The process $\ren$ is Markovian with respect to the augmentation of 
its own natural filtration;  
\item
The cumulative arrival process is independent of
the i.i.d.\ sequence of service requirements
$\{v_j, j = -N+1, \ldots, \}$ and, given $\ren (0)$,
$(\en (t), t > 0)$ is independent of $\xn (0)$ and
the ages of the customers in service at time zero, where 
the age of a customer is defined to be the amount of time elapsed since 
the customer entered  service; 
 \item
 $G$ has density $g$;
\item
Without loss of generality, we can (and will) assume that the mean service requirement
is $1$:
\be
\label{def-mean1}
\int_{[0,\infty)} \left( 1 - G(x) \right) \, dx = \int_{[0,\infty)} x g(x) \, dx = 1.
\ee
Also,  the right-end of the support of the service distribution  is denoted
by
\[  \endsup \doteq \sup \{x \in [0,\infty): G(x) < 1 \}. \]
\end{itemize}
 Note that the existence of a density for $G$ 
implies, in particular, that $G(0+) = 0$.

\begin{remark}
\label{rem-cumarr}
{\em 
The assumptions above are fairly general,  allowing for a large class of
arrival processes and service distributions, and 
this model is sometimes referred to as the G/GI/N queueing model. 
When $\en$ is a renewal process, $\ren$ is simply the
forward recurrence time process,
the second assumption holds (see Proposition V.1.5 of Asmussen \cite{asmbook})  and the model corresponds to a GI/GI/N queueing system.  
However, the second assumption  holds more generally  
such as, for example, when $\en$ is an inhomogeneous Poisson process 
(see, for example, Lemma II.2.2 of Asmussen \cite{asmbook}). 
}
\end{remark}

The sequence of processes
$\{\ren, \en, \xn(0), v_i, i = -N +1, \ldots, 0, 1, \ldots\}_{N \in \N}$
are all assumed to be defined on a common probability
space $(\Omega, {\cal F}, \P)$ that is large enough for the
independence assumptions stated above to hold.

\subsection{State Descriptor and Dynamical Equations}
\label{subs-aux}

As in the study of the functional strong law of large numbers limit
for this model, which was carried out in Kaspi and Ramanan \cite{KasRam07}, we will represent
the state of the system by the vector of processes $(\ren, \xn, \measn)$,
where $\ren$ determines the cumulative arrival process via (\ref{def-ren}),
$\xn(t) \in \Z_+$ represents the total number of customers in system
(including those in service and those waiting in queue) at time
$t$ and $\measn_t$ is a discrete, non-negative finite measure on
$[0,\endsup)$ that has a unit mass at the age of
each customer in service at time $t$.  Here, the age $\agen_j$ of the $j$th
customer is (for each realization) a piecewise linear function
that is zero till the customer enters service, then increases
linearly while in service (representing the time elapsed since service
began) and then remains constant (equal to its service requirement) after
the customer completes service and departs the system.
In order to describe the state dynamics, we will find it
convenient to introduce the following auxiliary processes:

\begin{itemize}
\item
the cumulative departure process $\dn$, where $\dn(t)$ is the cumulative
number of customers that have departed the system in the interval $[0,t]$;
\item
the process $\kn$, where $\kn(t)$ represents the cumulative number of customers that
have entered service in the interval $[0,t]$.
\end{itemize}

\noi

A simple mass balance on the whole system shows that
\be
\label{def-dn}
\dn  \doteq \xn (0) - \xn + \en
\ee
Likewise, recalling that $\lan \f1, \measn \ran = \measn[0,\endsup)$ represents
the total number of customers in service, an analogous mass balance on
the number in service yields the relation
\be
\label{def-kn}
\kn \doteq \lan \f1, \measn \ran - \lan \f1, \measn_0 \ran + \dn.
\ee
For $j \in \N$, let
\[ \kninv_j \doteq \inf \{s \geq 0: \kn (s) \geq j\}, \]
with the usual convention that the infimum of an empty set is infinity, 
and note that $\kninv_j$
denotes the time of entry into service of the $j$th customer
to enter service after time $0$.
In addition,
for $j = - \xn (0)  \wedge N +1, \ldots, 0$,
set $\kninv_j = - \agen_j(0)$ to be the amount of time that the
$j$th customer in service at time $0$ has already been in service.
 Then, for $t \in [0,\infty)$ and
$j = -\xn(0) \wedge N+1, \ldots, 0, 1, \ldots$,
the age process is given explicitly by
\be
\label{def-agejn}
  \agen_j (t) = \left\{
\ba{ll}
\left[ t - \kninv_j \right] \vee 0 & \mbox{ if } t - \kninv_j < v_j, \\
v_j  & \mbox{ otherwise. }
\ea
\right.
\ee
Due to the FCFS nature of the service,   $\kn(t)$ is also
the highest index of any customer that has entered service
and  (\ref{def-agejn}) implies that for $j > \kn(t)$, $\kninv_j > t$ and
$\agen_j (t) = 0$.
The measure $\measn$ can then be expressed as 
\be
\label{def-nun}
\measn_t  = \sum\limits_{j = -\lan \f1, \measn_0 \ran   + 1}^{\kn(t)}
\delta_{\agen_j(t)} \ind_{\{\agen_j (t) < v_j\}},
\ee
where $\delta_x$ represents the Dirac mass at the point $x$.
The non-idling condition, which stipulates that there be no idle servers
when there are more than $N$ customers in the system, is
expressed via the relation
\be
\label{def-nonidling}
N - \lan \f1, \measn\ran = [N - \xn]^+.
\ee
For future purposes note that (\ref{def-dn}), (\ref{def-kn}) and
(\ref{def-nonidling}), together with the elementary identity
$x - x\vee 0 = x \wedge 0$, imply the relation
\be
\label{eq2-kn}
 \kn  = \xn \wedge N - \xn(0) \wedge N + \dn.
\ee

Note that $\lan \f1, \measn\ran \leq N$  
because the maximum number of customers in service at any given 
time is bounded by the number of servers.
  In addition, if the support of
$\measn_0$ lies in $[0,\endsup)$ then  it follows from
(\ref{def-agejn}) and (\ref{def-nun}) that $\measn_t$ takes values in 
$\mm_F[0,\endsup)$ for every $t \in [0,\infty)$. 
Thus,  the state of the system is represented by the
 c\`{a}dl\`{a}g process  $(\ren, \xn, \measn)$, which takes values in
$\R_+^2 \times \mm_F[0,\endsup)$.
For an explicit construction of the state that also shows that
the state and auxiliary processes are c\`{a}dl\`{a}g, see Lemma A.1 of
Kang and Ramanan \cite{KanRam08}.  
The results obtained in this paper are independent of the  particular rule used to
assign customers to stations, but for technical purposes 
we will  find it convenient to also introduce the additional ``station process''
$\sigma^{(N)} \doteq (\sigma_j^{(N)}, j \in  \{-N + 1,  \ldots, 0\} \cup \N)$.
For each $t \in [0,\infty)$, if customer $j$ has already entered service by time $t$,  then
 $\sigma_j^{(N)} (t)$ is equal to the index $i \in \{1, \ldots, N\}$ of the station  at which
customer $j$ receives/received service and  $\sigma_j^{(N)} (t) \doteq 0$
otherwise.
Finally, for $t \in [0,\infty)$,  
let $\overline{{\cal F}}_t^{(N)}$ be the $\sigma$-algbera generated by
$\{\ren(s),  a_j^{(N)}(s), \sigma_j^{(N)} (s), j \in \{-N, \ldots, 0\} \cup
\N, s \in [0,t]\}$, and let
$\{{\cal F}_t^{(N)},t \geq 0\}$ denote the associated  right continuous
filtration that is completed (with respect to $\P$) so that it
satisfies the usual conditions.
Then it is easy to verify that $(\ren,\xn,\measn)$ is 
$\{{\cal F}_t^{(N)}\}$-adapted (see, for example, Section 2.2 of Kaspi and Ramanan
\cite{KasRam07}). 
In fact, as shown in Lemma B.1 of Kang and 
Ramanan \cite{KanRam08}, 
 $\{(\ren (t), \xn (t), \measn_t)$, ${\cal F}_t^{(N)}, t \geq 0\}$ is a
strong Markov process.

\begin{remark}
\label{rem-indep}
{\em
The assumed Markov property of $\ren$ with respect to (the completed, right continuous
version of) its natural filtration 
and the independence properties of  $\en$ 
assumed in Section \ref{subs-modyn} together imply that 
for any $t \in [0,\infty)$, given $\ren(t)$ 
the future arrivals process $\{ \en(s), s > t\}$ is independent
of ${\cal F}_t^{(N)}$. 
}
\end{remark}

\beginsec

\section{Fluid Limit}
\label{sec-res}

We now recall the functional strong law of
large numbers limit or, equivalently, fluid limit obtained in \cite{KasRam07}. 
The initial data describing the system consists
of $\en$, the cumulative arrivals after zero,
$\xn(0)$, the number in system at time zero,  and $\measn_0$,
the age distribution of customers in
service at time zero.  The initial data belongs to the following space:
\be
\label{def-newspace}
\newspace \doteq \left\{(f,x,\mu) \in \incspace \times \R_+ \times \mmsp:
1 - \lan \f1, \mu \ran = [1-x]^+ \right\},
\ee
where $\incspace$  is the subset of non-decreasing functions
$f \in \dspacerp$ with $f(0) = 0$.  Assume that
$\newspace$ is equipped with the product topology.
Consider the ``fluid scaled''
versions of the  processes $H = E, X, K, D$ and measures
$H = \nu$ defined by
\be
\label{fl-scaling}
 \overline{H}^{(N)}  \doteq \dfrac{H^{(N)}}{N},
\ee
and let
\[  \fren(t) \doteq \ren (\en(t)),  \qquad t \in [0,\infty),  \]
for $N \in \N$.
The fluid results in \cite{KasRam07} were obtained 
under Assumptions \ref{as-flinit} and \ref{as-h} below.

\begin{ass}
\label{as-flinit}
There exists $(\fe, \finitx, \fmeas_0) \in \newspace$ such that, as $N \ra
\infty$,  $\E[\fxn(0)] \ra \E[\fx(0)]$, $\E[\fen(t)] \ra \E[\fe(t)]$ and,
almost surely,
\[
(\fen, \finitxn, \fmeasn_0) \ra (\fe, \finitx, \fmeas_0) \qquad \mbox{ in }
\newspace.
\]
\end{ass}

Next,
recall that $G$ has density $g$, and let $h$ denote its hazard rate: 
\be
\label{def-h}
h(x) \doteq \dfrac{g(x)}{1 - G(x)},  \quad \quad  x \in
[0,\endsup).
\ee
Observe that $h$ is automatically locally integrable on $[0,\endsup)$ because
for every $0 \leq a \leq b < \endsup$,
\be
\label{h-lint}
 \int_a^b h(x) \, dx = \ln (1 - G(a)) -\ln (1 - G(b)) < \infty.
\ee
However, $h$ is not integrable on $[0,\endsup)$.  In particular,
when $\endsup < \infty$, $h$ is unbounded on
$(\altendsup, \endsup)$ for every $\altendsup < \endsup$.

\begin{ass}
\label{as-h} At least one of the following two properties holds:
\begin{enumerate}
\item[(a)]
There exists  $\altendsup < \infty$ such that
$h$ is bounded on $(\altendsup, \infty)$;
\item[(b)]
There exists $\altendsup < \endsup$ such that
$h$  is lower-semicontinuous  on $(\altendsup, \endsup)$.
\end{enumerate}
\end{ass}

Note that Assumption \ref{as-h}(a) automatically implies that $\endsup =
\infty$. 
In Proposition \ref{prop-prelimit} (see also 
 Theorem 5.1 of \cite{KasRam07}) 
we provide a succinct description 
of the dynamics of the $N$-server system in terms of 
certain integral equations.  Here, we first introduce 
the deterministic analog of these equations, which 
we refer to as the fluid equations.

\begin{defn} {\bf (Fluid Equations)}
\label{def-fleqns} {\em The c\`{a}dl\`{a}g function $(\fx,
\fmeas)$ defined on $[0,\infty)$ and taking values in $\R_+ \times
\mmsp$ is said to solve the}  fluid equations {\em associated with
$(\fe, \overline{x}_0, \fmeas_0) \in \newspace$  if and only if }
$\fx(0) = \overline{x}_0$ {\em and for every} $t \in [0,\infty)$,
\be
\label{cond-radon}
 \int_0^t  \lan h, \fmeas_s \ran \, ds  < \infty 
\ee
{\em and  the following relations are satisfied: for every $\newf \in
  \newocdcpm$,}
\begin{eqnarray}
\label{eq-ftmeas}
\ds \lan \newft, \fmeas_t \ran  & = & \ds \lan  \newf(\cdot,0), \fmeas_0 \ran +
\int_0^t \lan \dtnewfs + \dxnewfs, \fmeas_s \ran \, ds  \\
 & & \nonumber
\quad  - \ds \int_0^t \lan  h(\cdot) \newfs,  \fmeas_s \ran \, ds+ \int_{[0,t]} \newf (0,s) \, d\fk (s),  \\
\label{eq-fx}
\fx (t) &  = & \fx (0) + \fe(t) - \ds \int_0^t \lan  h, \fmeas_s \ran \, ds
\end{eqnarray}
{\em and}
\be
\label{eq-fnonidling}
 1 - \lan \f1, \fmeas_t \ran = [1 - \fx(t)]^+,
\ee
{\em where }
\be
\label{eq-fk}
\fk(t) = \lan \f1, \fmeas_t \ran - \lan \f1, \fmeas_0 \ran + \int_0^t \lan h, \fmeas_s \ran \, ds.
\ee
\end{defn}

We now recall the result established in Kaspi and Ramanan 
\cite{KasRam07} (see Theorems 3.4 and 3.5 therein), which shows that
under Assumptions \ref{as-flinit} and \ref{as-h}, the fluid equations
uniquely characterize the functional strong law of large numbers or
 mean-field limit of the $N$-server system, in the asymptotic regime where
 the number of servers and arrival rates both tend to infinity.

\begin{theorem}[Kaspi and Ramanan \cite{KasRam07}]
\label{th-flimit}
Suppose Assumptions \ref{as-flinit} and  \ref{as-h}  are satisfied and
 $(\fe, \fx(0), \fmeas_0) \in \newspace$ is the limit of the initial 
data  as stated in Assumption \ref{as-flinit}.   
Then there exists a unique solution  $(\fx,\fmeas)$  to
the associated fluid equations (\ref{cond-radon})--(\ref{eq-fnonidling}) 
and, as $N \ra \infty$, 
 $(\fxn, \fmeasn)$ converges almost surely  to $(\fx,\fmeas)$. 
Moreover, $(\fx,\fmeas)$  satisfies the non-idling condition
(\ref{eq-fnonidling})  
and, for every $f \in \cb$,
\begin{eqnarray}
\label{final-fleqs}
\ds \int_{[0,M)} f(x) \,  \fmeas_t (dx)   & = &  \ds \int_{[0,M)} f(x+t)
\dfrac{1 - G(x+t)}{1 - G(x)} \fmeas_0 (dx) \\
\nonumber
& & \qquad + \int_{[0,t]} f(t-s) (1 - G(t-s)) d \fk (s),
\end{eqnarray}
where $\fk$
satisfies the relation (\ref{eq-fk}).
\end{theorem}

\begin{remark}
\label{rem-critical}
{\em The fluid limit will be said to be critical if
$\fx(t) = 1$ for all $t \in [0,\infty)$.  In addition, it will be
said to be subcritical (respectively, supercritical) if for every $T \in
[0,\infty)$,
$\sup_{t \in [0,T]} \fx(t) < 1$ (respectively, $\inf_{t \in [0,T]} \fx (t) >
1$).
Although, in general, the fluid limit may not stay in one regime
for all $t$ and may instead experience
 periods of subcriticality, criticality and supercriticality,
for many natural choices of 
initial data, such as either starting empty or starting 
on the so-called ``invariant manifold''
of the fluid limit, the fluid limit does belong to one of these 
three categories. 
Specifically, if we define the ``invariant'' fluid age measure to be  
\be
\label{eq-fmeasstar}
\fmeas_*(dx)  = (1-G(x)) \, dx, \quad x\in [0,\endsup), 
\ee
then  it follows from Remark 3.7 and Theorem 3.8 of 
Kaspi and Ramanan 
\cite{KasRam07} that the
fluid limit associated with the initial data $(\f1, 1, \fmeas_*)$
is critical, the fluid limit associated with the
initial data $(\f1, a, \fmeas_*)$ for some
$a > 1$ is supercritical, and if the support of $G$ is $[0,\infty)$, then
the fluid limit associated with the initial data $(\flam \f1, 0, \measzero)$ is
subcritical whenever $\flam \leq 1$.  A complete characterization of the 
invariant manifold of the fluid in the presence of abandonments 
can be found in \cite{KanRam10a}. 
}
\end{remark}

\beginsec

\section{Certain Martingale Measures and their Stochastic Integrals}
\label{subs-mmeas}

We now introduce some quantities that appear in the characterization of the 
functional central limit. The sequence of  martingales 
obtained by compensating the departure processes in each of 
the $N$-server systems played an important role in establishing  the 
fluid limit result in \cite{KasRam07}.  Whereas under the fluid scaling 
the limit of this sequence 
converges weakly to zero, under the diffusion scaling considered 
here, it converges to a non-trivial limit.  
This limit can be described in terms of 
a corresponding martingale measure, which is introduced in 
in Section \ref{subsub-mmeas}.  In Section \ref{subs-convint}  
 certain stochastic convolution integrals with respect to these 
martingale measures are introduced, 
which arise in the representation formula for the centered age 
process in the $N$-server system (see Proposition \ref{cor-sae1}). 
Finally, the associated ``limit'' quantities are 
 defined in Section \ref{subs-scaledmmeas}. 
The reader is referred to Chapter 2 of Walsh \cite{walshbook} for 
basic definitions of martingale measures and their stochastic integrals.

\subsection{A Martingale Measure Sequence} 
\label{subsub-mmeas} 

Throughout this section, suppose  that $(\en, \initxn, \measn_0)$ is an
$\newspace$-valued random element representing the initial data of
the $N$-server system, and let $(\ren, \xn, \measn)$ be the associated
state process  described in Section \ref{subs-aux}.
For any measurable function $\newf$ on $[0,\endsup)\times\R_+$, consider the
sequence of processes $\{\bare^{(N)}_\newf\}_{N \in \N}$ defined by 
 \be
\label{def-baren}
   \bare^{(N)}_\newf (t) \doteq  \sum_{s \in [0,t]} \sum_{j=-\lan \f1,
     \measn_0 \ran  + 1}^{\kn (t)}
\ind_{\left\{\frac{d}{dt} \agen_j (s-) > 0, \frac{d}{dt} \agen_j (s+)  = 0
  \right\}} \newf(\agen_j (s),s) \ee
for $N \in \N$ and $t \in [0,\infty)$,
where $\kn$ and $\agen_j$ are, respectively, 
the cumulative entry into service process and the age process 
of customer $j$ as defined by the relations
(\ref{def-kn}) and (\ref{def-agejn}). 
Note from (\ref{def-agejn})  that the $j$th customer 
completed service (and hence departed the system) at time $s$
 if and only if
\[ \dfrac{d}{dt} \agen_j (s-) > 0 \qquad \mbox{ and  } \qquad
\dfrac{d}{dt} \agen_j (s+) = 0. \]
 Thus, substituting  $\newf = \f1$ in (\ref{def-baren}),
we see that  $\bare^{(N)}_\f1$ is equal to   $\dn$,
the cumulative departure process  of (\ref{def-dn}).
Moreover, for $N \in \N$ and 
every bounded measurable function
$\newf$ on $[0,\endsup) \times [0,\infty)$,
consider the  process 
$\dcompn_{\newf}$ defined by
\be 
\label{def-dcompn}
\dcompn_{\newf}(t)\doteq
\int_0^t\left(\int_{[0,\endsup)} \newf(x,s)h(x)\, \measn_s(dx)\right)ds,
\qquad t \in [0,\infty), 
\ee
and set
\be
\label{def-martn}
 \martn_{\newf}\doteq \qn_{\newf}-\dcompn_{\newf}.
\ee
It was shown in Corollary 5.5 of Kaspi and Ramanan \cite{KasRam07} 
that for all bounded, continuous functions $\newf$
defined on $[0,\endsup) \times [0,\infty)$,
$\dcompn_{\newf}$ is the $\{{\cal F}_t^{(N)}\}$-compensator of
$\qn_{\newf}$, and that $\martn_{\newf}$ is a c\`{a}dl\`{a}g $\{{\cal
  F}_t^{(N)}\}$-martingale 
(see also Lemma 5.2 of Kang and Ramanan 
\cite{KanRam08} for a generalization of this result 
to a larger class of $\newf$).   
Moreover, from  the proof of Lemma 5.9 of Kaspi and Ramanan \cite{KasRam07}, 
it follows that the predictable quadratic variation of
 $\martn_{\newf}$ takes the form 
\be
\label{qv-martn} 
\lan \martn_{\newf} \ran_t = \dcompn_{\newf^2}(t) = 
\int_0^t \left(\int_{[0,\endsup)} \newf^2(x) h(x) \, \measn_s(dx) \right)\, ds,  
\quad t\in [0,\infty). 
\ee

Now, for $B \in {\cal B}[0,\endsup)$ and 
$t \in [0,\infty)$, 
 define 
\be
\label{def-cmn}
\cmn_t (B)  \doteq \mn_{\ind_B} (t) = 
\baren_{\ind_B} (t) - \dcompn_{\ind_B} (t). 
\ee
Let ${\cal B}_0[0,\endsup)$ denote the algebra 
generated by the intervals $[0,x]$, $x \in [0,\endsup)$. 
It is easy to verify that 
$\cmn = \{\cmn_t (B), {\cal F}_t^{(N)}, t \geq 0, B \in {\cal
  B}_0[0,\endsup)\}$  is a  martingale measure (for completeness, a proof 
is provided in Lemma \ref{lem-martmeas} of the Appendix). 
We now show that $\cmn$ is in fact an orthogonal  
martingale measure (see page 288 of Walsh \cite{walshbook} for a definition). 
Essentially, the orthogonality property holds  because
almost surely, no two departures occur at the same time.  
First, in Lemma \ref{lem-dep}  below, 
we first state a slight
generalization of this latter property,
which is also used in Section \ref{subs-asind}
to establish an asymptotic independence result.
Given $r, s \in [0,\infty)$,
 let $D^{(N),r}(s)$ denote the cumulative number of
departures during $(r,r+s]$ of customers that entered
service at or before $r$.   In what follows, recall that the 
notation $\Delta f (t) = f(t) - f(t-)$ is used to denote
the jump of a function $f$ at $t$.

\begin{lemma}
\label{lem-dep}
For every $N\in \N$, $\P$ almost surely,
\be
\label{ineq1-dep}
\Delta \dn(t) \leq 1,   \qquad  t \in [0,\infty),
\ee
and
\be
\label{ineq2-dep}
\sum_{s \in [0,\infty)} \Delta \en(r+s) \Delta D^{(N),r}(s) = 0, 
\qquad  r \in [0,\infty).
\ee
\end{lemma}

We relegate the proof of the lemma 
to Section \ref{subs-lemdep} and instead, now establish 
 the orthogonality property.

\begin{cor}
\label{cor-orth}
For each $N \in \N$, the martingale measure
$\cmn = \{\cmn_t(B), {\cal F}_t^{(N)}; t \geq 0, B \in {\cal
  B}_0[0,\endsup)\}$ is orthogonal and has
covariance functional
\begin{eqnarray}
\label{cov-fnal} \qquad
\covfnal^{(N)}_t (B,\tilde{B}) \doteq \left\lan \cmn (B), \cmn (\tilde{B})
\right\ran_t & = &  \dcompn_{\ind_{B \cap \tilde{B}}}
 (t) \\
\nonumber
& = & \int_0^t \left( \int_{B \cap \tilde{B}} h(x) \,
\measn_s (dx) \right) \, ds
\end{eqnarray}
for $B, \tilde{B} \in {\cal B}_0[0,\endsup)$.
\end{cor}
\begin{proof}
In order to show that the martingale measure $\cmn$ is orthogonal, it
suffices to show that for every $B, \tilde{B} \in {\cal B}_0[0,\endsup)$ 
such that $B\cap \tilde{B}=\emptyset$, the martingales $\{\cmn_t(B); t \geq 0\}$ and $\{ \cmn_t(\tilde{B}); t \geq 0\}$ are orthogonal or,
in other words,  that
\be
\label{orth}
B \cap \tilde{B} = \emptyset  \quad \Rightarrow \quad
\lan \cmn (B), \cmn (\tilde{B})\ran  \equiv 0.
\ee
Here, $\lan \cdot, \cdot \ran$ represents the predictable 
quadratic covariation between the two martingales. 
Fix  two sets $B, \tilde{B} \in {\cal B}_0[0,\endsup)$ with 
$B \cap \tilde{B} = \emptyset$.  By 
Lemma 5.2 of Kang and Ramanan 
\cite{KanRam08}, 
$\cmn (B) = \martn_{\ind_B}$ and $\cmn (\tilde{B}) =
\martn_{\ind_{\tilde{B}}}$ are martingales that are
 compensated sums of jumps,
where the jumps occur at  departure times of customers whose ages lie 
in the sets $B$ and $\tilde{B}$, respectively.
 Since, by (\ref{ineq1-dep}) of Lemma \ref{lem-dep}, 
 there are almost surely no two departures that occur at the same time,
it follows that almost surely, the set of jump points of $\cmn (B)$ and $\cmn
(\tilde{B})$ are disjoint.
By Theorem 4.52 of
 Chapter 1 of Jacod and Shiryaev \cite{JacShiBook}, it then follows that
the martingales are orthogonal and (\ref{orth}) holds.  
Combining (\ref{orth}) with (\ref{qv-martn}) 
 and the
biadditivity of the covariance functional, we then obtain (\ref{cov-fnal}).
\end{proof}

Due to the orthogonality property established in Corollary
\ref{cor-orth}, we can now define stochastic integrals
with respect to the martingale measure $\cmn$.
 Since $\E[\dcompn_{\f1} (T)] < \infty$ by Lemma 5.6 of Kaspi and Ramanan
\cite{KasRam07} and $\measn$ is a non-negative measure, 
the stochastic integral is defined for the space of deterministic, 
continuous and bounded functions on $[0,\endsup) \times [0,\infty)$ 
(it is in fact defined for a larger class of so-called predictable 
integrands satisfying a suitable integrability property, 
see page 292 of Walsh \cite{walshbook}).  Moreover, by 
 Theorem 2.5 on page 295 of Walsh \cite{walshbook}, it follows 
that for all bounded and continuous $\newf$, 
the stochastic  integral
$\{ \cmn_t (\newf)(B), \filtn; t \geq 0, B \in
{\cal B}_0[0,\endsup)\}$  of $\newf$ with respect to $\cmn$ is 
 a c\`{a}dl\`{a}g orthogonal martingale measure with covariance functional
 \be
\label{orth-sint}
 \lan  \cmn(\newf) (B),   \cmn(\tilde{\newf}) (\tilde{B}) \ran_t
= \int_0^t \left( \int_{B \cap \tilde{B}} \newf(x,s) \tilde{\newf}(x,s) h(x)
  \measn_s (dx) \right) \, ds
\ee
for bounded, continuous $\newf, \tilde{\newf}$ and $B, \tilde{B} \in {\cal
  B}_0[0,\endsup)$.
When $B = [0,\endsup)$, we will drop the dependence on $B$ and simply write
\[ \cmn (\newf)  = \cmn (\newf) ([0,\endsup)).
\]

\begin{remark}
\label{rem-mncadlag}
{\em 
For  $\newf \in {\cal C}_b([0,\endsup) \times \R_+)$,  
 the stochastic integral $\cmn (\newf)$ admits a  c\`{a}dl\`{a}g 
version. Indeed, the  c\`{a}dl\`{a}g martingale
$\martn_{\newf}$ defined in (\ref{def-martn}) is a version  of 
the stochastic integral $\cmn (\newf)$.
}
\end{remark}

It was shown in Lemma 5.9 of Kaspi and Ramanan \cite{KasRam07} that 
\[ \overline{{\cal M}}^{(N)}  \doteq \dfrac{\cmn}{N} \Rightarrow \overline{\cm} \equiv
\measzero
\]
 in the space
of c\`{a}dl\`{a}g finite Radon measure valued processes.
Therefore, consistent with the diffusion scaling (\ref{diff-scaling3}), 
we set 
\be
\label{def-chatmn}
 \chatmn \doteq \dfrac{\cmn}{\sqrt{N}}.
\ee
It is clear from the above discussion that each $\chatmn$ is an orthogonal
martingale measure with covariance functional
\[ \widehat{\covfnal}^{(N)}_t (B, \tilde{B}) = \int_0^t \left( \int_{B \cap \tilde{B}}
  h(x) \, \fmeasn_s (dx) \right)\, ds, \]
and that for any $\newf$ in a suitable class of functions that 
includes the space of bounded and continuous functions, 
the stochastic integral $\chatmn(\newf)$
is a well defined c\`{a}dl\`{a}g, 
 orthogonal  $\{{\cal F}_t^{(N)}\}$ martingale measure. 
Moreover, for every bounded, continuous $\newf, \tnewf$ and $t \in [0,\infty)$, 
\be
\label{cfnal-sint}
\lan \chatmn(\newf),  \chatmn (\tilde{\newf}) \ran_t
= \int_0^t \left( \int_{[0,\endsup)} \newf(x,s) \tilde{\newf}(x,s) h(x)
\,   \fmeasn_s (dx) \right) \, ds.
\ee

\subsection{Some Associated Stochastic Convolution Integrals}
\label{subs-convint}

In  Proposition \ref{cor-sae1} we show that 
the stochastic measure-valued process $\{\measn_t,t \geq 0\}$ that 
describes the ages of customers in the $N$-server system admits 
a convenient representation that is similar to the 
 representation (\ref{final-fleqs}) for  its 
fluid counterpart $\{\fmeas_t,t \geq 0\}$, 
except that it contains an additional stochastic term involving 
the following stochastic convolution integral. 
For $N \in \N$, $\newf \in {\cal C}_b([0,\endsup) \times [0,\infty))$ and $t \in [0,\infty)$, define  
\be
\label{def-chn}
  \chn_t (\newf) \doteq 
\underset{[0,\endsup) \times [0,t]}{\int \int} \newf(x+t-s,s)  \dfrac{1-G(x+t-s)}{1
  -  G(x)} \cmn (dx, ds), 
\ee
where the latter  stochastic integral with respect to  $\cmn$ 
 is well defined because $\cmn$ is an  orthogonal 
martingale measure and the function  
$(x,s) \mapsto \newf(x+t-s,s) (1-G(x+t-s))/(1-G(x))$ lies 
in ${\cal C}_b([0,\endsup) \times \R_+)$ for all
$\newf \in   {\cal C}_b([0,\endsup) \times \R_+)$.  
The scaled version of this quantity is then defined in the 
obvious manner: 
for $N \in \N$, $\newf \in {\cal C}_b([0,\endsup) \times [0,\infty))$ and $t \in [0,\infty)$, let 
\be
\label{def-hathn}
 \hathn_t (\newf) \doteq 
\underset{[0,\endsup) \times [0,t]}{\int \int} \newf(x+t-s,s)  \dfrac{1-G(x+t-s)}{1
  -  G(x)} \chatmn (dx, ds).
\ee

\subsection{Related Limit Quantities}
\label{subs-scaledmmeas}

We now define some additional quantities, which we 
subsequently show (in Corollaries 
\ref{cor-mreg} and \ref{cor-hreg}) 
to be limits of the sequences 
$\{\chatmn\}_{N \in \N}$ and $\{\hathn\}_{N \in \N}$.  
Fix a probability space 
$(\widehat{\Omega},\widehat{{\cal F}}, \widehat{\P})$ and, on this space,
 let 
 $\chatm = \{\chatm_t(B), B \in {\cal B}_0[0,\endsup), t \in [0,\infty)\}$
be a  continuous martingale measure with the  deterministic 
covariance functional
\be
\label{def-covfnalchatm}
 \widehat{\covfnal}_t (B, \tilde{B})  \doteq \lan \chatm(B), \chatm(\tilde{B})
\ran_t = \int_0^t \left( \int_{[0,L)}
  \ind_{B \cap \tilde{B}}(x) h(x) \fmeas_s (dx) \right) \, ds 
\ee
for $t \in [0,\infty)$.  Thus,  $\chatm$  is  a white noise. 
Let $\integspace$ denote the subset of continuous  functions 
on $[0,\endsup) \times [0,\infty)$ that satisfies 
\be
\label{integrand-cond}
\int_0^t \left( \int_{[0,\endsup)} 
\newf^2(x,s) h(x) \, \fmeas_s (dx) \right)\, ds < \infty,  \qquad \ t \in
[0,\infty). 
\ee
Note that $\integspace$ includes, in particular, 
the space of bounded and continuous functions. 
For any such $\newf \in \integspace$, 
the stochastic integral $\chatmn(\newf)$
is a well defined c\`{a}dl\`{a}g, 
 orthogonal  $\{\widehat{{\cal F}}_t, t \geq 0\}$ martingale measure 
(see page 292 of Walsh \cite{walshbook} for details).   
For any $\newf\in\integspace$ and $t \in [0,\infty)$,  the stochastic 
integral of $\newf$ with respect to $\chatm$ on $[0,\endsup) \times [0,t]$, 
denoted by 
\be\label{def-chatm}
 \chatm_t(\newf) \doteq \underset{[0,\endsup) \times [0,t]}{\int \int} \newf(x,s) \, 
\chatm(dx, ds), 
\ee 
is well defined.  Moreover, because $\chatm$ is a continuous martingale measure,  
$\chatm(\newf)$ has a version as a continuous real-valued process.   
In fact, as 
Corollary \ref{cor-mreg} shows,  $\chatm$ admits a version as a continuous  
$\H_{-2}$-valued process.

Next, for $t \in [0,\infty)$ and $f \in {\cal C}_b[0,\endsup)$, let 
 $\hath_t(f)$ be the random variable given by the 
following convolution integral:
\be
\label{def-hath}
   \hath_t(f) \doteq \underset{[0,\endsup) \times [0,t]}{\int \int}
f(x+t-s) \dfrac{1 - G(x+t-s)}{1-G(x)}
\, \chatm (dx, ds).
\ee
It is shown in  Lemma \ref{lem-hathn} 
that if $f$ is bounded and H\"{o}lder continuous then  the real-valued 
stochastic process 
$\hath(f) = \{\hath_t(f), t \geq 0\}$  admits a continuous version, 
and that $\hath$ also admits a version as a continuous $\H_{-2}$-valued 
process. 

In order to write the convolution integrals
 in a more succinct fashion, we introduce 
the family of operators  $\{\Psi_t, t \geq 0\}$ defined, for $t > 0$ and 
$(x,s) \in [0,\endsup) \times [0,\infty)$, by 
\be
\label{def-thetat}
 \left( \Psi_t f \right) (x,s) \doteq  f (x + (t-s)^+)
\dfrac{1 - G(x+(t-s)^+)}{1-G(x)},  
\ee
for bounded measurable $f$, 
where recall $(t-s)^+ = \max(t-s,0)$. 
Each operator $\Psi_t$  maps the space of
bounded measurable functions on $[0,\endsup)$ to the 
space of bounded measurable functions on $[0,\endsup) \times [0,\infty)$
and we also have 
\be
\label{op-ineqs}
 \quad
\sup_{t\in [0,\infty)} \nrm{\Psi_t \newf}_\infty \leq \nrm{\newf}_{\infty}.
\ee
We can then write $\hath$ and $\hathn$, respectively, 
in terms of $\chatm$ and $\chatmn$ as follows: 
\be
\label{rel-hath}
\hath_t (f) = \chatm_t (\Psi_t f), \qquad \hathn_t (f) = \chatmn_t (\Psi_t f), 
\qquad t \geq 0.  
\ee

\beginsec

\section{Main Results}
\label{sec-mainres}

The main results of the paper are stated in
Section \ref{subs-mainres}. 
They rely on some basic assumptions and the definition of 
a certain map, which  are first introduced in Sections \ref{subs-ass} 
and \ref{subs-mssm}, respectively. 
Corollaries of the main results  
are discussed in Section \ref{subs-cor}.

\subsection{Basic Assumptions}
\label{subs-ass}

We now state our assumptions on the arrival processes and initial conditions.  
  For $Y =  E, x_0,\nu,  X,  K$, let $\overline{Y}$
be the corresponding fluid limit as described in Theorem \ref{th-flimit}.
For $N \in \N$, the diffusion scaled quantities $\widehat{Y}^{(N)}$ are
defined as follows:
\be
\label{diff-scaling3} \widehat{Y}^{(N)}  \doteq \sqrt{N} \left(
\overline{Y}^{(N)} - \overline{Y} \right).
\ee
 For simplicity, we  restrict the arrival processes  
to be either renewal processes or time-inhomogeneous Poisson processes.  

\begin{ass}
\label{as-hate}  
The sequence  $\{\en\}_{N \in \N}$
of cumulative arrival processes satisfies one of 
the following two conditions: 
\begin{enumerate}
\item[(a)]  there exist constants $\flam, \sigma^2 \in (0,\infty)$ 
and $\beta \in \R$, such that 
for every $N \in \N$, $\en$ is a renewal process with i.i.d.\ 
inter-renewal times $\{\xi_j^{(N)}\}_{j \in \N}$ 
that have mean $1/\lambda^{(N)}$ and variance 
$(\sigma^2/\flam)/ (\lambda^{(N)})^2$, where 
 \be
\label{cond-qed}
\lambda^{(N)} \doteq \flam N - \beta \sqrt{N}.
\ee
\item[(b)] there exist locally integrable functions $\flam$ and 
$\beta$ on $[0,\infty)$ such that for every $N \in \N$, 
$\en$ is an inhomogeneous Poisson process with intensity function 
\be
\label{cond-qed2}
\lambda^{(N)} (t) \doteq \flam(t)N - \beta(t) \sqrt{N}, 
\qquad t \in [0,\infty). 
\ee
\end{enumerate}
\end{ass}

\begin{remark}
\label{rem-asdiff1}  
{\em  Let $\flam$ and $\beta$ be the locally integrable 
functions defined in Assumption \ref{as-hate} (which are 
constant if Assumption \ref{as-hate}(a) holds) and let 
$\sigma(\cdot)$ be the locally square integrable function that 
equals the constant $\sqrt{\sigma^2}$ if Assumption \ref{as-hate}(a) holds, 
and equals $\sqrt{\lambda(\cdot)}$ if Assumption \ref{as-hate}(b) holds. 
Then, given a standard Brownian motion $B$, the process $\hate$ given by 
 \be
\label{def-hate}
 \hate (t) \doteq \int_0^t \sigma(s) \, dB(s) - \int_0^t \beta(s)\, ds,  \qquad t \in [0,\infty),
\ee 
is a well defined diffusion and therefore a semimartingale, with $\int_0^t
\sigma(s) \, dB(s), t \geq 0$ being 
the local martingale and $\int_0^t b(s) \, ds, t \geq 0,$ the finite variation process 
in the decomposition. 
If Assumption \ref{as-hate}  holds then 
it is easy to see that $\fe$ in 
Assumption \ref{as-flinit} is given by 
$\fe(t) = \int_0^t \flam (s) \, ds$, $t \geq 0$,  and 
  $\haten \Rightarrow \hate$ as $N \ra \infty$  
 (a proof of the latter convergence can be found in  
Proposition \ref{prop-martconv}, which establishes a more general 
result).  
}
\end{remark}

We now impose a technical condition on the service
distribution, which is used mainly
 to establish the convergence of $\hathn(f)$ to $\hath(f)$ 
in ${\cal D}_{\R}[0,\infty)$ for H\"{o}lder continuous $f$  
in Section \ref{sec-martconv}. 


\begin{ass}
\label{as-holder} For every  $x \in [0,\endsup)$, 
the function from $[0,\endsup)$ to $[0,1]$ that takes 
$y \mapsto (1-G(x+y))/(1-G(x))$ is H\"{o}lder continuous on $[0,\infty)$, 
uniformly in $x$, i.e., there exist   
$\hconst_G < \infty$ and $\hexp_G \in (0,1)$ and $\delta > 0$
such that for every $x \in [0,\endsup)$ and 
 $y, y^\prime \in [0,\endsup)$ with $|y-y^\prime| < \delta$,
\be
 \label{ineq-holder}
   \dfrac{|G(x+y) - G(x+y^\prime)|}{1-G(x)} \leq  
\hconst_G |y-y^\prime|^{\hexp_G}.
 \ee
 \end{ass}

\begin{remark}
\label{rem-holder}
{\em As shown below,  Assumption \ref{as-holder} is satisfied if 
either  $h$ is bounded, 
or if there exists $l_0 < \infty$ such that 
$\sup_{x \in [l_0, \infty)}h(x)  < \infty$ and 
$G$ is uniformly H\"{o}lder continuous on $[0,\endsup)$.  
The hazard rate function $h$ is locally integrable, but 
not integrable, on $[0,\endsup)$.  Thus,  if $h$ is bounded (either 
uniformly or on $[\ell_0, \infty)$ for some $\ell_0 < \infty$) 
it must be that $\endsup = \infty$. 
Under the first condition above, 
for any $x, y, y' \in [0,\infty)$, $y < y^\prime$, 
\[  \left|\dfrac{G(x+y) - G(x+y')}{1-G(x)}\right| = \int_{y}^{y'} \dfrac{g(x+u)}{1-G(x)} \,
du \leq \int_{y}^{y'} h(u) \, du \leq \nrm{h}_{\infty} |y - y^\prime|, 
\]
and so Assumption \ref{as-holder} is satisfied.  
On the other hand, if there only exists $\ell_0 < \infty$ such that $\sup_{x
  \in [\ell_0, \infty)} h(x) < \infty$, but 
$G$ is uniformly H\"{o}lder continuous on $[0,\infty)$, 
with constant $C < \infty$ and exponent $\gamma > 0$, 
then straightforward calculations show 
\[  \left|\dfrac{G(x+y) - G(x+y')}{1-G(x)}\right| \leq 
\max\left( \dfrac{C}{1-G(\ell_0)},
\nrm{\ind_{[\ell_0,\infty)}(x)h(x)}_{\infty}\right)(y-y')^{\gamma \wedge 1},
\]
and once again Assumption \ref{as-holder} is satisfied. 
A relatively easily verifiable  
sufficient condition for $G$ to be uniformly H\"{o}lder 
continuous is that  $g \in \L^{1+\alpha}$ for some $\alpha > 0$ (recall that 
since $g$ is a density, we automatically have $g \in \L^1$; thus the latter 
condition imposes just a little additional regularity on $g$). 
Indeed, in this case, by H\"{o}lder's inequality,  
\[ |G(x) - G(y)| = \left|\int_{x}^y g(u) \, du \right| \leq  \nrm{g}_{\L^{1+\alpha}} 
(y-x)^{\alpha/(1+\alpha)}, \]
and so $G$ is uniformly H\"{o}lder continuous with exponent
$\gamma = \alpha/(1+\alpha) < 1$.  
}
\end{remark}

Now, given $s \geq 0$, recall that 
$\hmeasn_s$ represents the (scaled and centered) 
initial age distribution at time $s$, and define
\[
\ops^{\hmeasn_s}_t (f) \doteq \int_{[0,\endsup)} f(x+t)
  \dfrac{1-G(x+t)}{1-G(x)} \hmeasn_s (dx),  \qquad f \in {\cal
    C}_b[0,\endsup),  t \geq 0. 
\]
$\{\ops^{\hmeasn_0}_t(f), t
\geq 0\}$ plays an important role in the analysis because it arises in 
the representation for $\lan f, \hmeasn\ran$ given 
in Proposition \ref{cor-sae1}.  
In order to write $\ops^{\hmeasn_s}$ 
more concisely, it 
will be convenient to introduce a certain family of operators. 
For $t \in [0,\infty)$, define 
\be
\label{def-aret}
  \left( \are_t f \right) (x) \doteq f(x+t) \dfrac{1-G(x+t)}{1-G(x)}, \qquad
x \in [0,\endsup).
\ee
Each $\are_t$  maps the space of (bounded) 
measurable functions on $[0,\endsup)$ into itself and 
\be
\label{eq-phibound}
  \sup_{t \in [0,\infty)} \nrm{\are_t f}_\infty \leq \nrm{f}_{\infty}.
\ee 
Moreover, for future purposes, we also note that 
 $\{\are_t,t \geq 0\}$ defines a semigroup, i.e.,
$\are_0 f = f$ and
\be
\label{eq-sgroup}
\are_t (\are_s f) = \are_{t+s} f, \qquad s, t \geq 0, 
\ee 
and, recalling the definition (\ref{def-thetat}) of the family 
of operators  
$\{\Psi_t,t \geq 0\}$, it is easily verified that for every bounded, measurable function 
$f$ on $[0,\endsup)$ and $s,t \geq 0$, 
\be
\label{rel-psiphi}
(\Psi_s \Phi_t f) (x,u) = (\Psi_{s+t} f)(x,u), \qquad (x,u) \in [0,\endsup)
\times [0,s].  
\ee
We can now rewrite $\ops^{\hmeasn_s}$ in terms of the 
operators $\Phi_t$, $t \geq 0$, as follows: 
\be
\label{def-ops}
\ops^{\hmeasn_s}_t (f) = \lan \Phi_tf, \hmeasn_s\ran, \qquad s, t \geq 0. 
\ee
Since $G$ is continuous, 
it is clear from (\ref{eq-phibound}) 
that $\Phi_t f \in {\cal C}_b[0,\endsup)$ when 
$f \in {\cal C}_b[0,\endsup)$ and hence, 
  $\{\ops_t^{\hmeasn_0}(f), f \in {\cal C}_b[0,\endsup), t \geq 0\}$ 
is a well defined stochastic process. 

\begin{remark}
\label{rem-holder2}
{\em  
Since $\hmeasn_0$ is a signed measure with finite total mass bounded 
by $\sqrt{N}$, it can be easily verified that almost surely for 
$f \in \H_1$, by the norm inequality (\ref{norm-ineq}), 
\[  |\lan f, \hmeasn_0\ran| \leq 
2 \sqrt{N} \nrm{f}_{\infty} \leq 2 \sqrt{N} \nrm{f}_{\H_1}. 
\]
Moreover, if Assumption \ref{as-holder} holds 
  calculations similar to  those in 
Remark \ref{rem-holder}, the norm inequality (\ref{norm-ineq}) and 
the Cauchy-Schwarz inequality show 
  that for  $f \in \H_1$ and $0 \leq s < t < \infty$, 
\[  |\ops^{\hmeasn_0}_t (f) - \ops^{\hmeasn_0}_s (f) |
\leq C_G (t-s)^{\gamma_G} \nrm{f}_{\infty} +  \nrm{f}_{\H_0} (t-s)^{1/2}
\leq (2C_G + 1) \nrm{f}_{\H_1} |t-s|^{\gamma_G \wedge 1/2}.  
\]
This shows that for every $N \in \N$, 
$\ops^{\hmeasn_0}$ is a \cad process that almost surely takes values 
in $\H_{-1}$. 
}
\end{remark}

We now consider the initial conditions. 
We impose fairly general assumptions on the 
initial age sequence so 
as to establish the Markov property for the limit process.  
As shown in Lemma \ref{lem-consistency}, these conditions are consistent 
 in the sense that they are satisfied  at any time $s > 0$ if they are satisfied at time $0$.  
In addition, they are trivially satisfied if each $N$-server system starts 
precisely at the fluid limit, i.e., if $\hmeasn_0 = 0$ for 
every $N$.  The reader may prefer to make the latter assumption on first 
reading to avoid  the technicalities in the statement of this assumption.
To motivate the form of the assumptions, first note that 
the total variation of the sequence of finite signed measures 
 $\{\hmeasn_0\}_{N \in \N}$  
tends to 
 infinity as $N \ra \infty$, and so it is not reasonable to 
expect the sequence to converge in the space of finite or Radon measures. 
Instead, we impose convergence in a slightly different 
space.  As observed in Remark \ref{rem-holder2}, $\hmeasn_0$ can be viewed 
as an $\H_{-1}$-valued (and therefore $\H_{-2}$-valued) random element 
and under Assumption \ref{as-holder}, $\{\ops^{\hmeasn_0}_t, t \geq 0\}$ is 
a \cad $\H_{-1}$-valued stochastic process and $\{\ops^{\hmeasn_0}_t(\f1), t \geq 0\}$ 
is a \cad real-valued process 
 (see Section \ref{subsub-fun} for a definition of these spaces). 

\begin{ass} 
\label{as-diffinit}  
There exists an 
$\R$-valued random variable $\hinitx$ and 
a family of random variables $\{\hmeas_0(f), f \in \acbl\}$,  
 all  defined on a 
common probability space, 
such that 
\begin{enumerate}
\item[(a)]
$\hmeas_0$ admits a version as an $\H_{-2}$-valued random element; 
\item[(b)]  For $f \in \acbl$, given 
\be
\label{rel-ops}
 \ops_t^{\hmeas_0} (f) \doteq \hmeas_0\left( \Phi_t f\right) = 
\hmeas_0\left(f(\cdot+t)
  \dfrac{1-G(\cdot+t)}{1-G(\cdot)}\right), \qquad t \geq 0, 
\ee
$\{\ops^{\hmeas_0}_t(f), f \in \acbl, t \geq 0\}$ is a family 
of random variables such that $\{\ops^{\hmeas_0}_t, t \geq 0\}$ 
admits a version as a  continuous $\H_{-2}$-valued process, 
$\{\ops^{\hmeas_0}_t(\f1), t \geq 0\}$ admits a version as  a continuous
$\R$-valued 
process and, for every $t > 0$ 
almost surely, $f \mapsto \ops^{\hmeas_0}_t(f)$ is a measurable 
mapping from $\acbl \subset {\cal C}_b[0,\endsup) \mapsto 
\R$ (both equipped with their respective Borel $\sigma$-algebras); 
 \item[(c)]
as $N \ra \infty$, $(\hinitxn, \hatmeasn_0, \ops^{\hmeasn_0}, \ops^{\hmeasn_0}(\f1))
\Rightarrow (\hinitx, \hmeas_0, \ops^{\hmeas_0}, \ops^{\hmeas_0}(\f1))$ in 
$\R \times \H_{-2} \times
{\cal D}_{\H_{-2}}[0,\infty) \times {\cal D}_{\R}[0,\infty)$.
\end{enumerate}
\end{ass}

For some results, we will require the following strengthening 
of Assumption \ref{as-diffinit}. \\

\noi 
{\bf Assumption \ref{as-diffinit}'}
{\em 
Assumption \ref{as-diffinit} holds and in addition, 
\begin{enumerate}
\item[(d)] 
Suppose that $\newf \in {\cal C}_b([0,\endsup) \times [0,\infty))$ is such that 
for every $r > 0$, $x \mapsto \newf(x, r)$ is absolutely continuous and, 
for every  $T < \infty$, 
$\newf_x(\cdot, \cdot)$ is integrable on $[0,\endsup) \times [0,T]$ 
and $x \mapsto \int_0^t \newf(x,r ) \, dr$ is H\"{o}lder continuous. 
Then $\P$-almost surely, $r \mapsto \hmeas_0\left(\Phi_r \varphi(\cdot,r) \right)$
is measurable and for every 
$t \geq 0$,    
\be
\label{as-fub2}
\int_0^t \hmeas_0 \left( \Phi_r \varphi(\cdot,r) \right) \, dr 
= \hmeas_0 \left( \int_0^t \Phi_r \varphi(\cdot,r) \, dr \right).  
\ee
\end{enumerate}
}

Now, let $(\widehat{\Omega}, \widehat{{\cal F}}, \widehat{\P})$ be a common
 probability space that supports 
the martingale measure $\chatm$ introduced in Section \ref{subs-scaledmmeas},
 the standard Brownian motion $B$ of Remark \ref{rem-asdiff1},  
the family of random variables $\hmeas_0(f)$, $f \in \acbl$, and 
the random variable 
$\hinitx$ of Assumption \ref{as-diffinit} 
such that $\chatm, B$ and $(\hinitx, \hmeas_0(f), f \in \acbl)$ are mutually independent. 
Let $\widehat{{\cal F}}_0$ be the  $\sigma$-algebra 
generated by $(\hinitx, \hmeas_0(f), f \in \acbl,$ 
and for $t \geq 0$, let $\widehat{{\cal F}}_t \doteq 
\widehat{{\cal F}}_0 \vee \sigma (B_s, \chatm_s, s \in [0, t])$.
Then for $t \geq 0$, 
$\ops_t^{\hmeas_0}(f)$, $f \in \acbl$, (and, in particular, 
$\ops_t^{\hmeas_0}(\f1)$)  
 are all well defined $\widehat{{\cal F}}_0$-measurable random variables, and 
$(\hate_t, \chatm_t)_{t \geq 0}$ are $\{\widehat{{\cal F}}_t\}_{t \geq
  0}$-adapted stochastic processes. 
The  description of the $N$-server model listed prior to 
Remark \ref{rem-cumarr} assumes that for each $N \in \N$, given $\ren(0)$,  
 $\en$ is independent 
of the initial conditions $\measn_0$ and $\xn(0)$.  
Together with Assumptions \ref{as-hate}, \ref{as-diffinit} and the fact that 
$\ren(0) \ra 0$ almost surely as $N \ra \infty$, this implies that  
as $N \ra \infty$, 
\be
\label{weak-limit}
(\haten, \hinitxn, \hatmeasn_0, \ops^{\hmeasn_0}, \ops^{\hmeasn_0}(\f1))
\Rightarrow (\hate, \hinitx, \hmeas_0, \ops^{\hmeas_0}, \ops^{\hmeas_0}(\f1))
\ee
in ${\cal D}_{\R}[0,\infty) \times \R \times \H_{-2} \times
{\cal D}_{\H_{-2}}[0,\infty) \times {\cal D}_{\R}[0,\infty)$.

\subsection{The Centered Many-Server Map}
\label{subs-mssm}

We introduce a map, which we refer to as the 
centered many-server map, which will be used to 
characterize the limit.  
Let ${\cal D}_{\R}^0[0,\infty)$ be the subset of functions $f$ in
${\cal D}_{\R}[0,\infty)$ with $f(0) = 0$. 
The input data for this map lies in the following space:
\[ \dataspace \doteq
{\cal D}_{\R}^0[0,\infty) \times \R \times {\cal D}_{\R}[0,\infty).  
\]

\begin{defn}
\label{def-mseq}
{\em ({\bf Centered Many-Server Equations})
Let $\genflx \in {\cal D}_{\R_+}[0,\infty)$ be fixed.
Given $(\gene,  \geninitx_0, \addfn)$ $ \in \dataspace$,
we say that
$(\genk, \genx, v) \in {\cal D}_{\R}^0[0,\infty)\times {\cal D}_{\R}[0,\infty)^2$ solves
the centered many-server equations (henceforth, abbreviated  CMSE) associated with
$\fx$ and $(\gene,  \geninitx_0, \addfn)$  
if and only if 
for $t \in [0,\infty)$,
\begin{eqnarray}
\label{ref-meas}
v(t) & = & \addfn (t) +  K(t) - \int_0^t g(t-s) K(s) \, ds, \\
\label{ref-sm}
K(t) & = & E(t) + x_0 - X(t) + v(t)  - v(0)  
\end{eqnarray}
and
\be
\label{ref-nonidling}
v(t) = \left\{
\ba{ll}
\genx(t) & \mbox{ if } \genflx(t) < 1, \\
\genx(t) \wedge 0 & \mbox{ if } \genflx(t)  = 1, \\
0 & \mbox{ if } \genflx(t) > 1.
\ea
\right.
\ee
}
\end{defn}

Note that this definition automatically requires that $E(0) = K(0) = 0$, 
$X(0) = x_0$ and  
$\addfn (0) = v(0)$,  
which  equals $x_0$, $x_0 \wedge 0$ or $0$, depending on whether 
$\fx(0) < 1$, $\fx(0) = 1$ or $\fx(0) > 1$. 
It is shown in  Proposition \ref{prop-cont} 
 that there exists at most one solution
to the CMSE for any given input data in $\dataspace$.
When a solution exists, we let $\lc$
denote the associated "centered many-server'' mapping (associated with $\genflx$)
that takes $(\gene,  \geninitx_0, \addfn)$ $\in \dataspace$ 
to the corresponding solution $(\genk, \genx,  v)$. 
 Let $\dom (\lc )$ denote the domain of $\lc$, which is defined to be
 the collection of input data in $\dataspace$ 
for which a solution to the CMSE exists. 

\begin{remark}
\label{rem-contlambda}
{\em  Suppose $(\genk, \genx, v) \in \Lambda (\gene,  \geninitx_0, \addfn)$ 
for some  $(\gene,  \geninitx_0, \addfn) \in \dataspace$.  Then 
(\ref{ref-meas}) and (\ref{ref-sm}) together show that for $t \geq 0$, 
\be
\label{reed-comp0}
 \genx(t) = \geninitx_0 + \gene (t) + \addfn(t) - \int_0^t g(t-s) \left[
  \gene(s) + \geninitx_0 - v(0) - X(s) + v(s) \right] \, ds. 
\ee
Thus, if $\gene$, $\addfn$ and $g$ are continuous then $\genx$ is also continuous. 
If, in addition,  
the fluid limit is either subcritical, critical or supercritical then the
continuity of $\genx$ and 
(\ref{ref-nonidling}) imply the continuity of $v$ and, in turn,   
(\ref{ref-sm}) implies the continuity of  $\genk$.
}
\end{remark}

The importance of the CMSE stems from the relation 
\be
\label{cmse-eq}
  (\hatkn, \hatxn, \hmeasn(\f1)) = \Lambda \left( \haten, \widehat{x}_0^{(N)},
    \ops^{\hmeasn_0}(\f1) - \hathn(\f1)\right), \qquad N \in \N, 
\ee
which is established in Lemma \ref{lem-lambdarep} under the 
assumption that the fluid limit is either subcritical, critical or supercritical.

\subsection{Statements of Main Results}
\label{subs-mainres}

The first result of the paper, Theorem \ref{th-main} below, 
identifies the limit of the sequence $\{\hatxn\}_{N \in \N}$.  
Let 
\be
\label{def-y1n}
\widehat{Y}^{(N)}_1 \doteq (\haten, \widehat{x}^{(N)}_0, \hmeasn_0, 
\ops^{\hmeasn_0}, \ops^{\hmeasn_0}(\f1), \chatmn, \hathn, \hathn (\f1)), 
\ee
and let $\widehat{Y}_1$ be the corresponding quantity  
without the superscript $N$, 
where $\hate, \widehat{x}_0$, $\hmeas_0, \ops^{\hmeas_0},
\ops^{\hmeas_0}(\f1)$  are as defined in  Remark \ref{rem-asdiff1} and 
Assumption \ref{as-diffinit}, and   
 $\chatm$, $\hath$ and $\hath(\f1)$ are as defined in Section \ref{subs-scaledmmeas}.  
Also, define 
\be
\label{def-caly1} {\cal Y}_1 \doteq {\cal D}_{\R}[0,\infty) 
\times \R \times \H_{-2} \times  {\cal D}_{\H_{-2}}[0,\infty) \times  {\cal D}_{\R}[0,\infty) \times
 {\cal D}_{\H_{-2}}[0,\infty)^2 \times 
{\cal D}_{\R}[0,\infty). 
\ee

\begin{theorem}
\label{th-main}
Suppose Assumptions \ref{as-flinit}--\ref{as-diffinit} are satisfied and 
suppose that the fluid limit is either  
subcritical, critical or  supercritical. 
Then $(\hate, \widehat{x}_0, \ops^{\hmeas_0}(\f1) - \hath (\f1))$ lies 
in the domain of the centered many-server map $\Lambda$ and, 
as $N \ra \infty$,  
\be
\label{conv-main}
(\widehat{Y}_1^{(N)}, \hatx^{(N)}, \hatk^{(N)}, \lan \f1, \hmeasn\ran) 
\Rightarrow 
(\widehat{Y}_1, \hatx, \hatk, \hmeas(\f1)) 
\ee  
in ${\cal Y}_1 \times {\cal D}_{\R}[0,\infty)^3$,  
where $(\hatx,\hatk, \hmeas(\f1)) \doteq
 \Lambda(\hate, \widehat{x}_0,\ops^{\hmeas_0}(\f1) - \hath(\f1))$ is 
almost surely continuous. 
Furthermore, if $g$ is continuous, then 
\be
\label{rep-hatx}
 \hatx (t) = \hatx_0 + \hate(t)- \chatm_t (\f1) - \cendep(t),  \qquad 
t \in [0,\infty), 
\ee
where 
\be
\label{def-calg}
\cendep(t) \doteq \hmeas_0(\f1) - \ops^{\hmeas_0}_t(\f1) - \chatm_t (\f1) 
+ \hath_t (\f1) + \int_0^t g(t-s) \hatk(s) \, ds. 
\ee
\end{theorem}

The proof of Theorem \ref{th-main} is presented in Section \ref{subs-fclt}. 
In addition to establishing the representation (\ref{cmse-eq}), 
the key elements of the proof involve  showing the convergence 
$\widehat{Y}^{(N)}_1 \Rightarrow \widehat{Y}_1$, which is 
carried out in Corollary \ref{cor-hreg}, and establishing 
continuity of the centered many-server map $\Lambda$ and another 
auxiliary map $\Gamma$, 
which are established in Proposition \ref{prop-cont} and 
Lemma \ref{lem-calk}, respectively.

We now establish a more general convergence result for 
the pair $\{(\hatxn, \hmeasn)\}_{N \in \N}$, which 
automatically yields convergence of several functionals  
of the process.   The proof of this result is
 also given in Section \ref{subs-fclt}. 
 With $\hatk$ equal to  the limit 
process  obtained in Theorem \ref{th-main},  we define 
for all bounded and absolutely continuous $f$,  
\begin{eqnarray}
\label{def-hmeas}
\quad \hmeas_t(f)  & \doteq & {\cal S}_t^{\hmeas_0}(f)  + f(0) \hatk (t) + 
\int_0^t 
\hatk(s) f^\prime(t-s) (1-G(t-s)) \, ds \\
\nonumber & & \quad 
- \int_0^t \hatk(s) g(t-s) f(t-s) \, ds - \hath_t(f). 
\end{eqnarray}
Note that the first term on the right-hand side 
is well defined by the discussion following  Assumption \ref{as-diffinit} 
(see also Lemma \ref{lem-asmarkov}),  
 the next three terms are well defined because $\hatk$ is continuous 
and $f^\prime, (1-G)$, $g$ and $f$ are all locally integrable, 
and the last term is well defined since $\hath_t (f) = \chatm_t (\Psi_t f)$
and the continuity and boundedness of $f$ implies 
$\Psi_t f \in {\cal C}_b([0,\endsup) \times \R_+)$.

\begin{theorem}
\label{th-fclt}
Suppose Assumptions \ref{as-flinit}-\ref{as-diffinit} 
are satisfied, the fluid limit is either subcritical, critical or
supercritical and $g$ is continuous. 
Then, as $N \ra \infty$,  
\be
\label{conv-main2}
(\widehat{Y}_1^{(N)}, \hatx^{(N)}, \hatk^{(N)}, \hmeasn) \Rightarrow 
(\widehat{Y}_1, \hatx, \hatk, \hmeas) 
\ee
in ${\cal Y} \doteq {\cal Y}_1 \times {\cal D}_{\R}[0,\infty)^2 \times
{\cal  D}_{\H_{-2}}[0,\infty)$. 
\end{theorem}

A main focus of this paper is to show 
that the approximating process is a tractable process, 
thus demonstrating the usefulness of the  
approximation theorem obtained. 
The next two theorems  show that this is indeed 
the case  under some additional regularity conditions 
on the hazard rate $h$.  
First, in Theorem \ref{th-main1}
we show that $\{\hatx_t, \widehat{{\cal F}}_t, t \geq 0\}$ is also 
a semimartingale. 
By It\^{o}'s formula this enables the 
description of 
the evolution of a large class of functionals of the process. 
The proof of Theorem \ref{th-main1}
is given in Section \ref{subs-smg}. 

\begin{theorem}
\label{th-main1} 
Suppose that Assumptions \ref{as-flinit}, \ref{as-hate} and 
\ref{as-diffinit}' are satisfied,  
the fluid limit is either subcritical, critical or supercritical and 
 $h$ is bounded and absolutely continuous. 
If  $(\hatk, \hatx, \hmeas)$ is 
the limit process of Theorem \ref{th-fclt}, then   
 $\hatx$ and $\hatk$ are semimartingales with 
decompositions $\hatx = \hatx(0) + M^X + \bv^X$ and $\hatk = M^K + \bv^K$, respectively, 
where 
\[ M^X(t) = \int_0^t \sigma(s) \, dB(s) - \chatm_t(\f1), \qquad C^X(t) = -\int_0^t \beta(s) \, ds - \int_0^t \hmeas_s
(h) \, ds,  \qquad t \geq 0, 
\]
and if $\fx$ is subcritical, then $\hatk  = \hate$ and so 
\[  M^K (t) = \int_0^t \sigma(s) \, dB(s),  \qquad \bv^K (t)= -\int_0^t \beta(s) \, ds, \qquad t \geq 0,\]
if $\fx$ is supercritical, then 
\[M^K(t) = \chatm_t(\f1), \qquad \bv^K(t) = \int_0^t \hmeas_s
(h) \, ds, \qquad t \geq 0, \]
and if $\fx$ is critical, then  
\[ M^K (t) =  \int_0^t \ind_{\{\hatx(s) \leq 0\}} \sigma(s) \, dB_s  + \int_0^t \ind_{\{\hatx(s) >
  0\}}\,  d \chatm_s(\f1),  \qquad t \geq 0, 
\]
and 
\[  \bv^K(t) =  -\int_0^t \beta(s) \ind_{\{\hatx(s) \leq 0\}} \, ds + \int_0^t
\ind_{\{\hatx(s) > 0\}}\hmeas_s(h)\, ds + 
  \dfrac{1}{2} L^{\hatx}_0(t), \qquad t \geq 0, \]
where, $L^{\hatx}_0(t)$ is the local time of $\hatx$ at zero on the 
interval $[0,t]$.  
Moreover, for each $t > 0$ and $f \in \acbl$,  $\hmeas_t(f)$ admits 
the alternative representation 
\be
\label{def-hmeas2}
 \hmeas_t (f) = \ops^{\hmeas_0}_t (f) +  \int_0^t f(t-s) (1 -G(t-s)) \, d\hatk(s)
- \hath_t (f), 
\ee 
where the second term is the stochastic convolution integral with respect 
to the semimartingale $\hatk$.   
\end{theorem}

\begin{remark}
\label{rem-main1}
{\em   
If for $f \in {\cal C}_b [0,\endsup)$, 
$\{\ops^{\hmeas_0}_t(f), t \geq 0\}$ is a well defined stochastic process  
then $\{\hmeas_t(f), t \geq 0\}$ is also a well 
defined stochastic process given by the right-hand side of 
(\ref{def-hmeas2}).  Moreover, 
under a slight strengthening of 
the conditions of Theorem \ref{th-main1} (specifically, of 
  Assumption \ref{as-diffinit}), we can in 
fact show convergence for a slightly larger class of 
functions than those in $\H_2$.  More precisely, 
if 
for any bounded, H\"{o}lder continuous $f$, 
$\{\ops^{\hmeas_0}_t(f), t \geq 0\}$  defined 
in (\ref{rel-ops}) is a well defined 
continuous stochastic process and, as $N \ra \infty$,  
$(\hinitxn, \hatmeasn_0(f), \ops^{\hmeasn_0}(f), \ops^{\hmeasn_0}(\f1))
\Rightarrow (\hinitx, \hmeas_0, \ops^{\hmeas_0}(f), \ops^{\hmeas_0}(\f1))$
in $\R^2 \times {\cal D}_{\R}[0,\infty)^2$, 
then $\hmeas(f)$ is also a continuous process
 and,  as $N \ra \infty$, 
 $\hmeasn(f) \Rightarrow \hmeas(f)$ in ${\cal D}_{\R}[0,\infty)$.  
A brief justification of this assertion is provided at the end of Section
\ref{subs-smg}. 
 By the independence assumptions of the model, the above 
conditions automatically imply the joint convergence 
\be
\label{jt-limit}
(\haten, \hinitxn, \hatmeasn_0(f), \ops^{\hmeasn_0}(f), \ops^{\hmeasn_0}(\f1))
\Rightarrow (\hate, \hinitx, \hmeas_0(f), \ops^{\hmeas_0}(f), \ops^{\hmeas_0}(\f1))
\ee
in ${\cal D}_{\R}[0,\infty) \times \R^2 \times {\cal D}_{\R}[0,\infty)^2$.  
}
\end{remark}

We now  show that 
 the limiting process
$(\hatk, \hatx, \hmeas)$ described in Theorem \ref{th-main}
 can alternatively be characterized as the unique solution to 
a stochastic partial differential equation (SPDE),  coupled with an
 It\^{o} diffusion equation, and also satisfies 
a strong Markov property.  
We first introduce the SPDE, which we refer to as the 
stochastic age equation. 
In the definition of the stochastic age equation given below,
  $h$ is the hazard 
rate function of the service distribution  and  $\hmeas_0$, $\chatm$ and
$\hatk$ are the limit processes 
defined on the filtered probability space $(\widehat{\Omega}, \widehat{{\cal
    F}}, \{\widehat{{\cal F}}_t\}, \widehat{\P})$,  
as specified in Theorem \ref{th-main1}.

\begin{defn}
\label{def-sae}
{\bf (Stochastic Age Equation)}
Given $(\hmeas_0, \hatk, \chatm)$ defined on 
the filtered probability space $(\widehat{\Omega},
\widehat{{\cal F}}, \{\widehat{{\cal F}}_t\}, \widehat{\P})$, 
 $\meas = \{\meas_t, t \geq 0\}$  is said to be 
  a strong solution to the {\em stochastic age equation}
associated with $(\hmeas_0, \hatk, \chatm)$ 
if and only if for every 
 $f \in \acbl$, $\meas_t(f)$ is an 
$\widehat{{\cal F}}_t$-measurable random variable 
for $t > 0$,  $s \mapsto \meas_s(f)$ is almost surely 
measurable on $[0,\infty)$, 
$\{\meas_t, t \geq 0\}$ admits a version as a continuous, 
$\H_{-2}$-valued  process and 
$\P$-almost surely, 
for every $\newf \in  {\cal C}^{1,1}_b([0,\endsup) \times \R)$  such that 
$\newf_t(\cdot,s) + \newf_x(\cdot,s)$ is Lipschitz continuous for every $s$,   
and every $t \in [0,\infty)$,
\begin{eqnarray}
 \label{eq-sae}
 \qquad \meas_{t} (\newf(\cdot,t))  &=&
\meas_0 (\newf(\cdot,0))   +  \ds \int_{0}^t
\meas_s \left( \newf_x(\cdot,s) + \newf_s(\cdot,s) - \newf(\cdot,s) h(\cdot)\right) \, ds \\
\nonumber
&& \qquad
 - \underset{[0,\endsup)\times [0,t]}{\int \int}
 \newf (x,s) \, \chatm (dx, ds) + \int_0^t \newf(0,s) \, d\hatk_s. 
\end{eqnarray}
\end{defn}

\begin{theorem}
\label{th-main2}
Suppose Assumptions \ref{as-flinit}, \ref{as-hate} and  
\ref{as-diffinit}' are satisfied, 
the fluid limit is either subcritical, critical or supercritical  
and $h$ is absolutely continuous and bounded.  
Given the limit process $(\hatk, \hatx, \hmeas)$ 
of Theorem \ref{th-fclt}, the following assertions are true:  
\begin{enumerate}
\item 
 If the density $g^\prime$ of $g$ lies in  $\L_{loc}^{2}[0,\endsup)
 \cup L^{\infty}_{loc}[0,\endsup)$ 
 then 
 $\{\hmeas_t, \widehat{{\cal F}}_t, t \geq 0\}$ is the unique strong solution 
to the stochastic age equation associated with 
$(\hmeas_0, \hatk, \chatm)$; 
\item if $g^\prime/(1-G)$ is bounded then
 $\{ (\hatx_t, \hmeas_t, \ops^{\hmeas_t}(\f1)), \widehat{{\cal F}}_t, t \geq 0\}$ is an  
$\R \times \H_{-2} \times {\cal C}_{\R}[0,\infty)$-valued strong Markov
process.  
\end{enumerate}
\end{theorem}

The characterization in terms of the stochastic age equation 
is established in Section \ref{sec-sae} and  
the proof of the strong Markov property is given in Section 
\ref{subs-pfmain2}.   It is more natural to expect  the process 
$\{(\hatx, \hmeas_t), \widehat{{\cal F}}_t, t \geq 0\}$   
to be strong Markov  with state $\R \times \H_{-2}$. 
   However, due to technical 
reasons (see Remark \ref{rem-markov} for a more detailed 
explanation) 
it was necessary to append an additional 
component to obtain  a Markov process. 
We expect that this additional component 
 should not pose too much of a problem in applications 
of this result.

\begin{remark}
\label{cond-verify} 
{\em  
Elementary calculations show that 
 the  boundedness assumptions 
imposed on $g/(1-G)$ and $g^\prime/(1-G)$ in Theorem \ref{th-main2} (which, in particular, 
imply Assumption \ref{as-holder}) are 
satisfied by many continuous service distributions of interest 
(with finite mean, normalized to have mean one) 
including the families of  
lognormal, Weibull, logistic and phase type distributions, 
the Gamma$(a,b)$ 
distribution with shape parameter $a - 1$ or $a \geq 2$
(and corresponding rate parameter $b = 1/a$ 
to produce a mean one distribution), 
 the Pareto distribution  with shape parameter $a > 1$ (and corresponding 
scale parameter $b = (a - 1)/a$ so the mean equals one), 
the inverted Beta$(a,b)$ distribution when $a > 2$ 
(and $b = a+1$ so that the mean is equal to one). Note that 
 the mean one exponential distribution 
is clearly also included as a special case of the Weibull and Gamma 
distributions. 
}
\end{remark}

\subsection{Corollaries of the Main Results}
\label{subs-cor}

As an important corollary of Theorem \ref{th-main1}, when the fluid 
limit is critical,
the limiting (scaled and centered) total number in system $\hatx$ can be 
characterized as an It\^{o} diffusion. 

\begin{cor}
\label{cor-ito}
Suppose Assumption \ref{as-flinit}, Assumption \ref{as-hate}(a) with $\flam =
1$ and Assumption \ref{as-diffinit}' are satisfied and   $h$ is bounded and
absolutely continuous.   
If $\overline{x}_0 = 1$ and $\fmeas_0(dx)=(1-G(x))dx$, the equilibrium age measure in the 
critical fluid limit, then $\hatx$ satisfies the following It\^{o} diffusion
equation: 
 \be
\label{eq-sdeito}
  d\hatx (t) =  \widehat{x}_0 + \sigma B(t) - \hatm_{\f1} (t) - \beta t - \int_0^t
\lan h, \hmeas_s \ran  \, ds, 
 \ee 
where $\hatm_{\f1}$ is a Brownian motion independent of $B$. 
\end{cor}

The asymptotic independence of $\chatm_{\f1}$ and $B$ in the 
last corollary follows from Proposition
\ref{prop-martconv}. 
In the particular case of an exponential service 
disribution, 
 this allows us to  immediately recover the form of the limiting
diffusion obtained in the  seminal
paper of Halfin and Whitt \cite{HalWhi81}.

\begin{cor}
\label{cor-martconv}
Suppose $G$ is the exponential distribution with parameter $1$.
Suppose Assumption \ref{as-flinit} holds with $\fmeas_0(dx) = (1-G(x)) dx$ 
and $\overline{x}_0 = 1$, 
Assumption \ref{as-hate}(a) holds with $\flam =
1$ and Assumption \ref{as-diffinit}' is satisfied.
Then $\hatx$ is 
 the unique strong solution to the stochastic
differential equation
\be
\label{eq-sdehw}
  \hatx (t) = \widehat{x}_0 + \sqrt{1 + \sigma^2} W(t) - \beta t - \int_0^t \hatx^- (s)
  \, ds
\ee
where $W$ is a standard Brownian motion.
\end{cor}
\begin{proof}
When $G$ is the exponential distribution, $h \equiv \f1$  and therefore 
\[ \int_0^t \lan h, \hmeas_s \ran \, ds = \int_0^t \lan \f1, \hmeas_s \ran \, ds = \int_0^t \left(
\hatx (s) \wedge 0 \right) \, ds,
\]
where the last equality uses the relations (\ref{cmse-eq}),
(\ref{ref-nonidling}) and the fact that $\fx \equiv \f1$.   
By the independence of $B$ and $\chatm_{\f1}$, 
 $\sigma B - \hatm_{\f1}$ has the same distribution as $\sqrt{1 + \sigma^2}W$, where $W$ is a
standard Brownian motion. 
Substituting this back into (\ref{eq-sdeito}), which is applicable 
since the hazard rate $h$ of the exponential distribution is trivially bounded
and absolutely continuous, 
we obtain (\ref{eq-sdehw}). 
 The
Lipschitz continuity of the drift coefficient $x \mapsto -\beta -
x^-$ guarantees that the   stochastic differential equation
(\ref{eq-sdehw}) has a unique strong solution. 
\end{proof}

\begin{remark}
\label{rem-fsde} {\bf (Insensitivity Result)}
{\em  As a comparison of (\ref{eq-sdeito}) and (\ref{eq-sdehw}) reveals,
 under the same assumptions on the
arrival process as in Corollary \ref{cor-martconv}, the dynamical equation for
$\hatx$ for general service distributions is remarkably close to the
exponential case.  Indeed, the ``diffusion'' coefficient is
the same in both cases (and is equal to $\sqrt{1 + \sigma^2}$), 
but the difference is that
in the case of general service distributions,  the drift is
an $\{\widehat{\cal F}_t\}$-adapted process that could in general depend on the past, and
not just on $\hatx_t$, so that the resulting process is no longer Markovian.  
}
\end{remark}

\beginsec

\section{Representation of the System Dynamics}
\label{subs-prelim}

In Section \ref{subs-character}
 we present a succinct characterization of the
dynamics of the centered state process and then use that in
Section \ref{subs-represent}
to derive  an alternative representation for the centered
age process.

\subsection{A Succinct Characterization of the Dynamics}
\label{subs-character}

We first recall the description of
the dynamics of the $N$-server system that was
established by Kaspi and Ramanan  \cite{KasRam07}.

\begin{prop}
\label{prop-prelimit}
 The process $(\xn,
\measn)$ almost surely satisfies the following coupled set of equations: for
$\newf \in {\cal C}_b^{1,1}([0,\endsup) \times \R_+)$ and  $t \in [0,\infty)$,
\begin{eqnarray}
\label{eqn-prelimit1}
  \quad \quad \left\lan \newft, \measn_{t} \right\ran
& = &  \left\lan \newf(\cdot, 0), \measn_{0} \right\ran   +  \ds \int_{0}^t \left\lan \dxnewfs + \dtnewfs, \measn_s \right\ran \, ds  \\
\nonumber
& &
 \qquad
 \ds - \int_0^t \lan \newfs h(\cdot), \measn_s \ran \, ds
 -\cmn_t (\newf) \\
\nonumber
& & \qquad + \int_{[0,t]}
\newf(0,s) \, d\kn (s), \\
\label{eqn-prelimit2}
\quad \quad \quad  \xn (t) & = &  \xn(0) + \en (t) - \int_0^t \lan h, \measn_s
\ran \, ds - \cmn_{t} (\f1) 
\end{eqnarray}
 and
\begin{eqnarray}
\label{comp-prelimit}
N - \left\lan 1, \measn_t \right\ran  & = &  [N - \xn (t) ]^+,
\end{eqnarray}
where $\kn$ satisfies
\be
\label{eqn-prelimit3}
\ba{rcl}
  \kn & = & \ds \lan \f1, \measn \ran - \lan \f1, \measn_0 \ran
+ \int_0^\cdot \lan h, \measn_s
 \ran \, ds+ \cmn (\f1) \\
& = & \ds \hatxn(0) + \en - \xn + \lan \f1, \measn \ran - \lan \f1,
\measn_0\ran.
\ea
\ee
\end{prop}
\begin{proof}
This is essentially a direct consequence of Theorem 5.1 of Kaspi and 
Ramanan \cite{KasRam07}.
Indeed,  by subtracting and adding $\dcompn_{\newf}$ on the right-hand side
of equations (5.4) and (5.5) in
\cite{KasRam07}, and then using (\ref{def-martn}) above and the fact that
$\martn_{\newf}$ is indistinguishable from $\cmn (\newf)$ (see Remark 
\ref{rem-mncadlag}) 
one obtains (\ref{eqn-prelimit1}) and (\ref{eqn-prelimit2}), respectively.
Equation (\ref{comp-prelimit}) coincides with equation (5.6)
in \cite{KasRam07}.  Finally, the first equality in (\ref{eqn-prelimit3}) 
follows from (2.6) of \cite{KasRam07} and (\ref{def-martn}) above, whereas 
the second equality in (\ref{eqn-prelimit3})
follows from  (\ref{eqn-prelimit2}).
\end{proof}

Combining the characterizations of the
$N$-server system and the fluid limit given in
Proposition \ref{prop-prelimit} and Theorem \ref{th-flimit}, respectively, we
obtain the following representation for the
centered diffusion-scaled state  dynamics. 
In what follows, recall the 
 centered, diffusion-scaled quantities defined
in (\ref{diff-scaling3}) and (\ref{def-chatmn}).

\begin{prop}
\label{prop-cenprelimit} For each $N \in \N$,
the process $(\hxn, \hmeasn)$ almost surely satisfies the
following coupled set of equations:
for every $\newf\in\newocdcpm$ and $t \in [0,\infty)$,
\begin{eqnarray}
 \label{eq-dprelimit1}
 \qquad \left\lan \newf(\cdot,t), \hmeasn_{t}
\right\ran &=& \left\lan
\newf(\cdot, 0), \hmeasn_{0} \right\ran   +  \ds \int_{0}^t
\left\lan \dxnewfs + \dtnewfs, \hmeasn_s \right\ran \, ds  \\
& &
\nonumber \qquad \ds - \int_0^t \left\lan  \newf (\cdot,s) h(\cdot), \hmeasn_s \right\ran \, ds - \chatmn_t (\newf)\\
&&
\nonumber \qquad \ds + \int_{[0,t]} \newf(0,s) \, d\hatkn (s), \\
\label{eq-dprelimit2} \hxn (t)  &  = &  \hxn(0) + \haten (t)
-\int_0^t \lan h,\hmeasn_s\ran ds - \chatmn_t (\f1),
\end{eqnarray}
and
\be
\label{eq-dnonidling}
  \lan \f1, \hmeasn_t \ran = \left\{
\ba{ll}
 \hatxn(t) \wedge \sqrt{N} (1 - \fx(t)) & \mbox{ if } \fx(t) < 1, \\
\hatxn(t) \wedge 0 & \mbox{ if } \fx(t) = 1,  \\
 \sqrt{N} (\fxn(t) - 1)  \wedge 0 & \mbox{ if } \fx(t) > 1,
\ea
\right.
\ee
where
\be
\label{eq-dprelimit3}
\begin{array}{rcl}
\ds \hatkn & = & \ds \lan \f1, \hmeasn \ran - \lan \f1, \hmeas_0 \ran +
\int_0^\cdot \lan h, \measn_s \ran \, ds + \chatmn (\f1)\\
& = & \ds \hinitxn + \haten - \hatxn
+ \lan \f1,\hatmeasn\ran-\lan\f1,\hatmeasn_0\ran.
\end{array}
\ee
\end{prop}
\begin{proof}
Equation (\ref{eq-dprelimit1}) is obtained by
dividing each side of the equation (\ref{eqn-prelimit1}) by $N$,
subtracting the corresponding side of (\ref{eq-ftmeas}) from it and
multiplying the resulting quantities by $\sqrt{N}$.
In an exactly analogous fashion, equation
(\ref{eq-dprelimit2}) can be derived from equations
 (\ref{eqn-prelimit2}) and (\ref{eq-fx}), and
equation  (\ref{eq-dprelimit3})
can be obtained from equations (\ref{eqn-prelimit3}), (\ref{eq-fx}) and
 (\ref{eq-fk}).
It only remains to justify the relation in (\ref{eq-dnonidling}).
Dividing (\ref{comp-prelimit}) by $N$, subtracting it from
(\ref{eq-fnonidling}) and multiplying this difference by
$\sqrt{N}$, we obtain
\be
\label{eq-hatidle}
\lan \f1, \hatmeasn_t \ran = \sqrt{N} \left(  [1- \fx(t)]^+ - [1 - \fxn(t)]^+
\right).
\ee
If $\fx(t) < 1$ then $[1 - \fx(t)]^+ = (1-\fx(t))$ and so the right-hand side
above equals
\[
\left\{
\ba{rl}
\sqrt{N} (1 - \fx(t) - (1 - \fxn(t)) = \hatxn(t)  & \mbox{ if }  \fxn(t) < 1, \\
\sqrt{N} (1 - \fx(t))  & \mbox{ if }
\fxn(t) \geq 1, 
\ea
\right.
\]
which can be expressed as $\hatxn(t) \wedge \sqrt{N} (1 - \fx(t))$.
On the other hand, if $\fx(t) = 1$ then $[1- \fx(t)]^+ = 0$ and the right-hand side of
(\ref{eq-hatidle}) equals
\[ -\sqrt{N}[1 - \fxn(t)]^+ = -\sqrt{N} [\fx(t) - \fxn(t)]^+ = \hatxn(t) \wedge
0.
\]
Lastly, if $\fx(t) > 1$ then $[1 - \fx(t)]^+ = 0$ and so the
right-hand side of (\ref{eq-hatidle}) reduces to
$\sqrt{N} (\fxn(t) - 1) \wedge 0$, and (\ref{eq-dnonidling}) follows.
\end{proof}

\begin{remark}
\label{rem-comp}
{\em We describe conditions under which,
for large $N$, the non-idling condition
(\ref{eq-dnonidling}) can be further
simplified and written purely in terms of $\lan \f1, \hmeasn\ran$ and $\hxn$.
Let $\Omega^*$ be the set of full $\P$-measure on which the fluid limit
theorem (Theorem \ref{th-flimit}) holds. Fix $\omega
\in \Omega^*$ (and henceforth suppress the dependence on $\omega$) and
let $t \in [0,\infty)$ be a continuity point of the fluid limit.
If $\fx(t) < 1$ then  by  Theorem \ref{th-flimit}
 there exists $N_0 = N_0(\omega, t) < \infty$ such
that for all $N \geq N_0$, $\fxn (t) < 1$ and so
\[ \hatxn(t) = \sqrt{N} \left(\fxn(t) - \fx(t) \right) \leq \sqrt{N} (1 - \fx(t)).
\]
On the other hand, if  $\fx(t)
> 1$ then there exists $N_0 = N_0(\omega,t) < \infty$ such that for all
$N \geq N_0$, $\fxn(t)
> 1$ and hence, 
\[ \sqrt{N} (\fxn(t) - 1) \geq 0. \]
Therefore, for any $t \in [0,\infty)$ there
exists $N_0 = N_0(\omega, t) < \infty$ such that
for all $N \geq N_0$, 
\be
\label{eq-dnonidling2}
  \lan \f1, \hmeasn_t \ran = \left\{
\ba{ll}
 \hatxn(t)  & \mbox{ if } \fx(t) < 1, \\
\hatxn(t) \wedge 0 & \mbox{ if } \fx(t) = 1,  \\
  0 & \mbox{ if } \fx(t) > 1.
\ea
\right.
\ee
Now, suppose the fluid limit $\fx$ is continuous and for some $T < \infty$,
 the fluid is subcritical on $[0,T]$ in the sense of Definition
\ref{rem-critical}. Then $N_0$ can clearly be chosen uniformly
in $t \in [0,T]$ and so there exists $N_0 = N_0(\omega, T) < \infty$ such
that for all $N \geq N_0(\omega, T)$,
\[ \lan \f1, \hmeasn_t \ran = \hatxn (t),  \quad t \in [0,T]. \]
An analogous statement holds for the supercritical case and
(trivially) for the critical case.
}
\end{remark}

\subsection{A Useful Representation}
\label{subs-represent}

Equations (\ref{eqn-prelimit1}) and (\ref{eq-dprelimit1})
for the age and (scaled) centered age processes, 
 respectively, 
in the $N$-server system have a  form that is analogous 
to the deterministic integral equation (\ref{eq-ftmeas}) that
describes the dynamics of the age
process in the fluid limit, except that they contain an additional
stochastic integral term.  
 Indeed, all three equations fall 
under the framework of the so-called abstract age equation 
introduced in Definition 4.9 of Kaspi and Ramanan \cite{KasRam07}. 
 Representations for solutions to the abstract 
age equation were obtained in Proposition 4.16 of \cite{KasRam07}.  
In Corollary \ref{cor-sae1} below, 
this result is applied to obtain  explicit representations 
for the age and centered age processes in the $N$-server system. 
Not surprisingly, these representations are similar to the 
corresponding representation (\ref{final-fleqs}) for solutions 
to the fluid  age equation, except that they contain an 
additional stochastic integral term.


We now  state the representation 
result, which is easily deduced from 
Proposition 4.16 of Kaspi and Ramanan \cite{KasRam07}; 
the details of the proof are deferred to Appendix \ref{rep-prelim}.  
For conciseness of notation, for $N \in \N$  and continuous $f$, 
  we  define 
\be
\label{def-hcalkn}
  \hcalkn_t (f) \doteq \int_{[0,t]} (1-G(t-s)) f(t-s) \, d\hatkn(s), \qquad t \in
[0,\infty). 
\ee
By applying integration by parts to  
 the right-hand side of (\ref{def-hcalkn}) and using the fact that 
$\hatkn$ has jumps at at most a countable number of points, we see that 
 for absolutely continuous $f$,   
\be
\label{eq-hcalkn}
 \hcalkn_t (f)  =   f(0) \hatkn(t) + \int_0^t \hatkn(s) \trans_f(t-s) 
\, ds, 
\ee 
where 
\be
\label{def-psif}
\trans_f \doteq (f(1-G))^\prime =  f^\prime (1-G) - fg.  
\ee 
Also, recall the definition of the process $\ops^{\hmeasn_0}$ 
given in (\ref{def-ops}).

\begin{prop}
\label{cor-sae1}
For every $N \in \N$, 
 $f \in {\cal C}_b[0,\endsup)$ and $t \in [0,\infty)$,
 \begin{eqnarray}
\label{rep-measn}
\lan f, \measn_t \ran & = & \ops_t^{\measn_0}(f) - \chn_t \left(f \right) 
+  {\cal K}^{(N)}_t(f), 
\end{eqnarray}
and, likewise, 
\begin{eqnarray}
\label{rep-hatmeasn}
\lan f, \hmeasn_t \ran & = &
\ops_t^{\hmeasn_0}(f)  - \hathn_t(f)  + \hcalkn_t(f). 
\end{eqnarray}
\end{prop}

\begin{remark}
\label{rem-hdn}
{\em 
For subsequent use, we make the simple observation that on substituting 
$\newf = \f1$ in (\ref{eq-dprelimit1}) and substracting it 
from (\ref{rep-hatmeasn}), with $f = \f1$, then rearranging terms and using 
(\ref{eq-hcalkn}) and the fact that $\trans_{\f1} = (1-G)^\prime  = -g$  by (\ref{def-psif}),   
we obtain  for every $N \in \N$ and $t > 0$, 
\be
\label{eq-hdn}
\begin{array}{rcl}\ds
\int_0^t \lan h, \hmeasn_s \ran \, ds & = &  
\lan \f1, \hmeasn_0 \ran - \ops^{\hmeasn_0}_t(\f1) - \chatmn_t(\f1) + 
\hathn_t(\f1) \\
& & \ds \quad  + \int_0^t g(t-s) \, \hatkn(s-)\, ds. 
\end{array}
\ee
}
\end{remark}

\beginsec

\section{Continuity Properties} 
\label{subs-cont}

In Section \ref{subs-auxcont} we establish continuity of the mapping that 
takes $\hkn$ to $\hcalkn$ and  in Section \ref{subsect-cont} we establish 
continuity of the centered many-server map $\Lambda$, which in particular 
shows that both $\hkn$ and $\hxn$ are obtained as continuous mappings of 
the initial data and $\hathn$.

\subsection{Continuity of an Auxiliary Map}
\label{subs-auxcont}

Here, we establish the continuity of a mapping related to the convolution integral 
$\hcalkn$ defined in (\ref{def-hcalkn}). 
Given any (deterministic) \cad function $K$, for   absolutely continuous
functions $f$ we define 
\be
\label{eq-calk}
 \calk_t (f)  \doteq  f(0) K(t) + \int_0^t K(u) \trans_f (t-u) \, du, 
\ee
where $\trans_f = (f(1-G))'$, as defined in (\ref{def-psif}). 
Since $K$, $g$ and  $f^\prime$ are all locally integrable, for each $t > 0$, 
$\calk_t$ is a well defined linear functional on the space of absolutely 
continuous functions.   
Moreover, from elementary properties of convolutions, it is clear that for any 
absolutely continuous $f$, 
if $K$ is c\`{a}dl\`{a}g (respectively, continuous) then so is $\calk(f)$. 
In Lemma \ref{lem-calk} below, we show that 
$\calk$ is in fact a 
\cad $\H_{-2}$-valued function and the mapping from 
$K$ to $\calk$, which we denote by $\Gamma$,  is continuous. 
Note that by (\ref{eq-hcalkn}) $\hcalkn = \Gamma (\hatkn)$ 
for $N \in \N$. The continuity of $\Gamma$
is used in the proof of Theorem \ref{th-fclt} to  establish 
convergence of $\hcalkn$ to the analogous limit quantity $\hcalk$,  defined for 
absolutely continuous $f$, by  
\be
\label{def-hcalk}
 \hcalk_t (f)  \doteq   f(0) \hatk(t) + \int_0^t \hatk(s) \trans_f(t-s) 
\, ds, \qquad t \in [0,\infty). 
\ee 
The third property in Lemma \ref{lem-calk} below is used in the proof of the 
strong Markov property in Section \ref{subs-pfmain2}.

\begin{lemma}
\label{lem-calk}
Suppose  $\Gamma$ is the map that takes $K$ to the
linear functional ${\cal K}$ defined  in (\ref{eq-calk}).  
If $g$ is continuous the following three properties are satisfied:  
 \begin{enumerate}
\item If $K \in {\cal D}_{\R}[0,\infty)$ (respectively, ${\cal
    C}_{\R}[0,\infty)$) then 
 ${\cal K} \in {\cal D}_{\H_{-2}}[0,\infty)$ (respectively, ${\cal
   C}_{\H_{-2}}[0,\infty)$). 
\item 
$\Gamma$ is a continuous map from ${\cal D}_{\R}[0,\infty)$ to 
${\cal D}_{\H_{-2}}[0,\infty)$, when both domain and range are 
either both equipped with the topology of uniform convergence on compact sets 
or both equipped with the Skorokhod topology.  Likewise, 
the map from ${\cal D}_{\R}[0,\infty)$ to itself
that takes $K \mapsto \calk(\f1)$ is also continuous with respect 
to the Skorokhod topology on ${\cal D}_{\R}[0,\infty)$. 
\item 
If $K \in {\cal C}_{\R}[0,\infty)$ then, for any $t \in [0,\infty)$,  
the real-valued function $u \mapsto {\cal K}_t (\Phi_u \f1)$ on $[0,\infty)$ is continuous and the 
map from ${\cal C}_{\R}[0,\infty)$ to itself 
that takes $K$ to this function is continuous (with respect to the 
uniform topology). 
\end{enumerate}
\end{lemma}
\begin{proof}  
Let $g$ be continuous.  We first derive a general 
inequality (see (\ref{use-5}) below) that is then used to prove both properties 1 and 2. 
Fix $K, K^{(n)} \in {\cal D}_{\R}[0,\infty)$,
 $T < \infty$ and $t, \tau^{(n)}(t) \in [0,T]$ with 
$\delta^{(n)}(t) \doteq |t-\tau^{(n)}(t)|$, $n \in \N$.  Also, let 
$\calk \doteq \Gamma (K)$ and $\calk^{(n)} \doteq \Gamma (K^{(n)})$, $n \in \N$.  
For $f \in \H_2$, we can write
\be
\label{use-4} \calk^{(n)}_{\tau^{(n)}(t)} (f) - \calk_t(f) = f(0) \left( K^{(n)}
  (\tau^{(n)}(t))  - K(t)\right) + \sum_{i=1}^3 \Delta_i^{(n)}(t), 
\ee
where 
\begin{eqnarray*}
\Delta_1^{(n)} (t) & \doteq & \int_0^{t \wedge \tau^{(n)}(t)} K(u) \left( \trans_f
  (t-u) - \trans_f (\tau_n(t) - u) \right) \, du, \\
\Delta_2^{(n)} (t) & \doteq & \int_0^{t \wedge \tau^{(n)}(t)} \left( K(u) - K^{(n)}
  (u) \right) \trans_f (\tau^{(n)}(t) - u) \, du, \\ 
\Delta_3^{(n)} (t) & \doteq & \int_{t \wedge \tau^{(n)}(t)}^t K(u) \trans_f (t-u) \,
du + \int_{t \wedge \tau^{(n)}(t)}^{\tau^{(n)}(t)} K^{(n)}(u) \trans_f
(\tau^{(n)}(t) - u) \, du. 
\end{eqnarray*}
To bound the above terms, first note that by the inequality $(1-G) \leq 1$, repeated application of
 the Cauchy-Schwarz inequality and Tonelli's theorem we obtain 
\[
\begin{array}{l} 
\ds \int_0^s \left| f^\prime(t-u) (1-G(t-u)) - f^\prime(s-u) (1-G(s-u)) \right| \,
du \\
\qquad \ds \leq \int_0^s \left|f^\prime(t-u)\right| \left(G(t-u) - G(s-u)\right) \,
  du 
+ \int_0^s \left| f^\prime(t-u) - f^\prime(s-u) \right| \, du \\
\qquad \ds \leq w_G(|t-s|) T^{1/2} \nrm{f^\prime}_{\H_0} + 
\int_0^s \left(\int_s^t \left| f^{\prime \prime} (w-u) \right| \, dw \right)
\, du \\
\qquad \ds \leq w_G(|t-s|) T^{1/2} \nrm{f^\prime}_{\H_0} + 
T |t-s|^{1/2} \nrm{f^{\prime \prime}}_{\H_0},  
\end{array}
\]
where $w_G$ is the modulus of continuity of $G$ as defined 
in (\ref{def-modcon}). 
Similarly, using the continuity of $g$ and, as usual, denoting 
its modulus of continuity by $w_g$, we have 
\[
\begin{array}{l}
 \ds \int_0^s \left| f(t-u) g(t-u)  - f(s-u) g(s-u)\right| \, du \\
\ds \qquad \leq \int_0^s |f(t-u)||g(t-u) - g(s-u))| \, du 
+ \int_0^s g(s-u)|f(t-u) - f(s-u)| \, du \\
\ds \qquad \leq T^{1/2}\nrm{f}_{\H_0} w_g(|t-s|) + 
\int_0^s g(s-u) \int_s^t |f^\prime(w-u)|\, dw \, du \\
\ds \qquad \leq T^{1/2}\nrm{f}_{\H_0} w_g(|t-s|) 
+ |t-s|^{1/2} \nrm{f^\prime}_{\H_0} 
\leq (T^{1/2} w_g(|t-s|) + |t-s|^{1/2}) \nrm{f}_{\H_1}. 
\end{array}
\]
Recalling that $\trans_f = (f(1-G))'$, 
 the last two inequalities show that 
\be
\label{use-1}
  \int_0^s |\trans_f(t-u) - \trans_f(s-u)| \,
  du  \leq c_1(T,|t-s|) 
 \nrm{f}_{\H_{2}}, 
\ee
where $c_1(T,\delta) \doteq \left( T^{1/2} (w_G(\delta) + w_g(\delta)) + (T+1)
 \delta^{1/2} \right)$ satisfies $\lim_{\delta \rightarrow 0} c_1(T,\delta) =
0$. 
On the other hand, another application of the 
Cauchy-Schwarz inequality and 
the norm inequality (\ref{norm-ineq}) shows that 
\be
\label{calk-ineq1}
 \begin{array}{rcl}
\ds \int_s^t \left| \trans_f(t-u) \right| \, du 
 & \leq &  \ds \int_s^t \left| f^\prime(u) \right| \, du + 
\nrm{f}_\infty (G(t) - G(s))\\
&  \leq & \ds |t-s|^{1/2} \nrm{f^\prime}_{\H_0} + \nrm{f}_{\infty} (G(t) -
G(s))  \\
&  \leq & \ds 3(|t-s|^{1/2} + w_G(|t-s|))\nrm{f}_{\H_1} \leq c_2(T)
\nrm{f}_{\H_1}, 
\end{array}
\ee
and 
\be
\label{use-3} 
\nrm{\trans_f}_{T} \leq \nrm{f^\prime}_{\infty} + \nrm{f}_{\infty} \nrm{g}_{T}
\leq c_2(T) \nrm{f}_{\H_2}
\ee
for an appropriate finite constant $c_2(T) < \infty$ that depends only on 
$G$ and $T$. 
Substituting (\ref{use-1})--(\ref{use-3}) into (\ref{use-4}), we obtain 
\begin{eqnarray}
\label{use-5}
 \dfrac{\left| \calk_t(f) - \calk_{\tau_n(t)}^{(n)}(f)\right|}{\nrm{f}_{\H_2}} 
& \leq & 
\left|K(t) - K^{(n)}(\tau^{(n)}(t))\right| + c_1(T,\delta^{(n)}(t)) \nrm{K}_T \\
& & \quad \nonumber 
+ 
c_2(T) \int_0^{t \wedge \tau^{(n)}(t)} \left( K(u) - K^{(n)}(u) \right) \,
du\\
& &\quad \nonumber 
 + 2 \delta^{(n)}(t) c_2(T) \left( \nrm{K}_T \vee \nrm{K^{(n)}}_T \right). 
\end{eqnarray}

Now, suppose $K^{(n)} = K$ so that $\calk^{(n)} = \calk$,  $n \in \N$, 
and consider $t < T$ and any 
sequence  of points $\tau^{(n)}(t) \in [0,T]$, $n \in \N$, such that
$\tau^{(n)}(t) \downarrow t$ as $n \ra \infty$. 
Then the third term on the right-hand side of (\ref{use-5}) vanishes, 
the first term converges to zero because $K \in {\cal D}_{\R}[0,\infty)$ and 
the second and fourth terms converge because $\nrm{K}_{T} < \infty$ and 
$\delta^{(n)}(t) \ra 0$.  
This shows that $\nrm{\calk_t - \calk_{\tau^{(n)}(t)}}_{\H_{-2}} \ra 0$ and 
hence, $\calk \in {\cal D}_{\H_{-2}}[0,\infty)$.  The same argument also
shows that $\calk$ is continuous if $K$ is. This proves the first 
property.

Next, suppose that $K^{(n)}$, $n \in \N$, is a sequence that converges to $K$ in the 
Skorokhod topology.  By the definition of the Skorokhod topology 
(see, for example, Chapter 3 of
\cite{bilbook}) 
there  exists a sequence of strictly increasing maps
$\tau^{(n)}$, $n \in \N$, that map $[0,T]$ onto $[0,T]$ and satisfy
 $\nrm{\delta^{(n)}}_T \doteq 
\sup_{t \in [0,T]} |t - \tau^{(n)}(t)| \ra 0$ and $\nrm{K^{(n)} \circ \tau^{(n)} -
  K}_{T} \ra 0$ as $n \ra \infty$.    Moreover, $\sup_{n} \nrm{K^{(n)}}_T < \infty$ and 
$K^{(n)}(u) \ra K(u)$ for a.e.\ $u \in [0,T]$. 
Taking first the supremum over $t \in [0,T]$ and then 
limits as $n \ra \infty$ in (\ref{use-5}), the above properties show
that the right-hand side goes to zero (where the 
dominated convergence theorem  is used to argue that the third term
vanishes).
In turn, this implies that $\sup_{t \in [0,T]}\nrm{\calk_t -
  \calk^{(n)}_{\tau^{(n)}(t)}}_{\H_{-2}} \ra 0$, thus establishing convergence
of $\calk^{(n)}$ to $\calk$ in the Skorokhod topology 
on ${\cal D}_{\H_{-2}}[0,\infty)$.  This establishes continuity of the map 
$\Gamma$ in the Skorokhod topology.  Continuity with respect to  
the uniform topology can be proved by setting $\tau^{(n)}(t) = t$, $n \in \N$, 
in the argument above. 
The continuity of the map that takes $K$ to $\calk(\f1)$ can be 
established in an analogous fashion. The proof is left to the reader. 

To prove the last property, fix $K \in {\cal C}_{\R}[0,\infty)$ and $t \in [0,\infty)$.  
  For $u \geq 0$, the function $\Phi_u \f1$ is absolutely
 continuous and $\trans_{\Phi_u \f1} = (1-G(\cdot + u))^\prime = - g(\cdot +
 u)$. Setting $f = \Phi_u \f1$
 in  (\ref{eq-calk}) yields 
\[ \calk_t (\Phi_u \f1) = (1-G(u)) K(t) - \int_0^t K(s) g(t-s+u) \,
ds. 
\]
The continuity of $G$ and $K$ and the bounded convergence theorem then show
that $u \mapsto \calk_t (\Phi_u \f1)$ lies in ${\cal C}_{\R}[0,\infty)$.
On the other hand, given $K^{(i)} \in {\cal C}_{\R}[0,\infty)$ for $i = 1, 2$
and the corresponding $\calk^{(i)}$, 
\begin{eqnarray*}
 \sup_{u \in [0,T]} |\calk^{(1)}_t (\Phi_u \f1) - \calk^{(2)}_t (\Phi_u \f1)| 
\leq 
\nrm{K^{(1)} - K^{(2)}}_{T}\left(1 + \int_0^t g(t-s+u) \, ds \right), 
\end{eqnarray*}
from which it is clear that the map from ${\cal C}_{\R}[0,\infty)$ to itself  
that takes $K$ to the function 
$u \mapsto \calk_t (\Phi_u \f1)$ is continuous.
\end{proof}

\subsection{Continuity of the Centered Many-Server Map}
\label{subsect-cont}

Recall the centered many-server map $\Lambda$ introduced 
in Definition \ref{def-mseq}. 
First, in Lemma \ref{lem-lambdarep}, 
we establish the representation for $(\hatxn, \hatkn, \hmeasn(\f1))$  in 
terms of the map $\Lambda$ specified in  (\ref{cmse-eq}),   
and  then in Proposition \ref{prop-cont} and Lemma \ref{lem-lambdameas} we  
 establish certain continuity and measurability properties 
of the map $\Lambda$.

\begin{lemma}
\label{lem-lambdarep}
If the fluid limit $\fx$ of the total number in system 
is either subcritical, critical or supercritical and Assumption
\ref{as-holder} holds, 
there exists $\Omega^* \in {\cal F}$ with $\P(\Omega^*) = 1$ such that
for every $\omega \in \Omega^*$, there exists
$N^*(\omega) < \infty$
such that for all $N \geq N^*(\omega)$, 
\[
\ds (\hatkn, \hatxn, \lan \f1, \hmeasn\ran)(\omega)
 =
\lc ((\haten, \hinitxn, \ops^{\hmeasn_0}
(\f1) -  \hathn(\f1))(\omega).
\]
\end{lemma}
\begin{proof} Fix $N \in \N$. 
By the basic definition of the model, $\haten$ is c\`{a}dl\`{a}g and 
$\hmeasn$ takes values in ${\cal D}_{{\cal M}_F[0,\endsup)}[0,\infty)$ and hence, in 
${\cal D}_{\H_{-2}}[0,\infty)$.  Due to Assumption \ref{as-holder} and Remark
\ref{rem-holder2}, it follows that $\ops^{\hmeasn_0}(\f1)$ is continuous.  
Moreover, as explained in the discussion right after definition
(\ref{eq-calk}),
 $\hcalk(\f1)$ is also continuous.  
By the representation (\ref{rep-hatmeasn}), it then follows that 
$\hathn(\f1)$ is also c\`{a}dl\`{a}g.   Thus, 
for almost surely every  $\omega \in \Omega$, 
 $(\haten, \hinitxn, \ops^{\hmeasn_0} (\f1)
-\hathn(\f1))(\omega) \in \dataspace$.
The result then follows on comparing the three many-server equations
(\ref{ref-meas}), (\ref{ref-sm}) and (\ref{ref-nonidling}) with
the second equation in (\ref{eq-dprelimit3}),
equation (\ref{eq-dnonidling2}) of Remark \ref{rem-comp} and  equation
(\ref{rep-hatmeasn}). 
\end{proof}

We now establish continuity and measurability properties of the 
 mapping $\lc$. 
Recall that $U$ denotes the renewal
function associated with the distribution $G$.

\begin{prop}
\label{prop-cont}
Fix $\genflx \in {\cal D}_{\R_+}[0,\infty)$.  For $i =1,2$, suppose
$(\gene^i, \geninitx_0^i, \addfn^i)$ $\in \dataspace$ and
$(K^i, \refx^i, v^i) \in \lc(\gene^i, \geninitx_0^i, \addfn^i)$.
Then, for any $T \in [0,\infty)$,
\be
\label{ineq-cont}
 \nrm{\del K}_T \vee \nrm{\del X}_T \vee \nrm{\del v^i}_T \leq  3(1 + U(T)) \ve_T,
\ee
where  $\del S \doteq S^2 - S^1$, $\nrm{f}_T \doteq \sup_{s \in [0,T]} |f(s)|$
and
\be
\label{def-vet}
\ve_T \doteq \left( \nrm{\del \refe}_T \vee |\del x_0| \vee
\nrm{\nabla \addfn}_T \right).
 \ee
Hence, $\lc$ is continuous with respect to the uniform topology  
and is single-valued on its domain.  
\end{prop}
\begin{proof}
Fix $T < \infty$.  We first show that $\nrm{\del K}_T \leq 2 \ve_T(1+U(T))$.
For any  $t \in [0,T]$, we consider two cases. \\
{\em Case 1:} Either $\fx(t) < 1$, or both $\fx(t) = 1$ and $\genx^1(t) \leq 0$.  \\
We claim that in this case we always have
\be
\label{ref-claim1}
  \del v(t)  - \del \genx(t) \leq 0.
\ee
Indeed, (\ref{ref-nonidling}) shows that
if $\fx(t) < 1$ then $v^i(t) = \genx^i(t)$ for
$i = 1, 2$ and so the left-hand side above is identically zero.
On the other hand, if
$\fx(t) = 1$ and $\genx^1(t) \leq 0$ then $(\genx^1(t))^+ = 0$, and so
 (\ref{ref-nonidling}), combined with  the elementary identity
$x\wedge 0 - x = -x^+$,  implies
\[ \del v(t) - \del \genx(t) = (\genx^1(t))^+ -
(\genx^2(t))^+ =  - (\genx^2(t))^+\leq 0,
\]
and so (\ref{ref-claim1}) holds.

In turn, combining (\ref{ref-claim1}) with the fact that each solution 
satisfies equation  (\ref{ref-sm}), we then conclude that
\be
\label{kineq-1}  \del \genk (t) = \del \geninitx_0 + \del \gene (t) - \del
\genx (t) +
\del v(t) \leq \del \geninitx_0 + \del \gene (t) \leq 2
\ve_T.
\ee
{\em Case 2:} Either $\fx(t) > 1$, or both $\fx(t) = 1$ and $\genx^1(t) > 0$. \\
First, we claim that in this case,
\be
\label{ref-claim2}
 \del v(t) = v^2(t)  -  v^1(t)   \leq 0.
\ee
If either $\fx(t) > 1$, or the relations $\fx(t) = 1$, $\refx^1(t) > 0$ and
$\refx^2(t) > 0$ hold,
this is trivially true since by (\ref{ref-nonidling}) each
term on the left-hand side of (\ref{ref-claim2}) is equal to zero.
In the remaining case when $\fx(t) = 1$, $\refx^1(t) > 0$ and
$\refx^2(t) \leq 0$,   (\ref{ref-nonidling}) shows that
$v^1 (t) = 0$ and $v^2(t)  = \refx^2(t) \leq 0$,
and  once again  (\ref{ref-claim2}) follows.

Next, considering that each solution satisfies the equation (\ref{ref-meas})
and taking the difference,
we have for every $t \in [0,\infty)$,
\begin{eqnarray}
\label{eq-k2}
\del \genk (t)  & =  &   \del v(t) 
 + \int_0^t g(t-s) \del \genk (s) \, ds - \del \addfn(t).
\end{eqnarray}
Now, define
\[ {\cal B} \doteq \left\{t: \int_0^t g(t-s) \del \genk (s) \, ds > 0
\right\}. \]
Then,  combining (\ref{eq-k2}) with (\ref{ref-claim2}),  we
conclude that
\[
\del \genk (t)  \leq  2 \ve_T + \ind_{{\cal B}}(t) \int_0^t g(t-s) \del \genk (s) \,
ds.
\]
Applying the same inequality to $\genk(s)$ for $s \in [0,t]$ and substituting
it into  the last inequality, we then obtain
\begin{eqnarray*}
\del\genk(t)
& \leq &  \ve_T +  \ind_{{\cal B}}(t) \int_0^t g(t-s) \left(2 \ve_T +  \ind_{{\cal
    B}}(s) \int_0^s g(s-r) \del \genk (r) \, dr \right) \, ds  \\
& \leq & \ve_T (1 + G(t))   + \ind_{{\cal B}}(t) \int_0^t g(t-s) \left(
  \ind_{{\cal B}} (s) \int_0^s g(s-r) \del \genk (r) \, dr \right) \, ds.
\end{eqnarray*}
Reiterating this procedure, we obtain
\be
\label{kineq-2}
\del \genk (t) \leq  \ve_T (1 + G(t) + G^{*,2} (t) + \ldots) \leq  \ve_T U(T),
\ee
where $G^{*,n}$ denotes the $n$-fold convolution of $G$.

By symmetry, the inequalities (\ref{kineq-1}) and (\ref{kineq-2}) obtained in
Cases 1 and 2, respectively, also hold with $\del \genk$ replaced by
$-\del \genk$.  Since  $U(T) \geq 1$, we then have
\[
  |\del \genk (t)| \leq 2 \ve_T U(T), \quad \mbox{ for every } t \in [0,T].
\]
Taking the supremum over $t \in [0,T]$, we obtain
\be
\label{kineq-3}
\nrm{\del \genk}_T \leq 2 \ve_T U(T).
\ee
Now,  relations
(\ref{ref-meas}) and (\ref{ref-sm}), together, show that for $i = 1,
2$ and $t \in [0,T]$,
\[
 \genx^i (t) =   \gene^i (t) +\geninitx_0^i  - v^i(0)
  -\int_0^t g(t-s)  \genk^i (s)\, ds + \addfn^i(t).
\]
Taking the difference and using the fact that (\ref{ref-nonidling}) implies  
$|\del (x_0^i - v^i(0))| \leq |\del x_0^i|,$ we obtain 
\begin{eqnarray*}
|\del \genx(t)| & \leq &  \nrm{\del\gene}_T +|\del \geninitx_0| 
 + \int_0^t g(t-s) \nrm{\del \genk}_T \, ds +
\nrm{\del \addfn}_T.
\end{eqnarray*}
Taking the supremum
over $t \in [0,T]$  and using (\ref{kineq-3}), we then conclude that
\[
\nrm{\del \genx}_T \leq 3 \ve_T + 2  \ve_T U(T) G(T) \leq  3 \ve_T (1 + U(T)).
\]
Together with (\ref{kineq-3}) and the fact that
(\ref{ref-nonidling}) implies
$\nrm{\del v(t)}_T \leq \nrm{\del \genx}_T$, 
 this establishes (\ref{ineq-cont}).  
Since the Skorokhod topology coincides with the uniform 
topology on the space of continuous functions, (\ref{ineq-cont}) implies
 that the map $\Lambda$ is continuous at points  
$(E, x_0, Z) \in {\cal C}_{\R_+} \times \R \times {\cal C}_{\R_+}$.   
\end{proof}

\begin{lemma}
\label{lem-lambdameas} 
Suppose the fluid limit is subcritical, critical or supercritical. Then 
the map $\Lambda$ from $\dom(\Lambda) \subseteq 
{\cal D}_{\R}[0,\infty) \times \R \times {\cal D}_{\R}[0,\infty)$ 
to ${\cal D}_{\R}[0,\infty)^3$,  
with ${\cal D_{\R}}[0,\infty)$  
equipped with the Skorokhod topology, is measurable. 
\end{lemma}   
\begin{proof}  
We first observe that it suffices to establish the measurability
of the map from $(E,x_0, Z) \in dom (\Lambda)$ to $X$, where 
$(K, X, v) = \Lambda(E,x_0,Z)$.  Indeed, in this case, 
because the maps $f \mapsto f$, $f \mapsto f \wedge 0$ and $f \mapsto 0$ from 
${\cal D}_{\R}[0,\infty)$, equipped with the Skorokhod topology,
 to itself are 
all measurable (in fact, continuous) and addition is 
also a measurable mapping from ${\cal D}_{\R}[0,\infty)^2$ to 
${\cal D}_{\R}[0,\infty)$, it follows from (\ref{ref-sm}) and 
(\ref{ref-nonidling}) that the map to $(K,v)$, and therefore the map $\Lambda$, 
is also measurable. 

By Remark \ref{rem-contlambda},  
if $(K,X,v) = \Lambda(E,x_0,Z)$ then $X$ satisfies the 
 integral equation 
\[ X(t) = R(t) + \int_0^t g(t-s) F(X(s)) \, ds, \qquad t \geq 0,  
\]
where $R(t) = R_1(t) \doteq Z(t) + E(t) - \int_0^t g(t-s) E(s) \, ds$ and $F = 0$
if $\fx$ is subcritical, $R(t) = R_2 (t) \doteq R_1(t) + (1-G(t)) x_0$ and $F(x) =x$
if $\fx$ is supercritical, and $R(t) \doteq R_2(t)$ and 
$F(x) = x^+$ if $\fx$ is critical. Note that in all cases, 
  $F$ is Lipschitz.  Also, for fixed $t$, 
 the map  $(E,Z,x_0) \mapsto R(t)$  from ${\cal D}_{\R}[0,\infty)^2 \times \R 
\mapsto \R$ is clearly measurable. 
The latter fact implies the proof for the subcritical case. 
For the other two cases, by standard arguments from 
the theory of Volterra integral equations 
(see Theorem 3.2.1 of \cite{burbook}) it follows that 
$X (t) = \lim_{n \ra \infty} ({\cal T}^{(n)} 0)(t)$, where ${\cal
  T}_t^{(n)}$ is the $n$-fold composition of the operator 
${\cal T}: {\cal D}_{\R}[0,\infty) \mapsto {\cal D}_{\R}[0,\infty)$ given by 
$({\cal T}\xi) (t) = R_{2}(t) + \int_0^t g(t-s) F(\xi(s)) \, ds$, $\xi \in {\cal
  D}_{\R}[0,\infty)$.   
Due to the fact that convergence in the Skorokhod topology implies
convergence in $\L^1_{loc}$, the map $\xi \mapsto F(\xi)$ is 
a continuous mapping from $\L^1_{loc}$ to itself and the 
Laplace convolution $\theta \mapsto \int_0^{\cdot} g(\cdot-s) \theta(s) \, ds$ 
is a continuous map from $\L^1_{loc}$ to ${\cal C}[0,\infty)$, 
it follows that for every $t > 0$,
 $(R, \xi) \mapsto {\cal T}(\xi)(t)$ is 
a measurable map.  
 Because the Borel algebra associated with the 
Skorokhod topology is generated by cylinder sets, 
and measurability 
is preserved under compositions and limits, this implies that the 
map from $R$ to  $X$ is measurable.  Note that in the critical case, 
the above equation is of the same form as the 
one obtained in Theorem 3.1 of Reed \cite{reed07},   
and a more detailed  proof of measurability in the critical case can also 
be found  in the Appendix 
of \cite{reed07}.   
\end{proof}

\beginsec

\section{Convergence Results}
\label{sec-martconv}

The representation (\ref{rep-hatmeasn}) of the pre-limit 
dynamics and the continuity 
properties established in Section \ref{subs-cont} reduce 
the problem of convergence of the sequence  
$\{\hmeasn\}_{N \in \N}$ to that of the joint convergence of 
the sequence of stochastic convolution integrals 
$\{\hathn\}_{N \in \N}$ and the sequences representing the initial data. 
In this section, we establish these convergence results. 
First, in Section \ref{subs-asind} (see Proposition \ref{prop-martconv}) 
we  establish the joint convergence of the sequence
$\{\chatmn\}_{N \in \N}$ and the sequence of centered arrival processes
and initial conditions showing, in particular, that 
the centered departure process is asymptotically independent
of the centered arrival process and initial conditions.
Then, in Section \ref{subs-convconv} (see  Corollary \ref{cor-hreg})
we identify (for a suitable family of $f$)
the limits  of the sequence  $\{\hathn(f)\}_{N \in \N}$. 
Both limit theorems are proved using some basic estimates, which are first
obtained in Section \ref{subs-est}.

\subsection{Preliminary Estimates}
\label{subs-est}

Let $U$ denote the renewal function associated with the service distribution
$G$.
We begin with a useful bound,
whose proof is relegated to  Appendix \ref{subs-bdint}.

\begin{lemma}\label{lem-bdint} 
Fix $T < \infty$. For every $N \in \N$ and  positive integer $k$,
\be
\label{eq-bdint}
 \E\left[\left( \fdcompn_\f1 (T) \right)^k \right] =
\E\left[\left(\int_0^T \int_{[0,\endsup)}
h(x)\, \fmeasn_s(dx) \, ds\right)^k
  \right] \leq k! (U(T))^k.  
\ee
Moreover, there exists 
$\overline{C}(T) < \infty$ such that for every 
 positive integer $k$ and measurable function $\newf$
on $[0,\endsup) \times [0,T]$,
\[
 \sup_{N \in \N} \E\left[\left( \fdcompn_\newf (T) \right)^k \right]
\leq k!\left( \overline{C}(T)\right)^k \left( \int_{[0,\endsup)}
\newf^*(x) h(x) \, dx \right)^k,
\]
where $\newf^*(x) \doteq \sup_{s \in [0,T]} |\newf(x,s)|$.  
Furthermore, if Assumptions \ref{as-flinit} and \ref{as-h} hold, then 
\be
\label{eq-bdintl}
\left( \fdcomp_\f1 (T) \right)^k  =
\left(\int_0^T \int_{[0,\endsup)}
h(x)\, \fmeas_s(dx) \, ds\right)^k \leq k! (U(T))^k. 
\ee
\end{lemma}

We now establish some estimates on the martingale measure $\chatmn$, 
which are used in Sections \ref{subs-asind} and 
\ref{subs-convconv} to establish various convergence  
and sample path regularity results.

\begin{lemma}
\label{lem-martest} 
For  every even integer $r$, there
exists a universal constant $C_r < \infty$ such that for every 
 bounded and continuous function  $\newf$ on $[0,\endsup) \times \R$
and $T < \infty$,
\be
\label{eq-martest}
 \E\left[\sup_{s \in [0,T]} \left|\chatmn_s (\newf)\right|^r\right]
\leq C_r \nrm{\newf}_{\infty}^r \left[ \left(\dfrac{r}{2}\right)! \,
  \left(U(T)\right)^{r/2} +  \dfrac{1}{N^{r/2}}\right], \qquad N \in \N, 
\ee  
\be
\label{eq-martest2}
 \E\left[\sup_{s \in [0,T]} \left|\chatm_s (\newf)\right|^r\right]
\leq C_r  \left(\dfrac{r}{2}\right)! \left(U(T)\right)^{r/2} \nrm{\newf}_{\infty}^r
\ee
and, for $0 \leq s \leq t$, 
\be
\label{eq-martest3}
\E\left[ |\chatm_t (\newf) - \chatm_s (\newf)|^r \right] \leq
C_r \left(\fdcomp_{\newf^2}(t) - \fdcomp_{\newf^2}(s) \right)^{r/2}. 
\ee
\end{lemma} 
\begin{proof}
Since $\chatmn(\newf)$ is a martingale,
by the Burkholder-Davis-Gundy (BDG) inequality
(see, for example, Theorem 7.11 of Walsh \cite{walshbook}) it follows that
for any $r > 1$, there exists a universal constant $C_r < \infty$
(independent of $\newf$ and $\chatmn$) such that
\be
\label{start1}
  \E \left[ \sup_{s \leq T} \left| \chatmn_s (\newf) \right|^r \right]
\leq C_r \E \left[ \left( \lan \chatmn (\newf) \ran_T \right)^{r/2} \right]
+ C_r \E \left[ \left| \Delta \widehat{{\cal M}}^{(N),*}_T (\newf) \right|^r
\right],
\ee
where
\[ \Delta \widehat{{\cal M}}^{(N),*}_T (\newf)  \doteq
\sup_{t \in [0,T]} \left| \Delta\chatmn_t (\newf) \right| = 
\sup_{t \in [0,T]} \left|\chatmn_t (\newf) - \chatmn_{t-} (\newf) \right|.
\]
Because the jumps of $\chatmn(\newf)$ are bounded by
$\nrm{\newf}_\infty/\sqrt{N}$, we have
\be
\label{start2}
  \E \left[ \left| \Delta \widehat{{\cal M}}^{(N),*}_T (\newf) \right|^r
\right] \leq \dfrac{\nrm{\newf}_{\infty}^r}{N^{r/2}}.
\ee
On the other hand, by (\ref{cfnal-sint}) it follows that for any $r > 0$,
\[
  \E\left[ \lan \chatmn (\newf) \ran_T^{r/2} \right]  =
\E\left[ \left(\fdcompn_{\newf^2} (T) \right)^{r/2} \right]
 \leq  \nrm{\newf}^{r}_{\infty} \E\left[ \left( \fdcompn_{\f1} (T)
  \right)^{r/2} \right]. 
\]
When combined with (\ref{eq-bdint}) of Lemma \ref{lem-bdint} this shows that if
$r = 2k$, where $k$ is an integer, then
\be
\label{start3}
 \E\left[ \left(\lan \chatmn (\newf) \ran_T \right)^{r/2} \right]
\leq  \nrm{\newf}^{r}_{\infty} \left( \frac{r}{2}  \right)! \,  
\left(U(T)\right)^{r/2}.
\ee
Combining the estimates (\ref{start1})--(\ref{start3}) obtained above,
we obtain (\ref{eq-martest}).

In an exactly analogous fashion, replacing $\chatmn$ and $\fdcompn$,
respectively, by $\chatm$ and $\fdcomp$,  and using the continuity of   
$\chatm(\newf)$ and inequality 
(\ref{eq-bdintl}) of Lemma \ref{lem-bdint}, 
 we obtain (\ref{eq-martest2}). 
Furthermore, for fixed $s \geq 0$, because  
$\{\chatm_{t} (\newf) - \chatm_{s}(\newf)\}_{t \geq s}$ is a continuous martingale with 
quadratic variation process $\{\fdcomp_{\newf^2} (t) - \fdcomp_{\newf^2} (s)\}_{t
  \geq s}$,  
another application of the  Burkholder-Davis-Gundy (BDG) inequality 
yields (\ref{eq-martest3}). 
\end{proof}

As a corollary, we  obtain results on the regularity 
of the processes $\chatmn$ and $\chatm$, which make use  
of the norm inequalities in (\ref{norm-ineq}).

\begin{cor} 
\label{cor-mreg} Each  
$\chatmn, N \in \N$, is a \cad $\H_{-2}$-valued (and hence 
$\mrangespace$-valued) process.  
$\chatm$ is a continuous $\H_{-2}$-valued (and hence $\mrangespace$-valued) 
 process. 
 Moreover, for any $T <
\infty$,  if for every $f \in \testspace$, 
$\chatmn(f) \Rightarrow \chatm(f)$ in ${\cal D}_{\R}[0,T]$ as 
$N \ra \infty$ 
then $\chatmn \Rightarrow \chatm$ in ${\cal D}_{\H_{-2}}[0,T]$ as 
$N \ra \infty$. 
\end{cor}
\begin{proof}Fix $N \in \N$. 
By Remark \ref{rem-mncadlag}, for every 
$f \in \testspace$ there exists a \cad version $\hatmn_{f}$  
of $\chatmn(f)$.
Moreover, for any $T < \infty$  it follows from 
(\ref{eq-martest}) and (\ref{norm-ineq}) that given any  
$\epsilon > 0$ and $\lambda < \infty$  there exists 
 $\delta > 0$ such that if $\nrm{f}_{\H_1} \leq \delta$ then 
\be
\label{mreg-bd}
 \limsup_{N}\P\left( \sup_{t \in [0,T]} |\chatmn_t( f)| > \lambda \right) \leq
 \ve.  
\ee 
Thus, each $\chatmn$ is a $1$-continuous stochastic process in the 
sense of Mitoma \cite{mitoma83b}. 
Since   $\testspace$ is a nuclear Fr\'{e}chet space 
and $\nrm{\cdot}_{\H_1} \overset{\mbox{\tiny{HS}}}{<} 
\nrm{\cdot}_{\H_2}$  (refer to the properties stated in 
Section \ref{subsub-fun}),   by 
Theorem 4.1 of Walsh \cite{walshbook} 
and Corollary 2 of Mitoma \cite{mitoma83b} it follows that $\chatmn$ is a 
\cad $\H_{-2}$-valued, and hence $\mrangespace$-valued, process.

On the other hand, by Lemma \ref{lem-martest} $\chatm(f)$ is a continuous process 
 for every $f \in \testspace$. 
An analogous argument to the one above, that  now  invokes 
 Corollary 1 of Mitoma \cite{mitoma83b} and  (\ref{eq-martest2}),  
shows that $\chatm$ is a continuous  $\H_{-2}$-valued process. 
The last assertion of the corollary follows from  
(\ref{mreg-bd}) and  Corollary 6.16 of Walsh 
\cite{walshbook}. 
\end{proof}

\subsection{Asymptotic Independence}
\label{subs-asind}

We now identify the limit of the sequence of martingale 
measures $\{\chatmn\}_{N \in \N}$ and also show that 
it is asymptotically independent  of 
the centered arrival process and initial conditions. 
We recall  from Assumption \ref{as-diffinit} and Remark \ref{rem-asdiff1}
 that  $\chatm$ is independent of 
the initial conditions 
$(\widehat{x}_0, \hmeas_0, \ops^{\hmeas_0}, \ops^{\hmeas_0}(\f1))$ and 
$\hate$,   where $\hate$ 
 is a diffusion with drift coefficient 
$-\beta$ and diffusion coefficient $\sigma^2$.

\begin{prop}
\label{prop-martconv}
 Suppose Assumptions \ref{as-flinit}--\ref{as-hate}   
and Assumption \ref{as-diffinit}
hold.  Then for every $\newf \in \newcbm$,
$\chatmn ({\newf}) \Rightarrow \chatm ({\newf})$
in ${\cal D}_{\R}[0,\infty)$ as $N \ra \infty$.
Moreover, as $N \ra \infty$,
\[  (\haten, \widehat{x}_0^{(N)}, \hmeasn_0,
\ops^{\hmeasn_0},\ops^{\hmeasn_0}(\f1), \chatmn)
\Rightarrow  (\hate, \widehat{x}_0, \hmeas_0, 
\ops^{\hmeas_0}, \ops^{\hmeas_0}(\f1), \chatm)
\]
in ${\cal D}_{\R}[0,\infty) \times \R \times \H_{-2} \times{\cal
  D}_{\H_{-2}}[0,\infty) \times {\cal
  D}_{\R}[0,\infty) \times {\cal
  D}_{\H_{-2}}[0,\infty)$.
\end{prop} 
\begin{proof}
  We shall first prove the assertion 
 under the supposition that Assumption \ref{as-hate}(a) 
is satisfied,  in which case $\flam, \sigma^2$ are positive 
constants and $\beta$ is a constant in $\R$.  
We start by using  results of Puhalskii and Reiman \cite{PuhRei00} to 
recast the problem in a more convenient form. 
Fix $N \in \N$ and define  
\[ \widehat{L}^{(N)}(t) \doteq \dfrac{1}{\sqrt{N}} \suli_{j=2}^{\en(t)+1}\left(1-\lambda^{(N)}\xi^{(N)}_j\right), \quad
t \in [0,\infty), \]
where recall that $\{\xi^{(N)}_j\}_{j \in \N}$ is the i.i.d.\ sequence of
interarrival times of the $N$th renewal arrival process  $\en$, 
which has mean $1/\lambda^{(N)}$ and variance $(\sigma^2/\flam)/(\lambda^{(N)})^2$. 
Define 
\[ \widehat{\gamma}^{(N)}(t)\doteq
\dfrac{\lambda^{(N)}}{\sqrt{N}}
\left(\sum_{j=2}^{\en(t)+1}\xi_j^{(N)}- \flam t\right),  \qquad t \geq 0. 
\]
Using 
 the definition (\ref{cond-qed}) of $\beta$ and the fact that 
$\fen(t) = \flam t$, we see that  
\be
\label{en-ln}
\haten (t) = \frac{\en(t)-  N \flam
  t}{\sqrt{N}}=\frac{\en(t)-\lambda^{(N)}t}{\sqrt{N}}+ \beta 
t=\widehat{L}^{(N)}(t)+\widehat{\gamma}^{(N)}(t) + \beta t.
\ee 
Puhalskii and Reiman  
(see page 30, Lemma A.1 and  (5.15) of \cite{PuhRei00}) showed 
that $\{\widehat{L}^{(N)}(t), {\cal F}_t^{(N)},t \geq 0\}$
is a locally square integrable martingale 
and, as $N \ra \infty$,  
 $\sup_{t\leq T}|\widehat{\gamma}^{(N)}(t)|\ra 0$ in probability, 
which implies $\gamma^{(N)} \Rightarrow 0$.

We will now show that for every bounded and continuous $f$, 
\be
\label{first-conv}
(\widehat{L}^{(N)}, \hatmn_f) \Rightarrow (B,
\chatm (f)) \quad \mbox{ as } N \ra \infty, 
\ee 
and for real-valued 
bounded, continuous functions $\xi_1$ on $\R^2$ 
and $\xi_2$ on $\R \times \H_{-2} \times {\cal D}_{\H_{-2}}[0,\infty) \times {\cal D}_{\R}[0,\infty)$,  
\be
\label{sec-conv}
\ba{l}
\ds \lim_{N \ra \infty} 
\E[\xi_1(\hatmn_{f}, \widehat{L}^{(N)}) \xi_2(\widehat{x}^{(N)}_0,\hmeasn_0,
\ops^{\hmeasn_0}, \ops^{\hmeasn_0}(\f1))] \\
\qquad \ds = \E[\xi_1(\chatm_{f},B)] \E[
\xi_2(\widehat{x}_0,\hmeas_0,\ops^{\hmeas_0}, \ops^{\hmeas_0}(\f1))]. 
\ea
\ee
Before presenting the proofs of these results, first note that 
on combining (\ref{first-conv}) with  (\ref{en-ln}), the fact that $\gamma^{(N)} \Rightarrow 0$ as $N \ra
\infty$ , the relation $\hate(t) = B(t) - \beta t$, $t \geq 0$, 
and the continuous mapping theorem, it follows that for every bounded and
continuous $f$, 
$(\haten, \chatmn(f)) \Rightarrow (\hate, \chatm(f))$ in ${\cal
  D}_{\R}[0,\infty)^2$.  Together with (\ref{sec-conv}) this implies that 
for every bounded and
continuous $f$, $(\haten, \widehat{x}^{(N)}_0,\hmeasn_0,
\ops^{\hmeasn_0}, \ops^{\hmeasn_0}(\f1), \chatmn(f)) \Rightarrow (\hate,
\widehat{x}_0,\hmeas_0,\ops^{\hmeas_0}, \ops^{\hmeas_0}(\f1), \chatm(f))$ as $N
\ra \infty$. 
Together  with Corollary \ref{cor-mreg} this implies the desired
convergence stated in the proposition.

Thus, to complete the proof (when Assumption \ref{as-hate}(a) holds), it 
suffices to establish  (\ref{first-conv}) and  (\ref{sec-conv}). This is done 
in the following four claims; 
the first three claims below verify conditions of the martingale central 
limit theorem to  establish (\ref{first-conv}), whereas the last claim proves
(\ref{sec-conv}). 

\noi 
{\sc Claim 1.} For  $t \geq 0$,  as $N \ra\infty$, 
$\left[\widehat{L}^{(N)}\right]_t \to \sigma^2 t$ and
 $\left[\hatmn_{f}\right]_t \to \fdcomp_{f^2} (t)$ in
 probability. \\
{\sc Proof of Claim 1.} Note that 
 $\widehat{L}^{(N)}$, being a compensated sum of jumps, 
is a local martingale of finite variation.  
Thus, it is a purely discontinuous martingale (see Lemma 4.14(b) of Chapter I of
Jacod and Shiryaev \cite{JacShiBook}) and hence 
 (by Theorem 4.5.2 of Chapter 1 of
\cite{JacShiBook}), its $\{{\cal F}_t^{(N)}\}$ optional quadratic variation 
is given by 
 \[ \left[\widehat{L}^{(N)}\right]_t=
\frac{1}{N}\suli_{j=2}^{\en(t)+1}(1-\lambda^{(N)}\xi^{(N)}_j)^2.
\] For every $j \in \N$, 
$\E[ (1- \lambda^{(N)} \xi_j^{(N)})^2] = \sigma^2/\bar{\lambda} < \infty$ 
by Assumption \ref{as-hate}(a), and almost surely 
$\fen(t) \ra \flam t$ by Remark  \ref{rem-asdiff1} and Theorem
\ref{th-flimit}.  
Therefore, by the strong law of large numbers for triangular  
arrays of random variables, 
$\left[\widehat{L}^{(N)}\right]_t$ converges   to 
\[\lim_{N\ra\infty} \dfrac{\en(t)-1}{N} \E\left[\left(1-\lambda^{(N)}\xi^{(N)}_j\right)^2\right]=\sigma^2
t\]
as   $N \ra \infty$.  
This establishes the first limit of Claim 1. 

Now, $\hatmn_{f}$ is also a 
compensated sum of jumps. By the same logic we then have  
\[ \left[ \hatmn_{f} \right]_t  =
\sum_{s\leq t} \left( \Delta \hatmn_{f} (s) \right)^2 = \fbaren_{f^2} (t), \]
where the last equality follows because the jumps of $\mn_f$ coincide with
those of $\bare^{(N)}_f$.
The results of Kaspi and Ramanan (see Theorem 5.4,
the discussion below Theorem 5.15 and Proposition 5.17 of \cite{KasRam07}) 
show that $\fbaren_{f^2}(t) \ra \fdcomp_{f^2}(t)$ in probability,
 and so the second limit in Claim 1 is also established. \\
{\sc Claim 2.}  For every $t > 0$,
$\left[\widehat{L}^{(N)},\hatmn_{f}\right]_t \to 0$ 
in probability as $N \ra \infty$. \\
{\sc Proof of Claim 2.} Let $\tau_i^{(N)} \doteq \sum_{j=1}^i \xi_j^{(N)}$ be
the time of the $i$th jump of $\en$. Since $\en$ has unit 
jumps, it follows that 
\be
\label{L-M0}
\left[\widehat{L}^{(N)},\hatmn_f \right]_t 
=\dfrac{1}{\sqrt{N}} \suli_{i \geq 2:\tau_i^{(N)} \leq
  t}\left(1-\lambda^{(N)}\xi^{(N)}_{i+1}\right)
\Delta \hatmn_f\left(\tau^{(N)}_i\right).
\ee 
To prove the claim, it suffices to  show that 
\be\label{L-M}
\E\left[ \left[\widehat{L}^{(N)},\hatmn_f\right]_t^2\right] \leq
\dfrac{\sigma^2\nrm{f}_\infty^2}{\flam N}\E\left[ \sum_{i: \tau_i^{(N)} \leq t} 
  \Delta \fdn\left(\tau_i^{(N)}\right) \right].
\ee
Indeed, then the right-hand side goes to zero as $N \ra \infty$ because the
expectation on the right-hand side is bounded by  $\sup_{N} \E[\fdn(t)]$, 
which is finite by Lemma 5.6 of \cite{KasRam07}
(alternatively, the convergence to zero can also be deduced  from the stronger result stated
in Corollary \ref{cor-dep}).  
The claim then follows by an application of Chebysev's
inequality. 

To establish (\ref{L-M}), we first introduce the filtration 
$\{\tilde{{\cal F}}^{(N)}_t, t \geq 0\}$, which is defined exactly like the filtration 
$\{{\cal F}^{(N)}_t,t \geq 0\}$, except that the  forward recurrence time
process $\ren$ associated with the renewal arrival process $\en$ 
is replaced by the age or backward recurrence time process
$\alpha_E^{(N)}$, which satisfies 
$\alpha_E^{(N)}(s) = s - \sup \{u < s: \en(u) < \en(s) \} \vee 0$ for $s \geq 0$. 
It is easy to verify that $\hatmn_{f}$ is an $\{\tilde{{\cal F}}_t^{(N)}\}$-adapted process and, for each $i \in \N$,  
 $\tau_i^{(N)}$ is an $\{\tilde{{\cal F}}_t^{(N)}\}$-stopping time.
Moreover, for every $i \in \N$,  using the independence of $\xi_{i+1}^{(N)}$
from $\tilde{{\cal F}}^{(N)}_{\tau_i^{(N)}}$ we have 
\begin{eqnarray}
\label{condexp1}
\E[1 - \lambda^{(N)} \xi_{i+1}^{(N)}|\tilde{{\cal F}}_{\tau_i^{(N)}}^{(N)}] &
= &  \E[1 - \lambda^{(N)}
\xi_{i+1}^{(N)}] \qquad = \quad 0, \\
\label{condexp2}
\E[\left(1 - \lambda^{(N)} \xi_{i+1}^{(N)}\right)^2|\tilde{{\cal
    F}}_{\tau_i^{(N)}}^{(N)}] 
& = & \E[\left(1 - \lambda^{(N)}
\xi_{i+1}^{(N)}\right)^2] \, \, = \quad \sigma^2/\flam. 
\end{eqnarray}
For $i \in \N$, 
 conditioning on $\tilde{{\cal F}}^{(N)}_{\tau_i^{(N)}}$, noting that the
 jump times of 
$\hatmn_{f}$  and $\fdn$ coincide, and using  (\ref{condexp2}) and the estimate
$(\Delta \hatmn_f)^2 \leq \nrm{f}_{\infty}^2 \Delta \fdn$, we obtain 
\begin{eqnarray*}
 \E\left[\left(1-\lambda^{(N)}\xi^{(N)}_{i+1}\right)^2
\Delta \hatmn_f\left(\tau^{(N)}_i\right)^2 |\tilde{\widehat{{\cal
      F}}}_{\tau_i^{(N)}}\right]
& = & \dfrac{\sigma^2}{\flam}\Delta \hatmn_f\left(\tau^{(N)}_i\right)^2  \\
& \leq & \dfrac{\sigma^2\nrm{f}_{\infty}^2}{\flam}\Delta \fdn
\left(\tau^{(N)}_i\right).
\end{eqnarray*}
A similar conditioning argument using (\ref{condexp1}) 
shows that for $2 \leq k < i$, $i, k \in \N$,  
\[ \E\left[\left(1-\lambda^{(N)}\xi^{(N)}_{k+1}\right)
\Delta \hatmn_f\left(\tau^{(N)}_k\right)\left(1-\lambda^{(N)}\xi^{(N)}_{i+1}\right)
\Delta \hatmn_f\left(\tau^{(N)}_i\right)\right] = 0. 
\]
Taking first the square and then the expectation of each side of 
(\ref{L-M0}) and using the last two relations, we obtain (\ref{L-M}). 
As argued above, this proves the claim. 


\noi 
{\sc Claim 3.}  The jumps of $(\widehat{L}^{(N)}, \chatmn(f))$ are
asymptotically negligible and  (\ref{first-conv}) holds. \\
{\sc Proof of Claim 3.} The jumps of $\haten$  and $\widehat{L}^{(N)}$
converge to zero as $N \ra \infty$ 
because $\en$ is a counting process with unit 
 jumps and $\sup_{t\leq T}|\widehat{\gamma}^{(N)}(t)|\ra 0$ in probability. 
Also, by Lemma
\ref{lem-dep} and the continuity of  $\dcompn_{f}$,
the jumps of $\chatmn(f) = \hatmn_{f}$ are uniformly bounded by
$\nrm{f}_{\infty}/\sqrt{N}$, and so they also 
converge to zero in probability.  
Because $\{(\widehat{L}^{(N)}, \chatmn(f))\}_{N \in \N}$ is a sequence of 
martingales starting at zero, we can apply the martingale central limit theorem 
(see, e.g., Theorem 1.4 on page 339 of
Ethier and Kurtz \cite{ethkurbook}).  
The conditions of that theorem are verified 
by claims 1-3 above, and (\ref{first-conv}) follows from the observation that 
  $B$ and $\chatm(f)$ are independent, centered mean Gaussian processes with 
variance processes $\sigma^2 t$ and $\fdcomp_{f^2}(t)$, $t \geq 0$,
respectively.

\noi 
{\sc Claim 4.} The asymptotic independence property in (\ref{sec-conv})
holds. \\
{\sc Proof of Claim 4.} 
Conditioned on ${\cal F}^{(N)}_0$, 
 $\widehat{L}^{(N)}$ and $\hatmn_f$  are still compensated 
sums of jumps with the same optional quadratic variation processes as without 
conditioning.  
Thus, the same argument provided above in claims 1--3 above show 
that, conditioned on ${\cal F}^{(N)}_0$, the sequence 
 $(\widehat{L}^{(N)}, \hatmn_f)$ 
still converges weakly to 
$(B, \chatm(f))$.  In particular, by the continuous mapping 
theorem, this then implies that for any 
bounded continuous function $F_1$ on $\R^2$, 
\[  \lim_{N \ra \infty} \E\left[ F_1(\widehat{L}^{(N)}(t), \hatmn_f(t))|{\cal F}_0^{(N)}\right] 
= \E[F_1(B(t), \chatm_t(f))]. 
\]
This establishes the desired asymptotic independence because  for any
bounded, continuous function 
$F_2$ on  $\R \times \H_{-2} \times {\cal D}_{\H_{-2}}[0,\infty) \times {\cal
  D}_{\R}[0,\infty)$, 
by the bounded convergence theorem 
and another application of the continuous mapping theorem for $F_2$, 
\[
\begin{array}{l}
 \ds \lim_{N \ra \infty} \E\left[F_1(\widehat{L}^{(N)}(t), \hatmn_f(t))
  F_2(\widehat{x}^{(N)}_0, \hmeasn_0, \ops^{\hmeasn_0}, \ops^{\hmeasn_0}(\f1))\right]  \\
 \qquad = \ds
 \lim_{N \ra \infty} \E\left[\E\left[F_1(\widehat{L}^{(N)}(t), \hatmn_f(t))|{\cal
  F}_0^{(N)}\right]F_2(\widehat{x}^{(N)}_0, \hmeasn_0,\ops^{\hmeasn_0}, \ops^{\hmeasn_0}(\f1))\right] \\
 \qquad \ds  =  \E\left[F_1(B(t), \chatm_t(f))]\E[F_2(\widehat{x}^{(N)}_0, \hmeasn_0,\ops^{\hmeasn_0}, \ops^{\hmeasn_0}(\f1))\right].
\end{array}
\]
This completes the proof for the case when  Assumption \ref{as-hate}(a) is
satisfied. 

We now turn to the proof for the case when Assumption \ref{as-hate}(b) 
is satisfied.  The proof in this case is similar, and so we only 
elaborate on the differences.  First, for $N \in \N$, define 
$\widehat{L}^{(N)} = L^{(N)}/\sqrt{N}$, where now 
\[ L^{(N)}(t) \doteq 
\left(\en(t)-\int_0^t\lambda^{(N)}(s)\, ds \right), 
\qquad t \in [0,\infty), 
\]
is the scaled and centered inhomogeneous Poisson process, 
and note that 
\be
\label{eq-hate2}
 \haten(t) = \sqrt{N}\left(\fen(t)-\int_0^t 
\lambda(s) \, ds\right)=
\widehat{L}^{(N)} (t) + \int_0^t\beta(s) \, ds. 
\ee
Fix $f$ that is bounded and continuous.  
To complete the proof of the proposition, it suffices to show that 
$(\widehat{L}^{(N)}, \hatmn_f) \Rightarrow (\int_0^\cdot \sqrt{\lambda(s)} \,
dB(s), \chatm(f))$.  
Let $\{\tilde{{\cal F}}_t^{(N)}\}$ be the 
filtration defined in Claim 2 above. Then, 
as is well known,  
 $\{ \widehat{L}^{(N)}(t), \tilde{{\cal F}}_t^{(N)}, t \geq 0\}$ and 
$\{ \hatmn_f(t), \tilde{{\cal F}}_t^{(N)}, t \geq 0\}$ are 
 martingales. 
Hence, once again, we need only verify the conditions 
of the martingale central limit theorem 
(see Theorem 1.4 on page 339 of Ethier and Kurtz \cite{ethkurbook}). 
 Arguing exactly as in Claims 3 and 1 of the proof for case (a), 
it is clear that the jumps of  $\haten$ and 
$\hatmn_f$ are uniformly bounded by $(1 + \nrm{f}_\infty)/\sqrt{N}$ and 
for each $t > 0$, 
  $[\hatmn_f]_t \to \fdcomp_{f^2} (t)$ in probability. 
Keeping in mind that the candidate limit 
 $(\hate, \chatm(f))$ is a pair of  independent, 
continuous Gaussian martingales with respective quadratic variations 
$\int_0^t\lambda(s)ds$ and $\fdcomp_{f^2}(t)$, 
to complete the proof it suffices to  
verify that for every $t \in [0,\infty)$, as $N \ra \infty$, 
the following limits hold in probability:  
\be \label{ind-toshow2}
[\widehat{L}^{(N)}]_t \ra \int_0^t\lambda(s) \, ds, \qquad \quad
[\widehat{L}^{(N)},\mn_{f}]_t \to 0. \ee  
Clearly,  the $\{\tilde{{\cal F}}_t^{(N)}\}$-predictable 
quadratic variation of $\widehat{L}^{(N)}$ is given by  
$\lan \widehat{L}^{(N)}\ran_t = \int_0^t \lambda^{(N)} \, ds$,  
which converges to $\int_0^t \lambda(s) \, ds$ as $N \ra \infty$.  
By Theorem 3.11 of Chapter VIII 
of  \cite{JacShiBook}, this implies  
that $[\widehat{L}^{(N)}]_t$ converges in law to 
$\int_0^t\lambda(s)ds.$  
Because the limit $\int_0^t\lambda(s)ds$ is  deterministic, 
the convergence is also in probability. 
This establishes the first limit in (\ref{ind-toshow2}). 
To establish the second limit, note that 
$\widehat{L}^{(N)}$ and $\hatmn_f$ are both compensated pure jump processes
with continuous compensators,  
and so 
their optional quadratic 
covariation takes the form 
\[
[\widehat{L}^{(N)},\hatmn_f]_t = 
\dfrac{1}{N}\suli_{s\leq t}\Delta L^{(N)} (s) \Delta \hatmn_f(s) 
= \dfrac{1}{N} \suli_{s \leq t} \Delta E^{(N)}(s)  \Delta Q_f^{(N)} (s). 
\]
Noting that $\Delta Q_f^{(N)} \leq \nrm{f}_{\infty} \Delta D^{(N)}$, 
taking expectations of both sides above and then the limit as $N \ra \infty$, 
Corollary \ref{cor-dep} shows that 
\[
\lim_{N \ra \infty} \E\left[\left|[\widehat{L}^{(N)},\hatmn_{f}]_t\right|\right]
\leq \lim_{N \ra \infty} \dfrac{\nrm{f}_\infty}{N} \E\left[\suli_{s\leq
t}\Delta\en(s)\Delta\dn(s)\right] = 0. 
\]
An application of Markov's inequality then yields the 
 second limit in (\ref{ind-toshow2}). 
  The asymptotic independence 
from the initial conditions is proved exactly in the same way as when 
Assumption \ref{as-hate}(a) holds  (see the proof of Claim 4) and is thus omitted. 
\end{proof}

\subsection{Convergence of Stochastic Convolution Integrals }
\label{subs-convconv}

 We now show that for suitable $f$, 
$\{\hathn(f)\}_{N \in \N}$ is a tight sequence of c\`{a}dl\`{a}g
processes.  Since each $\hathn(f)$ is not a martingale, 
the proof is more involved than the corresponding result 
for $\{\chatmn(f)\}_{N \in \N}$ and 
we require an additional regularity assumption (Assumption \ref{as-holder}) 
 on $G$, which holds if the hazard rate function $h$ is bounded and  
$g$ is in $\L^{1+\alpha}$ for some $\alpha > 0$ 
(see Remark \ref{rem-holder}). 
For conciseness,  we will use the notation
\be
\label{not-cg}
 \cg (x) = 1 - G(x), \qquad f\cg(x) = f(x) \cg(x), \qquad x \in [0,\infty). 
\ee
We first derive an elementary inequality.

\begin{lemma}
\label{ineq-basic} Suppose Assumption \ref{as-holder} is 
satisfied, let $f$ be a bounded, H\"{o}lder continuous function with 
constant $C_f$ and exponent $\gamma_f$, and let 
$\gamma_f^\prime \doteq \gamma_G \wedge
\gamma_f$. 
The family of operators $\{\Psi_t,t \geq 0\}$ defined in 
(\ref{def-thetat}) satisfies, for all $0 < t < t^\prime < \infty$, 
\be
\label{est-v3}
 \nrm{\Psi_t f - \Psi_{t^\prime} f}_{\infty} 
\leq \left( C_f + C_G \nrm{f}_{\infty} \right) |t-t^\prime|^{\gamma_f^\prime}. 
\ee
Moreover, if $f \in \H_1$ then $C_f \leq \nrm{f}_{\H_1}$ and there exists a constant 
$C_0 < \infty$, independent of 
$f$, such that 
the right-hand side of (\ref{est-v3}) can be replaced by 
$C_0 \nrm{f}_{\H_1} |t-t^\prime|^{\gamma_f^\prime}$. 
\end{lemma}
\begin{proof} Fix  a bounded, H\"{o}lder continuous function $f$, as in the
  statement of the lemma.  Then we can write 
 $\Psi_t f - \Psi_{t^\prime} f  = \newf^{(1)} + \newf^{(2)}$, where
\[  \newf^{(1)}(x,s) = \dfrac{\cg (x+(t-s)^+)}{\cg(x)}
 \left( f(x+(t-s)^+) - f(x + (t'-s)^+) \right) \]
and
\[ \newf^{(2)}(x,s) =  f(x + (t'-s)^+) \dfrac{ \cg(x + (t-s)^+) - \cg
  (x+(t'-s)^+)}{\cg(x)}.
\]
The H\"{o}lder continuity of $f$
and the fact that $\cg$ is non-increasing show that
$\nrm{\newf^{(1)}}_{\infty} \leq C_f |t- t'|^{\gamma_f}$, 
and  Assumption \ref{as-holder} shows that 
$\nrm{\newf^{(2)}}_{\infty} \leq C_G  \nrm{f}_{\infty}|t- t'|^{\gamma_G}$.  
When combined, these two inequalities yield  
(\ref{est-v3}). 
If $f \in \testspace$,  the Cauchy-Schwarz inequality implies 
\[ \left|f(t) - f(t^\prime)\right| = \left|\int_t^{t^\prime} f^\prime(u)
\, du\right| \leq  \nrm{f^\prime}_{\L^2} (t-t^\prime)^{1/2} \leq 
 \nrm{f}_{\H_1} (t-t^\prime)^{1/2}. 
\] 
Thus, $C_f = \nrm{f^\prime}_{\L^2} \leq \nrm{f}_{\H_1}$ and $\gamma_f = 1/2$, respectively, 
serve as a  H\"{o}lder constant 
and exponent for $f$. 
When combined with (\ref{norm-ineq}) this shows that $f$ is bounded and H\"{o}lder
continuous and the second assertion of the 
lemma holds with $C_0 \doteq (1+ 4 C_G)$. 
\end{proof}

In what follows, for $t > 0$, let $\opint_t:{\cal C}_b([0,\endsup)) \mapsto 
{\cal C}_b([0,\endsup) \times [0,t])$  be the operator 
given by 
\be
\label{def-opint}
(\opint_t f) (x,u) \doteq \int_u^t (\Psi_s \newf(\cdot,s))(x,u) \, ds 
= \int_u^t  \newf (x+s-u,s) \dfrac{1-G(x+s-u)}{1-G(x)} \, ds 
\ee
for $(x,u) \in [0,\endsup) \times [0,t]$ and 
$f \in {\cal C}_b[0,\endsup)$.   The first 
two properties of the next lemma are used in Corollary 
\ref{cor-hreg} to establish 
convergence of the sequence $\{\hathn(f)\}_{N \in \N}$ and 
regularity of the limit.  The third property below is used in the proof of the
Fubini type result in Lemma \ref{lem-fub} and the 
last property is used in the proof of Theorem \ref{th-main2}. 

\begin{lemma}
\label{lem-hathn}
If Assumption \ref{as-holder} 
is  satisfied, the following properties hold: 
\begin{enumerate}
\item 
Given a bounded and H\"{o}lder continuous function $f$ on $[0,\endsup)$,
the sequence of processes 
$\{\hathn(f)\}_{N \in \N}$
is tight in $\D_{\R}[0,\infty)$ and 
$\hath(f)$ is $\P$-a.s.\ continuous. 
\item 
There exists $r \geq 2$ and a constant $C_0 < \infty$ such that for 
any $f \in \testspace$, 
\be
\label{reg3}
  \sup_{N} \E \left[ \sup_{t \in [0,T]} \left|
\hathn_t(f) \right|^r \right] \leq 
C_0 \nrm{f}_{\H_1}^r. 
\ee
\item  
Suppose  $\newf:[0,\endsup) \times [0,\infty) \mapsto \R$ 
is a Borel measurable function such that for every $x \in [0,\endsup)$, 
the function  $r \mapsto \newf(x,r)$ is locally integrable and, for every 
$t \in [0,T]$,  the function $x \mapsto \int_{0}^t \newf(x,r)\, dr$ is 
bounded and H\"{o}lder continuous with constant $C_{\newf,T}$ and
 exponent $\gamma_{\newf,T}$ (that is independent of $t$). 
Then $\P$ almost surely, the  
random field $\{\hath_s(\int_0^t (\newf(\cdot, r)) \, dr)), s, t \geq
0\}$ is jointly continuous in $s$ and $t$.   
\item 
Suppose $\newf \in {\cal C}_b([0,\endsup) \times [0,\infty))$. 
Then the process $\{\chatm_t (\Theta_t \newf), t \geq 0\}$ 
admits a continuous version. 
\end{enumerate}
\end{lemma} 
\begin{proof}
The proof of the lemma is based on a modification of the approach
used in Walsh \cite{walshbook} to establish convergence of 
stochastic convolution integrals, tailored to the present context
(the proof of Theorem 7.13 in \cite{walshbook} works with
a different space of test functions and imposes different conditions on the
martingale measure $\chatmn$, and hence does not apply directly).  
 Fix a bounded and H\"{o}lder continuous $f$ with constant $C_f$ and exponent
 $\gamma_f$ and fix $T < \infty$. 
  Recall from (\ref{rel-hath}) that $\hathn_t (f) = \chatmn_t (\Psi_t f)$, $t \geq 0$. 
The proof of the first two properties will be split into four main claims.\\
\noi
{\sc Claim 1.} For each $N \in \N$,  $\{\hathn_t(f), t \geq 0\}$  
admits  a \cad version.  \\
{\sc Proof of Claim 1.}  The estimates obtained in this 
proof are also used to establish the other claims. 
Fix $N \in \N$ and consider the  following stochastic integral: 
\be
\label{def-vnt}
 V^{(N)}_t (f) \doteq \chatmn_T (\Psi_t f), \qquad t \in [0,T]. 
\ee
Because $\chatmn$ is a martingale measure,  we have
\be
\label{rel-vnhn}
 \hathn_t(f) = \chatmn_t (\Psi_t f)
 =  \E \left[ V^{(N)}_t(f)|{\cal F}_t^{(N)} \right], \qquad t \in [0,T], 
\ee
which shows that the process 
$\{\chatmn_{t}(\Psi_{t} f),t \geq 0\}$ 
is a version of the optional projection of $V^{(N)}(f)$. 
It is well known from
the general theory of stochastic processes (see, for example, Theorem 7.10 of
Chapter V of Rogers and Williams \cite{RogWilbook2}) that the
optional projection of a continuous process is an 
adapted c\`{a}dl\`{a}g process.  
Therefore,  to show that $\{\chatmn_t(\Psi_t f),
t \in [0,T]\}$ admits a c\`{a}dl\`{a}g version, it 
suffices to  show that $V^{(N)}(f)$ admits a continuous version. 
In turn, to establish  continuity, it suffices to  verify
 Kolmogorov's continuity criterion, namely, to show 
that there exist $\tilde{C}_f < \infty$, $\tilde{\theta} > 1$
and $r < \infty$ such that for every $0 \leq t' < t < T$,
\be
\label{est-v1}
 \E \left[ \left| V^{(N)}_t(f) - V^{(N)}_{t'}(f) \right|^r \right]
\leq \tilde{C}_f |t-t'|^{\tilde{\theta} }
\ee
(see, for example,  Corollary 1.2
of Walsh \cite{walshbook}). 
Fix  $0 \leq t' \leq t \leq T$ and note that
\be
\label{rep-vn}
 \left| V^{(N)}_t (f)- V^{(N)}_{t'}(f)  \right| =
\left|\chatmn_T (\Psi_t f - \Psi_{t^\prime} f)\right|. 
\ee
Let $r$ be any positive even integer greater than
$1/\gamma_f'$. 
Together with (\ref{eq-martest}) and (\ref{est-v3}), this implies  
that (\ref{est-v1}) is satisfied with 
  $\tilde{\theta} = r \gamma_f' > 1$ and
\be
\label{form-cf}
\tilde{C}_f = C_r (C_f + C_G \nrm{f}_{\infty})^r ((r/2)! U^{r/2}(T) + 1). 
\ee

\noi 
{\sc Claim 2.} $\hath(f)$ has a continuous version. \\
{\sc Proof of Claim 2.} Analogous to (\ref{def-vnt}) and (\ref{rel-vnhn}), 
we define
$\widehat{V}_t(f) \doteq \chatm_T(\Psi_t f)$,  $t \geq 0$, 
and observe that 
\be
\label{rel-vnhnl}
\hath_t(f) = \chatm_t (\Psi_t f) = \E\left[ \widehat{V}_t (f) |\widehat{{\cal F}}_t\right], \qquad
t \geq 0.  
\ee  
Arguments analogous to those used in Claim 1, with 
the  inequalities (\ref{eq-martest2}) and (\ref{eq-bdintl}), respectively, 
 now playing the role of  (\ref{eq-martest}) and (\ref{eq-bdint}), 
can  be used to show that  
\be 
\label{est-vl}
 \E \left[ \left| \widehat{V}_t(f) - \widehat{V}_{t^\prime}(f) \right|^r \right]
\leq \tilde{C}_f |t-t'|^{\tilde{\theta} }
\ee
with $\tilde{\theta} = r \gamma_f^\prime$. 
  Fix $0 < t^\prime < t < \infty$ with 
$|t-t^\prime| < 1$ and a bounded, H\"{o}lder continuous $f$. 
Using  (\ref{rel-vnhnl}) and adding and subtracting 
$\chatm_{t^\prime} (\Psi_{t} f) = \E[V_{t}(f)|\widehat{{\cal F}}_{t^\prime}]$, we obtain  
\be
\label{est-vmore}
\hath_{t}(f) - \hath_{t^\prime} (f)  =  \E\left[V_t (f) - V_{t^\prime}(f)|\widehat{{\cal F}}_{t^\prime}\right]
+ \chatm_{t} (\Psi_{t}f) - \chatm_{t^\prime} (\Psi_{t}f).  
\ee
Consider any even integer $r > 2/\gamma_f^\prime \vee 4$  so 
that (\ref{est-vl}) holds with $\tilde{\theta} > 2$,  let 
$\bar{\theta} \doteq \lfloor r/2 \wedge \tilde{\theta} \rfloor$  and note that 
$\bar{\theta}$ is an integer greater than or equal to $2$.  
 Taking first the $r$th power and then expectations of both sides 
of (\ref{est-vmore}), and using 
the inequality $(x + y)^r \leq 2^r(x^r+y^r)$ and Jensen's inequality, 
we obtain 
\[  \E\left[ \left| \hath_{t^\prime}(f) - \hath_t(f)\right|^r \right] 
 \leq 2^r \left( \E \left[ \left| V_t (f) - V_{t^\prime}(f) \right|^r \right] 
+  \E\left[\left| \chatm_t(\Psi_{t^\prime} f) - \chatm_{t^\prime}
  (\Psi_{t^\prime}(f)\right|^r \right] \right). \]
Applying the estimates (\ref{est-v1}), 
(\ref{eq-martest3}) and 
the fact that $\nrm{\Psi_{t} f}_{\infty} \leq \nrm{f}_\infty$, and 
then the inequality $x^2 + y^2 \leq (x+y)^2$ for $x,y \geq 0$, this implies that 
\begin{eqnarray}
\nonumber
\ds \E\left[ \left| \hath_{t^\prime}(f) - \hath_t(f)\right|^r \right] 
 & \leq & \ds  2^r \tilde{C}_f|t-t^\prime|^{\tilde{\theta}} + 
2^r C_r 
\left( \fdcomp_{(\Psi_{t^\prime} f)^2}(t) - \fdcomp_{(\Psi_{t^\prime}
    f)^2}(t^\prime) \right)^{r/2} \\
\label{addeq-1}
  & \leq & \ds  2^r (\tilde{C}_f \vee C_r\nrm{f}_{\infty}^2) 
\left( t+\fdcomp_{\f1}(t)
- t^\prime - \fdcomp_{\f1}(t^\prime) \right)^{2}. 
\end{eqnarray} 
Since $t+\fdcomp_{\f1}(t)$ is a non-negative,   increasing 
function of $t$, 
the generalized Kolmogorov's continuity criterion 
(see, for example, Corollary 3 of \cite{mitoma83b}) implies that 
$\hath(f)$ has a continuous version.

\noi 
{\sc Claim 3.}  The estimate (\ref{reg3}) is satisfied. \\
{\sc Proof of Claim 3.} 
From the proof of Corollary 1.2 of Walsh \cite{walshbook} 
it is straightforward to deduce that  (\ref{est-v1}) also implies  that 
there exists a constant $\tilde{C}_r < \infty$, which depends on $r$ but is independent of 
$N$ and $f$, such that 
\be
\label{est-supv} \E\left[\sup_{s\in [0, T]}\left|V^{(N)}_s(f)\right|^r \right]
\leq \tilde{C}_r \tilde{C}_f. 
\ee
By (\ref{rel-vnhn}) and Jensen's inequality, for every $t \in [0,T]$, 
\[  |\hathn_t f)|^r \leq \E\left[\sup_{s \in [0,T]} V_s^{(N)}(f)|{\cal F}_t^{(N)}\right]^r
\leq \E\left[ \sup_{s \in [0,T]}\left|V_s^{(N)}(f)\right|^r |{\cal F}_t^{(N)}\right].  
\]
 By (\ref{est-supv}), the last term above (viewed as a processs in $t$) is a martingale. So  
Doob's inequality   and (\ref{est-supv}) imply that  
\begin{eqnarray*} 
 \E\left[ \sup_{t \in [0,T]} \left|\hathn_t (f)\right|^r\right]
& \leq & 
  \dfrac{r^2}{r-1} 
\E\left[ \E\left[ \sup_{s \in [0,T]} \left|V_s^{(N)}(f)\right|^r|{\cal
  F}_T^{(N)}\right] \right] 
\\ & = & 
 \dfrac{r^2}{r-1} \E\left[ \sup_{s \in [0,T]} |V_s^{(N)}(f)|^r\right]
 \leq  \dfrac{r^2}{r-1} \tilde{C}_r \tilde{C}_f. 
\end{eqnarray*}
If $f \in \testspace$ then by the expression for $\tilde{C}_f$ given in 
(\ref{form-cf}) and the inequalities $\nrm{f}_{\infty} \leq \sqrt{6}
\nrm{f}_{\H_1}$ and $C_f \leq \nrm{f}_{\H_1}$ established in (\ref{norm-ineq}) and 
Lemma \ref{ineq-basic}, respectively, 
the right-hand side above can be replaced by
$C_0 \nrm{f}_{\H_1}^r$, for an appopriate constant $C_0 = C_0(G,r,T) < \infty$
that is independent of $N$ and $f$. Thus, 
 (\ref{reg3}) follows. 

\noi 
{\sc Claim 4.}  The sequence $\{\hathn_t(f), t \geq 0\}_{N \in \N}$ 
is tight in ${\cal D}_{\R}[0,\infty)$. \\
{\sc Proof of Claim 4.}  We will prove the claim by verifying 
Aldous' criteria for tightness of stochastic processes. 
A minor modification of the arguments in Claims 1-3 shows that 
if $\delta_N \in (0,1)$ and $T_N$ is an $\{{\cal F}^{(N)}_t\}$ stopping time such that
$T_N + \delta_N \leq T$, then for any even integer $r \geq 2$, 
\be
\label{eq-delv}
  \E\left[ |V_{T_N+\delta_N}^{(N)}(f) - V_{T_N}^{(N)}(f)|^r \right]
\leq \tilde{C}_f \delta_N^{r \gamma_f^\prime}.
\ee
Let $\delta_N \in (0,1)$ and let $T_N$ be an $\{{\cal F}^{(N)}_t\}$ stopping time
such that $T_N + \delta_N \leq T$.  Using (\ref{rel-vnhn}) and
(\ref{def-vnt}), the difference $\hathn_{T_N+\delta_N} (f) - \hathn_{T_N} (f)$
can be rewritten as 
\[
\begin{array}{l}
\ds
  \E\left[V_{T_N+\delta_N}^{(N)}|{\cal F}_{T_N + \delta_N}^{(N)} \right]
- \E\left[V_{T_N}^{(N)}|{\cal F}_{T_N}^{(N)} \right]  \\
\qquad \ds =   \E\left[V_{T_N+\delta_N}^{(N)} - V_{T_N}^{(N)}|{\cal F}_{T_N + \delta_N}^{(N)}
\right]
+ \E \left[ V_{T_N}^{(N)} |{\cal F}_{T_N + \delta_N}^{(N)}\right] 
  -
\E \left[ V_{T_N}^{(N)} |{\cal F}_{T_N}^{(N)} \right]\\
\qquad \ds = 
 \E\left[V_{T_N+\delta_N}^{(N)} - V_{T_N}^{(N)}|{\cal F}_{T_N + \delta_N}^{(N)}\right]  + \underset{[0,\endsup) \times (T_N, T_N +
  \delta_N]}{\int \int}
\Psi_{T_N} (f) (x,s) \,  \chatmn(dx, ds).
\end{array}
\]
Recalling the covariance functional of
 $\chatmn$ specified in (\ref{cfnal-sint})   and the
fact that $\nrm{\Psi_{T_N} (f)}_{\infty} \leq \nrm{f}_{\infty}$,
this implies that
\[
\begin{array}{l}
\ds \E \left[ \left|\hathn_{T_N+\delta_N} (f) - \hathn_{T_N} (f)
  \right|^2\right] \\
\qquad  \leq
\ds 2 \E \left[ \left|V^{(N)}_{T_N+\delta_N} (f) - V^{(N)}_{T_N} (f)
  \right|^2\right]  + 2 \nrm{f}_{\infty}^2 \E \left[
\fdcompn_{\f1} (T_N + \delta_N) - \fdcompn_{\f1} (T_N) \right] \\
\ds \qquad \leq  2 \tilde{C}_f  \delta_N^{2 \gamma_f^\prime} + 
2 \nrm{f}_{\infty}^2 \sup_{\tilde{N}}\E \left[\sup_{t \in [0,T]}
\left( \fdcomp^{(\tilde{N})}_{\f1} (t+\delta_N) - \fdcomp^{(\tilde{N})}_{\f1} (t)\right) \right],
\end{array}
\]
where the last equality uses
 (\ref{eq-delv}) with $r=2$.  As $\delta_N \ra 0$, the 
first term on the right-hand side clearly converges to zero, whereas 
Lemma 5.8(2) of Kaspi and Ramanan \cite{KasRam07} shows that the 
second term also converges to zero. 
We  conclude that 
$\hathn_{T_N+\delta_N}(f) - \hathn_{T_N}(f)$ converges to zero in $\L^2$, and
hence in probability.  On the other hand, (\ref{reg3}) shows that 
the sequence $\{\hathn(f)\}_{N\in \N}$ is uniformly bounded in $\L^r$.
We have thus verified Aldous'
criteria  (see, for example, Theorem 6.8 of Walsh
\cite{walshbook}), and hence 
 the sequence $\{\hathn(f)\}_{N \in \N}$ is tight.  \\

We now turn to the proof of property 3.  Fix $T < \infty$, 
 let $\newf$ be as stated in the lemma and for 
$t \in [0,T]$, 
define $f_{\newf}^t(x) = \int_0^t \newf(x,r) \, dr$.  
For $s, t, s^\prime, t^\prime \in [0, T]$ with $t^\prime <
t$, we have
\[ \hath_{s}\left(f_{\newf}^{t}\right) - 
\hath_{s^\prime} \left(f_{\newf}^{t^\prime}\right) = 
\hath_{s}\left(f_{\newf}^{t}\right) - 
\hath_{s^\prime} \left(f_{\newf}^{t} \right) + 
 \chatm_{s^\prime} \left( \Psi_s \left(\int_{t^\prime}^{t} \newf(\cdot, r) \,
    dr\right)\right). 
\]
Due to the assumed boundedness and H\"{o}lder continuity of 
$f_{\newf}^{t}$, (\ref{addeq-1}) and (\ref{eq-martest2}) together with the
above
relation imply that 
 there exists a sufficiently large integer $r$, constant $C(T,r,\newf) < \infty$ 
and $\tilde{\theta} = \tilde{\theta} (r,\newf) > 1$ such that 
\[ 
\begin{array}{l}
\ds 
 \E\left[\left|\hath_{s^\prime} \left( \int_0^{t^\prime} \newf(\cdot, r) \, dr \right) 
- \hath_{s} \left( \int_0^t \newf(\cdot, r) \, dr \right)\right|^r\right] = \E\left[\left| \hath_{s^\prime} (f^{t^\prime}_{\newf}) - \hath_s(f^t_{\newf})  \right|^r \right]\\
\qquad \leq  \ds C(T,r,\newf) \left(
\left|s + \fdcomp_1 (s) - s^\prime -
  \fdcomp_1(s^\prime)\right|^{\tilde{\theta}} 
+ |t-t^\prime|^{\tilde{\theta}}\right). 
\end{array}
\]
Property 3 then follows from the generalized Kolmogorov's criterion for
continuity of random fields. 

The proof of the last property of 
the lemma is similar to the  proof of 
the continuity of $\hath$ given in Claim 2, 
and so we only provide a rough sketch.
 Let $R_t(\newf) \doteq \chatm_t (\Theta_t \newf)$,  define 
 $\tilde{V}_t (\newf) \doteq \chatm_T(\Theta_t \newf)$ and note that  
$R_t (\newf)= \E[\tilde{V}_t(\newf)|{\cal F}_t]$.  
In a manner similar to (\ref{est-vmore}), we can write 
\[ R_t(\newf) - R_{t^\prime} (\newf) = 
\E\left[\chatm_T \left(\int_{t^\prime}^t \Psi_s \newf \, ds \right)| {\cal F}_{t^\prime} \right]
+ \chatm_t \left( \Theta_t \newf\right) - \chatm_{t^\prime} \left( \Theta_t \newf \right). 
\]
Using Jensen's inequality,  (\ref{eq-martest2}) with $r = 4$ and the inequalities 
$\nrm{\int_{t^\prime}^t \Psi_s \newf \, ds}_{\infty} \leq
|t-t^\prime|\nrm{\newf}_{\infty}$  and $\nrm{\Theta_{t^\prime}
\newf}_{\infty} \leq T \nrm{\newf}_{\infty}$, 
it follows that for  $0 < t^\prime < t < T$, $t-t^\prime \leq 1$,   
\begin{eqnarray*}
 \E\left[\left|R_t(\newf) - R_{t^\prime}(\newf)\right|^4 \right] 
& \leq & 2^4\left(\E\left[\left|\chatm_T \left(\int_{t^\prime}^t \Psi_s \newf \, ds\right)\right|^4\right] 
+ \nrm{\Theta_{t^\prime}(\newf)}_{\infty}^4\left(\fdcomp_{\f1}(t) - \fdcomp_{\f1}(t^\prime)\right)^{2} \right) \\
& \leq &  2^4 \tilde{C}(T) \nrm{\newf}_{\infty}^4 \left( \left| t - t^\prime\right|^4 
+ \left(\fdcomp_{\f1}(t) - \fdcomp_{\f1}(t^\prime)\right)^{2} \right) \\
& \leq & 2^4 \tilde{C}(T) \nrm{\newf}_{\infty}^4 
\left( (t-t^\prime)^2 + \left(\fdcomp_{\f1}(t) -
    \fdcomp_{\f1}(t^\prime)\right)^{2} \right) \\
& \leq &  2^4 \tilde{C}(T)\nrm{\newf}_{\infty}^4  
\left( t-t^\prime +  \fdcomp_{\f1}(t) - \fdcomp_{\f1}(t^\prime)\right)^{2},  
\end{eqnarray*}
where $\tilde{C}(T) \doteq \left(2 C_4 U(T)^2\right) \vee T^4$. 
The claim then follows from the generalized Kolmogorov continuity 
criterion. 
\end{proof}

Combining Lemma \ref{lem-hathn} with arguments similar to those
 used in the proof of Corollary \ref{cor-mreg},
we now obtain the main convergence result of the section.

\begin{cor}
\label{cor-hreg} 
As $N \ra \infty$, 
$\widehat{Y}_1^{(N)} \Rightarrow \widehat{Y}_1$ in ${\cal Y}_1$.  
 Also, if for any bounded, H\"{o}lder continuous 
$f$,  (\ref{jt-limit}) holds then  
$(\widehat{Y}_1^{(N)}, \hathn(f)) \Rightarrow 
(\widehat{Y}_1,\hath(f))$ as $N \ra \infty$. 
\end{cor}
\begin{proof}
Fix $N \in \N$. For $\tilde{k}, k \in \N$, $i = 1, \ldots, \tilde{k}$, 
 $j = 1, \ldots, k$, let $\tilde{f}_i$ and $f_j$, respectively,  be 
bounded, continuous and bounded, H\"{o}lder continuous functions. 
Proposition \ref{prop-martconv} and  Lemma \ref{lem-hathn} 
imply that the sequence 
\[ \{(\chatmn(\tilde{f}_1), \ldots, \chatmn(\tilde{f}_{\tilde{k}}), 
\hathn(f_1), \ldots, \hathn(f_k))\}_{N \in \N}
\] 
 is tight in ${\cal D}_{\R}[0,\infty)^{\tilde{k} +k}$.  
Since $\hathn_t(f) = \chatmn_t(\Psi_t f)$ and, likewise, 
 $\hath_t(f) = \chatm_t (\Psi_t f)$, Proposition \ref{prop-martconv} 
also shows that for $\tilde{t}_i, t_j \in [0,\infty)$, 
$i = 1, \ldots, \tilde{k}$,  $j = 1, \ldots, k$,
as $N \ra \infty$,  the corresponding projections converge: 
\begin{eqnarray*} \left(\chatmn_{\tilde{t}_1}(\tilde{f}_1), \ldots,
  \chatmn_{\tilde{t}_k}(\tilde{f}_k), \hathn_{t_{1}}(f_{1}), \ldots, 
\hathn_{t_k}(f_{k}) \right) \Rightarrow \\
\qquad 
\left(\chatm_{\tilde{t}_1}(\tilde{f}_1), \ldots, \chatm_{\tilde{t}_k}(\tilde{f}_k),
\hath_{t_{1}}(f_{1}), \ldots, \hath_{t_k}(f_{k}) \right). 
\end{eqnarray*}
Since $\testspace$ is a subset of the space of 
bounded and H\"{o}lder continuous functions, 
together the last two statements show that 
\[ (\chatmn(\tilde{f}), \hathn(f), \hathn(f_1)) 
\Rightarrow (\chatm(\tilde{f}), \hath(f), \hath(f_1)) \]
 for $f, \tilde{f} \in \testspace$ and $f_1$ bounded 
and H\"{o}lder continuous.   
 Because  $\testspace$ and $\dualspace$ are  nuclear Fr\'{e}chet spaces, 
 by Mitoma's theorem 
(see Corollary 2 of \cite{mitoma83b}) it follows that 
$(\chatmn, \hathn, \hathn(f_1)) \Rightarrow (\chatm, \hath, \hath(f_1))$  
in ${\cal D}_{\dualspace}[0,\infty)^2 \times {\cal D}_{\R}[0,\infty)$. 
Since  $\nrm{\cdot}_{\H_1} \overset{\mbox{\tiny{HS}}}{<} 
\nrm{\cdot}_{\H_2}$ and the estimate  (\ref{reg3}) holds, 
 Corollary 6.16 of Walsh \cite{walshbook} then 
shows that $(\chatmn, \hathn, \hathn(f_1)) \Rightarrow 
(\chatm, \hath, \hath(f_1))$ in ${\cal D}_{\H_{-2}}[0,\infty)^2 \times 
{\cal D}_{\R}[0,\infty)$, as $N \ra \infty$. 
Now, $(\hath, \hath(f_1))$ is  
 adapted to the filtration generated by 
 $\chatm$, and $\chatm$ is independent of 
$(\hate, \hinitx, \hmeas_0, \ops^{\hmeas_0}, \ops^{\hmeas_0}(\f1))$ 
by Assumption \ref{as-diffinit}.  
Thus,  the same argument used to establish asymptotic independence in 
Proposition \ref{prop-martconv} also shows that 
the convergence above can be strengthened to 
$\widehat{Y}_1^{(N)} \Rightarrow \widehat{Y}_1$ and 
$(\widehat{Y}_1^{(N)}, \hathn(f)) \Rightarrow (\widehat{Y}_1, \hath(f))$. 
\end{proof}

\beginsec

\section{Proofs of Main Theorems}
\label{sec-proofs}

\subsection{The Functional Central Limit Theorem}
\label{subs-fclt}

Before presenting the proof of Theorem \ref{th-main}, 
we first establish  the main convergence result.

\begin{prop}
\label{prop-fclt} 
Suppose Assumptions \ref{as-flinit}--\ref{as-diffinit} are satisfied and 
suppose that the fluid limit is either  
subcritical, critical or  supercritical.  Then 
the convergence (\ref{conv-main}) holds 
 and $(\hatk, \hatx, \hmeas(\f1))$ has
almost surely continuous sample paths.  
Moreover, if $g$ is continuous and $\hmeas$ is defined as in (\ref{def-hmeas}), then as $N \ra \infty$, 
\be
\label{conv-fclt}
  (\widehat{Y}_1^{(N)}, \hatkn, \hatxn, \hmeasn, \hcalkn, \hcalkn (\f1))  \Rightarrow 
(\widehat{Y}, \hatk, \hatx, \hmeas, \hcalk, \hcalk(\f1)) \ee
in ${\cal Y}_1 \times {\cal D}_{\R}[0,\infty)^2 \times {\cal D}^2_{\H_{-2}}
[0,\infty) \times {\cal D}_{\R}[0,\infty)$. 
\end{prop}
\begin{proof}  
Corollary \ref{cor-hreg} shows that $\widehat{Y}_1^{(N)} \Rightarrow 
\widehat{Y}_1$ in ${\cal Y}_1$ as $N \ra \infty$, which in particular 
implies that 
\[ (\haten, \hinitxn, \ops^{\hmeasn_0}(\f1),\hathn (\f1)) 
\Rightarrow (\hate, \hinitx, \ops^{\hmeas_0}(\f1), \hath(\f1))
\]
 as $N \ra \infty$.  
 By Remark \ref{rem-asdiff1}, 
Assumption \ref{as-diffinit} and Lemma \ref{lem-hathn}(1), 
 $(\hate,\ops^{\hmeas_0}(\f1), \hath(\f1),
\ops^{\hmeas_0}, \hath)$  has 
almost surely continuous sample paths with values in 
$\R_+^3 \times \H_{-2}^2$. 
Recalling  the definition (\ref{def-y1n}) of $\widehat{Y}_1^{(N)}$ and the 
fact that addition in the Skorokhod topology is continuous at points in 
${\cal C}[0,\infty)$, 
 this  implies that as $N \ra \infty$,  
\be
\label{lim-1}
 (\widehat{Y}_1^{(N)}, \haten, \hinitxn, \ops^{\hmeasn_0}(\f1) - \hathn (\f1)) 
\Rightarrow (\widehat{Y}_1, \hate, \hinitx, \ops^{\hmeas_0} (\f1) - \hath(\f1)).
\ee 
By Lemma \ref{lem-lambdarep}, almost surely 
$(\hatkn, \hatxn, \lan \f1, \hmeasn\ran) = \lc (\haten, \hinitxn,
\ops^{\hmeasn_0}(\f1) - \hathn (\f1))$  for all $N$ large enough.  
The continuity of $\lc$ with respect to the uniform topology on 
${\cal D}_{\R}[0,\infty)$ established in Proposition \ref{prop-cont}, 
the measurability of $\lc$ with respect to the Skorokhod topology 
on ${\cal D}_{\R}[0,\infty)$ established in Lemma \ref{lem-lambdameas} 
and a generalized version of the continuous mapping theorem 
(see, for example, Theorem 10.2 of Chapter 3 of \cite{ethkurbook}) then 
shows that the convergence (\ref{conv-main}) holds with 
 $(\hatk, \hatx, \hmeas(\f1)) \doteq 
\lc (\hate, \hinitx, \ops^{\hmeas_0}(\f1) - \hath(\f1))$. 
 By the model assumptions and Lemma \ref{lem-dep}, 
almost surely, 
 $\Delta \en (t) \leq 1$ and $\Delta \dn(t) \leq 1$ for every $t \geq 0$. 
Combining this with (\ref{def-dn}), (\ref{eq-dnonidling2}) 
and the second equation for $\hatkn$ in 
(\ref{eq-dprelimit3}), it follows that almost surely for every $t \geq 0$, 
\[ \max(\Delta \hatkn(t), \Delta \hatxn (t), \Delta \lan \f1, \hmeasn_t\ran) 
\leq \dfrac{3}{\sqrt{N}}. 
\]
Because jumps are continuous 
in the Skorokhod topology, the weak 
convergence  of $(\hatkn, \hatxn, \hmeasn(\f1))$ to 
$(\hatk, \hatx, \hmeas(\f1))$ established in 
(\ref{conv-fclt}) 
shows that $(\hatk, \hatx, \hmeas(\f1))$ is almost surely  continuous. 
(Note that when $g$ is continuous, the continuity 
of $(\hatk, \hatx, \hmeas(\f1))$ is also 
guaranteed by Remark \ref{rem-contlambda}.)

Next, suppose $g$ is continuous.
By Lemma \ref{lem-calk}(2), both 
the map $\Gamma$ that takes $\hatkn$ to $\hcalkn$ and 
the map that takes $\hatkn$ to $\hcalkn(\f1)$ is continuous (with respect to the Skorokhod topology on 
both the domain and range).  So by (\ref{conv-main}) and 
the continuous mapping theorem, as $N \ra \infty$, 
\be
\label{thm2-limit1}
 (\widehat{Y}_1^{(N)}, \hatkn, \hatxn,  \hcalkn, \hcalkn (\f1))
\Rightarrow (\widehat{Y}, \hatk, \hatx,  \hcalk, \hcalk(\f1)).  
\ee
In turn, the representation (\ref{rep-hatmeasn}) shows that 
$\hmeasn = \ops^{\hmeasn_0} - \hathn + \hcalkn$ and hence, 
is a continuous mapping of $\widehat{Y}_1^{(N)}$ and $\hcalkn$. 
Thus,  (\ref{thm2-limit1}) and another application of 
the continuous mapping theorem 
show that (\ref{conv-main2}) holds with $\hmeas = \ops^{\hmeas_0} - \hath
+ \hcalk$. Since this coincides with the definition of $\hmeas$ given 
in (\ref{def-hmeas}), this establishes the proposition.  
\end{proof}

We now  prove the first two  main results of the paper. 

\begin{proof}[Proof of Theorems \ref{th-main} and \ref{th-fclt}]  
The limit in (\ref{conv-main}), the continuity of 
$(\hatk, \hatx, \hmeas(\f1))$ and Theorem \ref{th-fclt} 
follow from Proposition \ref{prop-fclt}. 
The  relation (\ref{eq-hdn}) shows that 
\[  \int_0^{\cdot} \lan h, \hmeasn_s\ran \, ds = \lan \f1, \hmeasn_0 \ran - 
\ops^{\hmeasn_0}_{\cdot} (\f1) - \chatmn_{\cdot} (\f1) + \hathn_{\cdot}(\f1) + \int_0^{\cdot} \hatkn(s) g(\cdot-s)
\, ds. 
\]
The last term equals $\hatkn - \hcalkn(\f1)$, and so by 
Lemma \ref{lem-calk}(2) the mapping from $\hatkn$ to the last 
term is continuous.   The limit (\ref{conv-main}) along with 
 the continuous mapping theorem then shows  
 that  $\int_0^\cdot \lan h, \hmeasn_s \ran \, ds \Rightarrow \cendep$, 
where $\cendep$ is as defined in (\ref{def-calg}).  
The 
relation (\ref{eq-dprelimit2}) for $\hatxn$, the continuity of the 
limit and another 
application of the continuous mapping theorem then yields 
 the representation (\ref{rep-hatx}) for $\hatx$.  This completes 
the proof of the theorem. 
\end{proof}

\subsection{The Semimartingale Property}
\label{subs-smg}

In view of the representation (\ref{rep-hatx}) for $\hatx$ and the 
fact that $\chatm_{\f1}$ and $\hate$ are by definition 
semimartingales, to show that $\hatx$ is a semimartingale it suffices 
to show that $\cendep$ is a semimartingale.
This is first carried out in Lemma \ref{lem-inth} below. 
Throughout, we assume that Assumptions \ref{as-flinit}, \ref{as-hate}, 
and \ref{as-diffinit}' are satisfied, the fluid limit is 
subcritical, critical or supercritical 
and that, in addition,  $h$ is bounded and absolutely continuous.  
As stated in Remark \ref{rem-holder}, if $h$ is bounded then Assumptions 
 \ref{as-h} and \ref{as-holder} are automatically satisfied. 
Thus, the results of Theorems \ref{th-main} and \ref{th-fclt}
are valid. 

\begin{lemma}
\label{lem-inth} 
Almost surely, the function $t \mapsto \cendep(t)$ is absolutely continuous
and 
\be
\label{dercalg}
 \dfrac{d\cendep(t)}{dt} = \hmeas_t (h),   \qquad \mbox{ a.e.\ } t \in
 [0,\infty). 
\ee
\end{lemma}
\begin{proof} 
We start by rewriting the expression (\ref{def-calg}) for 
$\cendep$ obtained in Theorem \ref{th-main} in a more convenient form.   
By the definitions of $\are_t$ and $\ops^{\hmeas_0}$ given 
in (\ref{def-aret}) and
 (\ref{rel-ops}) , respectively,  for $t > 0$, 
\[
\begin{array}{rcl}
\ds\hmeas_0 (\f1)  - \ops_t^{\hmeas_0} (\f1)  =  
\ds\hmeas_0 \left( \dfrac{G(\cdot+t) - G(\cdot)}{1 - G(\cdot)} \right)   
& = & \ds\hmeas_0 \left( \int_0^t h(\cdot+r) \dfrac{1 -
      G(\cdot+r)}{1-G(\cdot)} \, dr\right)  \\
& = & \ds\hmeas_0 \left( \int_0^t \Phi_r h (\cdot) \, dr \right).
\end{array}
\]
By (\ref{eq-phibound}) and the boundedness of $h$, 
$\Phi_r h$ is bounded (uniformly in $r$) and absolutely continuous, 
 and Assumption \ref{as-holder} 
implies that $\int_0^t \Phi_r h \, dr  = (1-G(\cdot+r))/(1-G(\cdot))$ is H\"{o}lder
 continuous. 
Therefore,  applying Assumption \ref{as-diffinit}'(d) with  $\newf = \Phi_r h$,
 it follows  that
\be
 \label{eq-inth1}
\ds\hmeas_0 (\f1)  - \ops_t^{\hmeas_0} (\f1) =  \ds\int_0^t  \hmeas_0 \left( \Phi_r h \right) \, dr = \ds \int_0^t
\ops^{\hmeas_0}_r (h) \, dr.
\ee 
In a similar fashion, for $t > 0$, using the identity 
$\hath_t(\f1) = \chatm_t(\Psi_t \f1)$ we have 
\begin{eqnarray*}
\ds\chatm_{t} (\f1) - \hath_t (\f1) & = & \ds
\underset{[0,\endsup) \times [0,t]}{\int\int} 
\dfrac{G(x+t-u) - G(x)}{1- G(x)} \, \chatm (dx, du) \\
& = & 
\ds \underset{[0,L) \times [0,t]}{\int\int}\left(  
\int_{u}^t  \dfrac{h(x+r-u) (1-G(x+r-u)}{1-G(x)} \,dr \right) \, \chatm (dx,
du).
\end{eqnarray*}
Because $h \in {\cal C}_b[0,\endsup)$, we can set $\tilde{\newf} = h$
in (\ref{fub-1}) of Lemma \ref{lem-fub} to obtain 
\be
\label{eq-inth2}
 \ds\chatm_{t} (\f1) - \hath_t (\f1) 
  =  \int_0^t \hath_r (h) \, dr.
\ee

If $h$ is absolutely continuous, then $g$ is absolutely continuous and 
 by the commutativity of the convolution and differentiation operations,     
the function $t \mapsto  \int_0^t g(t-s) \hatk(s) \, ds$ is absolutely continuous with derivative 
$g(0) \hatk(t) + \int_0^t g^\prime (t-s) \hatk(s) \, ds$. 
Together with 
the relations  (\ref{eq-inth1}) and (\ref{eq-inth2}) and 
the definition (\ref{def-calg}) of $\cendep$, 
it follows that almost surely, $\cendep$ is   absolutely continuous with respect to 
Lebesgue measure, and has density equal to 
\[\dfrac{d{\cal \cendep_t}}{dt} = \ops^{\hmeas_0}_t (h) - \hath_t (h)  
 + g(0) \hatk (t) + \int_0^t g^\prime (t-s) \hatk(s) \, ds. 
\]
The relation  (\ref{dercalg}) then follows on   
 comparing the right-hand side above with the right-hand side 
of the equation (\ref{def-hmeas}) for $\hmeas(f)$, setting $f = h$ therein 
and using the elementary relations  $h(0) = g(0)$ and  $g^\prime = h^\prime
(1-G) - hg$. 
\end{proof}

\begin{proof}[Proof of Theorem \ref{th-main1}]
From Lemma 
\ref{lem-inth} and the discussion prior to it, 
it follows  that $\hatx$ is a semimartingale 
with the decomposition stated  in Theorem  \ref{th-main1}. 
Combining the non-idling condition (\ref{eq-dnonidling2}) with 
the equation (\ref{eq-dprelimit3})  for $\hatk$, it follows that 
\be
\label{hat-smg}
 \hatk (t) = \left\{ 
\begin{array}{rl} 
\hate (t) & \mbox{ if } \fx \mbox{ is subcritical, } \\
\hate (t) + \widehat{x}_0 - \hatx (t) \vee 0 & \mbox{ if } \fx \mbox{ is critical, } \\
\hate(t) + \widehat{x}_0 - \hatx(t) & \mbox{ if } \fx \mbox{ is supercritical. } 
\end{array}
\right.
\ee
Thus, in the subcritical case, the semimartingale 
decomposition of $\hatk$ follows from that of 
$\hate$ (see Remark \ref{rem-asdiff1}), 
whereas in the supercritical 
case the semimartingale decomposition of $\hatk$ follows 
from those of $\hatx$ and $\hate$. 
On the other hand, when $\fx$ is critical we need the 
additional observation that by   Tanaka's formula,   
\be
\label{rep-xhatvee}
 \hatx (t) \vee 0  = \widehat{x}_0 \vee 0 + \int_{0}^t \ind_{\{\hatx(s) >
  0 \}} d\hatx_s + \dfrac{1}{2}  L_t^{\hatx} (0), 
\ee
where $L_t^{\hatx}(0)$ is the local time of $\hatx$ at zero, over 
the interval $[0,t]$.  
When combined with (\ref{hat-smg}), this provides 
the semimartingale decomposition of $\hatk$ in the critical 
case.  
When $\hatk$ is a semimartingale, the stochastic 
integration by parts formula for semimartingales shows that 
for every $f \in \acb$, 
\be
\label{def2-hcalk}
 \hcalk_s(f) = \int_{[0,s]} f(s-u) (1-G(s-u)) \, d\hatk(u), \quad 
s \geq 0, 
\ee
where the latter is the convolution integral with respect to 
the semimartingale $\hatk$.
Thus, we obtain  
(\ref{def-hmeas2}) from (\ref{def-hmeas}). 
\end{proof}

\begin{proof}[Sketch of Justification of Remark \ref{rem-main1}]
By Corollary \ref{cor-hreg}, if $f$ is bounded and H\"{o}lder continuous, 
then $\hathn(f) \Rightarrow \hath (f)$ in 
${\cal D}_{\R}[0,\infty)$ and 
$\{\hath_t(f), t \geq 0\}$ is a continuous 
process. We now argue that one can, in fact, show that 
$\hcalkn(f) \Rightarrow \hcalk(f)$ as $N \ra \infty$ for 
all H\"{o}lder continuous $f$. 
Given  the semimartingale
decomposition $K = M^K + A^K$, the integral on the 
right-hand side of the expression (\ref{def2-hcalk}) for 
$\hcalk(f)$ can 
be decomposed into a stochastic 
convolution integral with respect to the local martingale $M^K$ and a
Lebesgue-Stieltjes convolution integral with respect to the 
finite variation process $A^K$.  
An argument exactly analogous to the one used 
in Lemma \ref{lem-hathn}(1) to analyze $\hath(f)$ can 
then be used to analyze the stochastic convolution integral 
with respect to $M^K$ and a similar, though simpler, 
argument can be used to  study the convolution integral 
with respect to $A^K$ to show, as in Lemma \ref{lem-hathn} 
and Corollary 
\ref{cor-hreg},  
that for $f$ H\"{o}lder continuous and bounded, 
 $\hcalkn(f) \Rightarrow \hcalk(f)$ as $N \ra \infty$, and 
$\hcalk(f)$ admits a continuous version. 
When combined with 
the convergence in (\ref{jt-limit}), it is easy to argue as in the 
proof of Theorem \ref{th-main} that, 
in fact, the joint convergence  $(\ops^{\hmeasn_0}(f), \hcalkn(f), \hathn(f)) \Rightarrow 
(\ops^{\hmeas_0}(f), \hcalk(f), \hath(f))$ holds. 
Due to (\ref{rep-hatmeasn}), by the continuous mapping theorem,
 this implies that $\hmeasn(f) \Rightarrow
\hmeas(f)$ in ${\cal D}_{\R}[0,\infty)$, where $\hmeas(f)$ is continuous. 
\end{proof}


\subsection{Stochastic Age Equation}
\label{sec-sae}

The focus of this section is the characterization 
of the limiting state process in terms of a stochastic 
partial differential equation (SPDE), which we 
have called the  
stochastic age equation in Definition \ref{def-sae}.  
First, in Section \ref{subs-smgprop}  
we  establish a representation for integrals of functionals 
of the limiting centered age process $\{\hmeas_s, s \geq 0\}$. 
This representation is then used in Section \ref{subsub-veri} 
to show that $\{\hmeas_t, t \geq 0\}$ is a solution to the 
stochastic age equation associated with $(\hmeas_0,  \hatk, \chatm)$.  
The proof of uniqueness of solutions to the stochastic age equation 
and the proof of Theorem \ref{th-main2}(1) 
 is presented in Section \ref{subsub-unique}.  
Throughout the section we assume that the conditions of Theorem 
\ref{th-main2}, namely Assumptions \ref{as-flinit}, \ref{as-hate} and 
\ref{as-diffinit}', the conditions on the fluid limit
 and the boundedness and absolute continuity of 
$h$, are satisfied and state only additional assumptions when imposed.

\subsubsection{An Integral Representation}
\label{subs-smgprop}

We start by establishing an integral representation that 
results from the semimartingale property for $\hatk$.  
In what follows, recall the definition of the operator 
$\opint_t$ given in (\ref{def-opint}).  
Also, note that if $h$ is bounded then 
Assumption \ref{as-h} and (by Remark \ref{rem-holder}) Assumption
\ref{as-holder} are both satisfied.  
Since $h$ is also absolutely continuous,  by 
Theorem \ref{th-main1} $\hatk$ is a semimartingale  and 
$\hmeas_t(f)$ is given by (\ref{def-hmeas2}) for every 
$f \in \acbl$.

\begin{lemma}
\label{lem-intsa} 
For any  $\newf \in {\cal C}_b([0,\endsup) \times [0,\infty)$   
 such that $\newf(\cdot, t)$ is  
H\"{o}lder  continuous uniformly in $t$ and  absolutely continuous, 
$\P$-almost surely for every $t > 0$, we have 
\begin{eqnarray}
\label{eq-intsa}
\int_0^t \hmeas_s \left( \newf (\cdot, s) \right) \,  ds
& = &  \hmeas_0 \left( \Theta_t (\newf(\cdot,s))(\cdot,0)\right)  - \chatm_t (\opint_t \newf) \\
& & \nonumber 
 + \int_0^t \left( \int_u^t \newf(s-u,s) (1-G(s-u)) \, ds\right) \, d \hatk
 (u). 
 \end{eqnarray}
 \end{lemma}
\begin{proof}  
 Setting $t = s$ and  
$f = \newf (\cdot, s)$  in (\ref{def-hmeas2}), then using
the identities $\ops^{\hmeas_0}_s = \hmeas_0 (\Phi_s \cdot)$, 
$\hath_s = \chatm_s(\Psi_s \cdot)$ and (\ref{def2-hcalk}) 
and lastly integrating 
over $s \in [0,t]$, we obtain 
 \begin{eqnarray} 
\label{neweq1}
 \int_0^t \hmeas_s \left( \newf (\cdot, s) \right) \,  ds & = &  
\ds \int_0^t  \hmeas_0 \left(\Phi_s\newf(\cdot,s) \right) \, ds 
- \int_0^t \chatm_s(\Psi_s \newf(\cdot,s)) \, ds  \\
\nonumber & & \qquad +  
\int_{[0,t]} \left(\int_{[0,s]} \newf(s-u,s) (1-G(s-u)) \, d\hatk(u) \right) \, ds. 
\end{eqnarray}
From the definition (\ref{def-opint}) of $\opint_t$ and the fact that 
$(\Psi_t f)(\cdot,0) = \Phi_t f (\cdot)$, 
it follows that 
\[  \Theta_t \left(\varphi(\cdot, s) \right) (x, 0) = \int_0^t \left( \Psi_s
  \varphi(\cdot, s) \right) (x,0) \, ds = \int_0^t \Phi_s \left(
  \varphi(\cdot,s) \right) (x) \,ds. 
\]
Together with 
 Assumption \ref{as-diffinit}'(d), this implies that 
\be
\label{neweq2}
 \int_0^t  \hmeas_0 \left(\Phi_s\newf(\cdot,s) \right) \, ds = 
\hmeas_0 \left( \Theta_t (\newf(\cdot,s)\right)(\cdot,0),
\ee
which shows that the first terms on the right-hand sides of (\ref{eq-intsa})
and (\ref{neweq1})  are equal. 
 The corresponding equality of the 
second terms on the right-hand sides of  (\ref{eq-intsa})
and (\ref{neweq1})  follows from  (\ref{fub-1}), whereas 
the equality of the third terms follows from 
Fubini's theorem for stochastic integrals 
with respect to semimartingales (see, for example, (5.17) of Revuz and Yor
\cite{revyorbook}). 
  This completes the  proof of the lemma. 
\end{proof}

\subsubsection{A Verification Lemma}
\label{subsub-veri}

We now show that the process $\hmeas$ of Theorem \ref{th-fclt} is a solution to the stochastic 
age equation. 
For this, it will 
be convenient to introduce the function $\psi_h$ defined 
as follows:
$\psi_h(x,t) \doteq \ds \exp (r_h(x,t))$ 
for $(x,t) \in
[0,\ubound) \times \R_+$, where \be \label{def-r}
r_h(x,t) \doteq \left\{ \ba{rl}
\ds - \int_{x-t}^x h(u) \, du & \ds \mbox{ if } 0 \leq t \leq x, \\
\ds - \int_{0}^x h(u) \, du & \ds \mbox{ if } 0 \leq x \leq t. \ea
\right. \ee
Since $h = g/(1-G)$, this implies that 
\be
\label{eq-psih}
\psi_h(x,t) = \left\{ \ba{rl}
\ds \dfrac{1-G(x)}{1-G(x-t)} & \ds \mbox{ if } 0 \leq t \leq x, \\
\ds  1 - G(x) & \ds \mbox{ if } 0 \leq x \leq t. \ea
\right. \ee
If $g$ is absolutely continuous, then $G$ is continuously differentiable 
and $\psi_h$  is  bounded, absolutely continuous and  satisfies 
\be
\label{pde-psil}
\dfrac{\partial \psi_{h}}{\partial x} + \dfrac{\partial
  \psi_{h}}{\partial t} = - h \psi_{h}
\ee
for a.e.\ $(x,t) \in [0,\endsup) \times \R_+$.
Furthermore, from the definition it is easy to see that  
 $\psi_{h}(0,s) = \psi_{h}(x,0) = 1$ and, 
for $(x,s) \in [0,\endsup) \times [0,\infty)$ and $u \in [0,s]$, 
\be
\label{psil-ids}
\ds \dfrac{\psi_{h} (x+s-u,s)}{\psi_{h} (x,u)}  =  
  \dfrac{1-G(x+s-u)}{1-G(x)} = 
\left\{ 
\begin{array}{rl}
 \dfrac{1-G(x+s)}{1-G(x)} & \mbox{ if } u = 0,  \\
 (1-G(s-u)) & \mbox{ if } x = 0. 
\end{array}
\right. 
\ee

\begin{prop}
\label{prop-saexi}
If $h$ is H\"{o}lder continuous, then the process  $\{\hmeas_t, t \geq 0\}$ defined by (\ref{def-hmeas2}) 
satisfies the stochastic age equation associated with  $\{\hmeas_0, \hatk, \chatm\}$. 
\end{prop}
\begin{proof}
 Theorem \ref{th-main2} shows that for every $t > 0$, 
$\{\hmeas_t(f), f \in \acbl\}$ is a family of $\widehat{{\cal
    F}}_t$-measurable random variables and 
 $\{\hmeas_t, t \geq 0\}$ admits a version as  an $\{\widehat{{\cal F}}_t\}$-adapted
continuous, $\H_{-2}$-valued process. 
Moreover, it follows from Lemma \ref{lem-consistency} that for every $f \in
\acbl$, almost surely $s \mapsto \hmeas_s(f)$ is measurable.  
Therefore, it only remains to show that $\hmeas$ satisfies 
the equation (\ref{eq-sae}). 
  Fix $t \in [0,\infty)$ and 
$\newf \in {\cal C}_b^{1,1}([0,\endsup)\times [0,\infty))$ 
such that $\newf_x(\cdot,s) + \newf_s(\cdot,s)$ is Lipschitz continuous for every $s$. 
Since $h$ is bounded, H\"{o}lder continuous and absolutely 
continuous, it follows that $\newf_x + \newf_s - h \newf$ 
is bounded, H\"{o}lder continuous and absolutely continuous. 
Moreover, it is clear from (\ref{pde-psil}) that   
\[
 \left(\newf_x  + \newf_s  -  h \newf \right) \psi_{h} =
(\newf \psi_{h})_x + (\newf\psi_{h})_s. 
\]
Substituting this and the identity (\ref{psil-ids}) 
into the definition 
(\ref{def-opint}) of $\opint_t$, it follows that 
\begin{eqnarray} 
\nonumber
\left(\opint_t(\newf_x + \newf_s - h \newf)\right)(x,u)   & = &  
\int_u^t \dfrac{\left( (\newf_x + \newf_s - h
    \newf)\psi_h\right)(x+s-u,s)}{\psi_h(x,u)} \, ds \\
\nonumber
& = & \int_u^t
\dfrac{\left( (\newf \psi_{h})_x + (\newf \psi_{h})_s \right)(x+s-u,s)}{\psi_h(x,u)} \, ds \\
\label{rel-theta}
& = &  \dfrac{\newf (x+t-u,t) \psi_{h}(x+t-u,t)}{\psi_h(x,u)}
- \newf(x,u).
\end{eqnarray} 
Applying Lemma \ref{lem-intsa} 
with $\newf$ replaced by $\newf_x + \newf_s - h\newf$, 
 using (\ref{rel-theta}) and 
the identity $\psi_{h}(0,u) = 1$,
it follows that
\be
\label{neweq4}
\begin{array}{l}
\ds \int_0^t \hmeas_s \left(\newf_x(\cdot,s) + \newf_s(\cdot,s)
- h \newf(\cdot,s)\right) \, ds \\ 
\ds \qquad = 
\hmeas_0 \left( \newf(\cdot + t,t) \psi_{h}(\cdot + t, t)\right) 
+ \int_{[0,t]} \newf(t -u,t) \psi_{h}(t-u, t) d \hatk (u) \\
\ds \qquad \qquad - \underset{[0,\ubound) \times [0,t]}{\int \int} \newf (x+t-u,t)
\dfrac{\psi_{h}(x+t-u,t)}{\psi_{h} (x,u)} \, \chatm (dx, du)
 \\
\ds \qquad \qquad - \hmeas_0( \newf(\cdot, 0))  -  
\int_{[0,t]} \newf(0,u) \, d\hatk (u) +
\underset{[0,\ubound) \times [0,t]}{\int \int} \newf (x,u)  \, \chatm (dx, du).
\end{array}
\ee
Since $\newf$ is bounded and $x \mapsto \newf(x,s)$ is absolutely continuous
for every $s$,   
by the definition (\ref{def-hmeas2})  of $\hmeas_t$ and the identities 
in (\ref{psil-ids}), it is clear that the sum of the
first three terms on the right-hand side of (\ref{neweq4}) equals $\hmeas_t
(\newf(\cdot,t))$.   With this substitution, (\ref{neweq4}) reduces 
to the stochastic age equation  (\ref{eq-sae}).  
This completes the proof that $\{\hmeas_t,t \geq 0\}$ is a solution 
to the stochastic age equation 
associated with $(\hmeas_0, \hatk, \chatm)$. 
\end{proof}

\subsubsection{Uniqueness of Solutions to the Stochastic Age Equation}
\label{subsub-unique}

In order to establish uniqueness, we begin with a basic
``variation of constants'' transformation result.  
Recall from Section \ref{subsub-fun} 
that ${\cal S}_c$ is the space of ${\cal C}^\infty$ functions with 
compact support 
on $[0,\endsup)$ equipped with the same norm as ${\cal S}$. 
In what follows $g^\prime$ is the density of $g$. 

\begin{lemma}
\label{lem-ste}
Suppose that $g^\prime \in \L^2_{loc}[0,\endsup) \cup
\L^{\infty}_{loc}[0,\endsup)$.   
Given a solution $\{\meas_t, t \geq 0\}$ 
to the stochastic age equation associated with 
$(\hmeas_0, \hatk, \chatm)$, define 
\be
\label{def-mu}
 \mu_t (\tilde{f}) \doteq \meas_t(\tilde{f} (1-G)^{-1}), \qquad 
\tilde{f} \in \testspace_c. 
\ee
Then $\{\mu_t, t \geq 0\}$ is a continuous $\dualspace_{c}$-valued 
process that satisfies the following stochastic transport equation 
associated with $(\hmeas_0, \hatk, \chatm)$:  
for every $\tilde{f} \in \dualspace_c$, $t \geq 0$, 
\begin{eqnarray}
\label{eq-ste}
 \qquad  \mu_t (\tilde{f})   &=&
\hmeas_0 (\tilde{f}(1-G)^{-1}) +   \int_{0}^t
\mu_s( \tilde{f}_x) \, ds
 + \tilde{f}(0) \hatk(t) - \chatm_t \left(\tilde{f} (1-G)^{-1}\right). 
 \end{eqnarray}
\end{lemma}  
\begin{proof} By the definition of the stochastic age equation, 
$\{\meas_t, t \geq 0\}$ is 
a continuous $\H_{-2}$-valued process.  
By the assumptions on the service distribution $G$ and Lemma \ref{lem-asmarkov}, it follows 
that $f = \tilde{f} (1-G)^{-1}$ is an absolutely continuous function 
with compact support and hence lies in  $\H_2$. 
Therefore, $\mu_t(\tilde{f})$ is a well defined random variable for every 
$f \in \testspace_c$, $t > 0$.   Moreover,  $f$ has  derivative 
\be
\label{eq-der}
f_x = \tilde{f}_x(1-G)^{-1} + h f. 
\ee
Since $f \in {\cal C}_b^1[0,\endsup)$, we can 
substitute $\newf =f$ in the stochastic
age equation (\ref{eq-sae}),  use (\ref{eq-der}) and the identity 
$1-G(0)=1$ to obtain for $t\geq 0$, 
\begin{eqnarray}
\label{eq-mut}
\qquad \quad \mu_t (\tilde{f})  =   \meas_t \left(\tilde{f}(1-G)^{-1} \right)
 & = &  \meas_0 (\tilde{f}(1-G)^{-1}) + \int_0^t \meas_s
\left(\tilde{f}_x(1-G)^{-1}  \right) \,
ds \\
\nonumber & & \quad + \tilde{f}(0) \hatk(t) - \chatm_t ( \tilde{f}(1-G)^{-1}). 
\end{eqnarray}
Due to the continuity of $\hatk$ and $\chatm$, it follows that 
 the right-hand side is continuous in $t$, which in turn implies 
that  
$t \mapsto \mu_t(\tilde{f})$ is continuous for each $\tilde{f} \in \testspace_c$.
Since $\dualspace_c$ is a Fr\'{e}chet nuclear space,
  by Mitoma's theorem $\mu$ is a 
 continuous $\dualspace_c$-space valued process.
Moreover, by (\ref{def-mu}) $\meas_s \left(\tilde{f}_x(1-G)^{-1}\right) =
\mu_s (\tilde{f}_x)$. Substituting this back into (\ref{eq-mut}), it
follows that $\{\mu_t, t \geq 0\}$ satisfies  (\ref{eq-ste}).
\end{proof}

 We can now wrap up the proof of Theorem \ref{th-main2}. 

\begin{proof}[Proof of Theorem \ref{th-main2}] 
By assumption, $h$ is H\"{o}lder
  continuous. Therefore, 
 Proposition \ref{prop-saexi} shows that 
$\{\hmeas_t, t \geq 0\}$ is a solution to the stochastic age equation 
associated with $(\hmeas_0, \hatk, \chatm)$.  Thus, in order to establish the
theorem, it suffices to show that the
stochastic age equation has a unique solution.
Suppose that the stochastic age equation 
associated with $(\hmeas_0, \hatk, \chatm)$ 
 has two 
solutions $\meas^{(1)}$ and $\meas^{(2)}$ and for $i = 1, 2$,  
let $\mu^{(i)}$  be the corresponding continuous $\dualspace_c$-valued process   
 defined as in (\ref{def-mu}), but with $\meas$ replaced
by $\meas^{(i)}$. 
By Lemma \ref{lem-ste}, each $\mu^{(i)}$ satisfies
the stochastic transport equation (\ref{eq-ste}) associated
with  $(\hmeas_0, \hatk, \chatm)$.  Define 
$\eta \doteq
\mu^{(1)} - \mu^{(2)}$.  It follows that for every
$\tilde{f}$ in $\testspace_c$,
 \[ \frac{d}{dt}\lan \tilde{f},\eta_t\ran- \lan \tilde{f}_x,\eta_t\ran=0, \qquad \lan \tilde{f},
 \eta_0 \ran = 0. \]
However, this is simply a deterministic transport equation and it is well
known that  the unique solution to this equation is the identically zero solution $\eta
 \equiv 0$  (see, for example, Theorem 4 on page 408 of 
\cite{evansbook}).

Thus, for every 
$\tilde{f} \in \testspace_c$,  $\mu^{(1)}(\tilde{f})= \mu^{(2)}( \tilde{f})$  
or equivalently, 
\be
\label{meas-eqspde}
 \meas^{(1)}_t (\tilde{f}(1-G)^{-1}) =  \meas^{(2)}_t (\tilde{f}(1-G)^{-1}). 
\ee 
Now for any $f \in \testspace_c$, $f(1-G) \in \H_2$. 
Since $\testspace_c$ is dense in $\H_2$ (see Theorem 3.18 on page 54 of 
\cite{adamsbook}),  there exists a sequence $\tilde{f}_n \in 
\testspace_c$ such that $\tilde{f}_n \ra f(1-G)$ in $\H_2$ as 
$n \ra \infty$.  Replacing $\tilde{f}$ by $\tilde{f}_n$ in 
(\ref{meas-eqspde}) and then letting $n \ra \infty$, 
it follows that 
$\meas^{(1)}_t(f) = \meas_t^{(2)}(f)$ for every $f \in \testspace_c$.  
Again using the fact that $\testspace_c$ is dense in $\H_{2}$, 
this shows $\meas^{(1)}_t$ and $\meas^{(2)}_t$ are indistinguishable as 
$\H_{-2}$-valued elements.  This 
proves uniqueness of 
solutions to the stochastic age equation and the theorem follows.   
\end{proof}

\subsection{The Strong Markov Property}
\label{subs-pfmain2}

To prove the strong Markov property, we first 
 show that the assumptions on the initial centered 
age distribution  imposed in Assumption \ref{as-diffinit} 
are consistent, in the sense that they imply that these  
assumptions are also satisfied at any future time $s  > 0$. 
We define the following 
shifted processes: for $F = \haten, \hatkn,  \hate$,  
$\hatk$, and ${\cal U} = \chatmn,  \chatm$, 
and  $s \geq 0$, $u \geq 0$, 
\be
\label{def-shift1}
 (\Theta_s F) (u) \doteq F(s+u) - F(s), \qquad 
(\Theta_s {\cal U})_u \doteq {\cal U}_{s+u} - {\cal U}_s, 
\ee
 for $f \in {\cal C}_b[0,\endsup)$, we define   
\be
\label{def-shift2}
 (\Theta_s \hathn)_t (f) \doteq (\Theta_s \chatmn)_t (\Psi_t f), 
\qquad (\Theta_s \hath)_t (f) \doteq (\Theta_s \chatm)_t (\Psi_t f),
\ee
\begin{eqnarray}
\label{def-shift3}
(\Theta_s \hcalkn)_t (f) & = &  \int_{[0,t]}(1-G(t-u))f(t-u)\, d(\Theta_s\hkn)
(u) \\
\label{def-shift4}
(\Theta_s \hcalk)_t (f) &\doteq & 
f(0) (\Theta_s \hatk)(t) +  \int_0^t(\Theta_s\hatk)(u)\newf_f(t-u)\, du, 
\end{eqnarray}
and, in analogy with (\ref{rel-ops}), for $f \in \acbl$ we define 
\be
\label{def-opss}
\ops^{\hmeas_s}_t(f) \doteq \hmeas_s (\Phi_t f),  \qquad s, t \geq 0. 
\ee

\begin{lemma}
\label{lem-consistency} 
For every bounded and continuous $f$, 
\begin{eqnarray}
\label{eq-consist0}
\hmeasn_{s+t}(f)&=&\ops_t^{\hmeasn_s}(f) + (\Theta_s \hcalkn)_t (f) - 
(\Theta_s\hathn)_t (f), \qquad s, t \geq 0. 
\end{eqnarray}
Likewise, if Assumptions \ref{as-flinit}--\ref{as-diffinit} hold and 
$g$ is continuous, then for every bounded and absolutely continuous $f$, 
\begin{eqnarray}
\label{eq-consist00}
\hmeas_{s+t}(f)&=&\ops_t^{\hmeas_s}(f) + (\Theta_s \hcalk)_t (f)
- (\Theta_s \hath)_t (f),  \qquad s, t \geq 0. 
\end{eqnarray}
In addition, for every $s > 0$, 
\be
\label{eq-consist02}
(\Theta_s \hatk,  \hatx_{s+\cdot}, \hmeas_{s+\cdot}(\f1)) = \Lambda(\Theta_s
\hate, \hatx(s), \ops^{\hmeas_s}(\f1)- (\Theta_s \hath)(\f1)).
\ee
Furthermore, for every $s > 0$, Assumption \ref{as-diffinit}' holds with 
the sequence $\{\hmeasn_0\}_{N \in \N}$ and limit $\hmeas_0$, respectively, 
 replaced by $\{\hmeasn_s\}_{N \in \N}$ and $\hmeas_s$. 
\end{lemma}

We defer the proof of this lemma to Appendix \ref{sec-consistency}, and 
instead now prove the strong Markov property of the state process.

\begin{proof}[Proof of Theorem \ref{th-main2}(2)] Fix $s, t > 0$. 
First, note that by Theorem \ref{th-fclt} it follows that $(\hatx, \hmeas)$
is an $\R \times \H_{-2}$-valued process. 
Moreover, by 
Lemma \ref{lem-consistency},  Assumption \ref{as-diffinit} 
is satisfied with $\hmeas_0$ replaced by $\hmeas_s$, 
which in particular implies that  
the random element $\ops^{\hmeas_s}_{t+\cdot} (\f1) = 
\{\ops^{\hmeas_s}_{t+u}(\f1), u \geq 0\}$ almost surely 
takes values in ${\cal C}_{\R}[0,\infty)$.  
Also, let $(\Theta_s \hath)_t (\Phi_{\cdot}\f1)$ represent 
the process $\{(\Theta_s \hath)_t (\Phi_{u}\f1), u \geq 0\}$.  
Writing $\Phi_u \f1 = \int_0^u \Phi_r h(\cdot) \, dr$ 
and observing that $\Phi_u \f1$ is bounded and 
(due to Assumption \ref{as-holder})   H\"{o}lder continuous 
uniformly in $u$, it follows from 
   Lemma \ref{lem-hathn}(3) (with $\chatm$ replaced by $\Theta_s \chatm$) 
that  the random element $(\Theta_s \hath)_t (\Phi_{\cdot}\f1)$ takes 
values in ${\cal C}_{\R}[0,\infty)$. 
In addition, Assumption \ref{as-hate} and 
 Corollary \ref{cor-hreg} show that
$\Theta_s \hate$ and $\Theta_s \hath$ are, respectively,  
 ${\cal C}_{\R}[0,\infty)$-valued and 
${\cal C}_{\H_{-2}}[0,\infty)$-valued. 
 We now claim that there exists a continuous mapping from 
$\R \times \H_{-2} \times {\cal C}_{\R}[0,\infty)^3 \times {\cal C}_{\H_{-2}}[0,\infty)$ to 
$\R \times \H_{-2} \times {\cal C}_{\R}[0,\infty)$, 
which we denote by $\tilde{\Lambda} = \tilde{\Lambda}_t$, such that 
$\P$-almost surely,  
\be
\label{rel-tildelam}
 (\hatx_{s+t}, \hmeas_{s+t}, \ops^{\hmeas_{s+t}}(\f1)) = 
\tilde{\Lambda} \left(\hatx(s), \hmeas_s, \ops^{\hmeas_s}(\f1), 
\Theta_s \hate, (\Theta_s \hath)_t (\Phi_{\cdot}\f1), \Theta_s \hath \right). 
\ee

To see why this is the case, first note that 
 equation (\ref{eq-consist00}) shows that $\hmeas_{s+t}$ is the
sum of $(\Theta_s \hath)_t$, $\ops^{\hmeas_s}_t$ and $(\Theta_s \hcalk)_t$ 
and, by   
Lemma \ref{lem-asmarkov}(2), $\ops^{\hmeas_s}_t$ is a continuous 
functional of $\hmeas_s$.   
Also, for $u > 0$, 
$\Phi_u \f1$ is bounded and absolutely continuous.  
Hence, by (\ref{eq-consist00}), the definition of $\ops^\meas$   
and the semigroup property for $\Phi$, 
 $\P$-almost surely,  for $u, s, t \geq 0$, 
\begin{eqnarray*}
 \ops^{\hmeas_{s+t}}_u(\f1)  & =  & 
\ops^{\hmeas_s}_t (\Phi_{u} \f1) + ( \Theta_s \hcalk)_t (\Phi_u \f1) 
-( \Theta_s \hath)_t (\Phi_u \f1)  \\
& = & \ops^{\hmeas_s}_{t+u} (\f1) + ( \Theta_s \hcalk)_t (\Phi_u \f1) 
- ( \Theta_s \hath)_t (\Phi_u \f1). 
\end{eqnarray*}
Next, note that 
 (\ref{eq-consist02}) of Lemma \ref{lem-consistency},  Proposition 
 \ref{prop-cont} and the continuity of $\hatx$ show that  
$\hatx_{s+t}$ is a continuous functional of 
$(\Theta_s \hate, \hatx(s), \ops^{\hmeas_s} (\f1)- \Theta_s \hath (\f1))$.  
Furthermore, due to the almost sure continuity of 
$\hatk$ established in Theorem \ref{th-main}, 
 definitions (\ref{eq-calk}) and (\ref{def-shift4}) of $\calk$ and 
$(\Theta_s \hatk)$, respectively,  
and properties 2 and 3 of  
Lemma \ref{lem-calk}, it follows that $\Theta_s \hcalk$ and 
$u \mapsto (\Theta_s \calk)_t (\Phi_{u} \f1)$ are almost surely 
obtained as continuous mappings of $\Theta_s \hatk$. 
In turn, by (\ref{eq-consist02}) of Lemma \ref{lem-consistency}, 
$\Theta_s \hatk = \Lambda (\Theta_s \hate, \hatx(s), \ops^{\hmeas_s} (\f1)- 
\Theta_s \hath(\f1))$, where $\Lambda$ is continuous by  
 Proposition \ref{prop-cont}.  When combined, the above observations 
show that the claim (\ref{rel-tildelam}) holds, with $\tilde{\Lambda}$ a 
suitable  continuous mapping.

We now show that the claim implies the  Markov property. First,  from  
(\ref{def-hate}) we observe that $\Theta_s \hate$ is adapted 
to the filtration generated by $\Theta_s B$ and, likewise, 
(\ref{def-shift2}) shows that 
$\Theta_s \hath$ is adapted to the filtration generated by 
$\Theta_s \chatm$. 
Moreover, by the definition of $B$ as a standard Brownian motion 
and the definition of $\chatm$ (see Section \ref{subs-scaledmmeas} and 
Remark \ref{rem-asdiff1}), both $B$ and $\chatm$ are 
 processes with independent increments with respect to the  
filtration $\{\widehat{{\cal F}}_t, t \geq 0\}$, which is the   
right continuous completion of 
$\{\sigma(\hmeas_0, \hatx(0)) \vee \sigma (B_s, \chatm_s, s \geq t), t \geq
0\}$. 
In particular, this implies that $\Theta_s B$, 
$\Theta_s \hath$ and $u \mapsto (\Theta_s \hath)_t(\Phi_u \f1)$ 
are independent of $\widehat{{\cal F}}_s$. 
Therefore, for  any bounded continuous function $F$ on 
$[0,\infty) \times (\R \times \H_{-2} \times {\cal C}_{\R}[0,\infty))$,  
\[
\begin{array}{l}
 \E[F(s, \hatx_{s+t}, \hmeas_{s+t}, \ops^{\hmeas_{s+t}}(\f1))|\widehat{{\cal F}}_{s}] \\
\qquad =   \E\left[F (s,\tilde{\Lambda} (  \hatx(s),
   \hmeas_s, \ops^{\hmeas_s}(\f1), \Theta_s \hate,
   (\Theta_s\hath)_t(\Phi_{\cdot}\f1), 
\Theta_s \hath(\f1)))| \widehat{{\cal F}}_s \right] \\
\qquad 
=   \E\left[F (s,\tilde{\Lambda} (  \hatx(s),
   \hmeas_s, \ops^{\hmeas_s}(\f1), \Theta_s \hate,
   (\Theta_s\hath)_t(\Phi_{\cdot}\f1), 
\Theta_s \hath(\f1)))| \hatx(s),
\hmeas_s,\ops^{\hmeas_{s}}(\f1)\right] \\
\qquad =  \E[F(s, \hatx_{s+t},
\hmeas_{s+t},\ops^{\hmeas_{s+t}}(\f1))|\hatx(s),
\hmeas_s,\ops^{\hmeas_{s}}(\f1)]. 
\end{array}
\]
This shows that $\{(\hatx_s, \hmeas_s, \ops^{\hmeas_s}(\f1)), \widehat{{\cal F}}_s, s \geq 0\}$ is a
Markov process. 

By Theorems \ref{th-main} and \ref{th-fclt}, the sample paths $s \mapsto 
(\hatx(s), \hmeas_s, \ops^{\hmeas_s} (\f1))$ taking values in the state 
space 
$\R \times \H_{-2} \times {\cal C}_{\H_{-2}}[0,\infty)$
 are continuous. 
Since the state 
space  is a  complete, separable metric space, 
there exists a Markov kernel 
$P:  \R_+ \times (\R \times \H_{-2}\times {\cal C}_{\R}[0,\infty)) \times \R_+ \times {\cal B}(\R \times \H_{-2}) 
\mapsto [0,1]$ such that for any $(x,\meas, \psi) \in  (\R \times
\H_{-2}\times {\cal C}_{\R}[0,\infty))$, and any measurable function
$F$ on $\R_+ \times (\R \times \H_{-2}\times {\cal C}_{\R}[0,\infty))$, 
\[
\begin{array}{l}
\ds \E[ F(s,\hatx_{s+t}, \hmeas_{s+t}, \ops^{\hmeas_{s+t}}(\f1))| (\hatx(s),
\hmeas_s, \ops^{\hmeas_s}(\f1)) = (x,\meas, \psi)]\\
\qquad \ds \qquad =  \int_{\R_+ \times \H_{-2}} F(s,u) P (s,(x,\meas, \psi), t, du). 
\end{array}
\]
Now, when $F$ is bounded
and continuous, it follows 
 from (\ref{rel-tildelam}) and the continuity of $\tilde{\Lambda}$
established above that the mapping from 
$(x,\meas, \psi)$ to the term on the right-hand side of 
 the last display is continuous. 
This implies that the Markov kernel is Feller, 
and it follows from Theorem 2.4   
of Friedman \cite{Friedmanbook}
that $\{(\hatx_s, \hmeas_s, \ops^{\hmeas_s}(\f1)), \widehat{{\cal F}}_s, s \geq 0\}$ 
is a strong Markov process 
(see also Theorem 7.4 of Chapter I of \cite{sharpebook}
for the time-homogeneous case, which would be applicable if  the 
arrival process satisfies Assumption \ref{as-hate}(a)  and the fluid age measure
starts at the equilibrium measure so that 
$\fmeas_s(dx) = (1-G(x)) \, dx$ and $\fx(s) = 1$, $s \geq 0$, leading 
to a critical fluid limit)   
\end{proof}

\begin{remark}
\label{rem-markov}
{\em   A more natural candidate for the  
(strong) Markov process would be the process 
$\{(\hatx_{t}, \hmeas_{t}), {\cal F}_t, t\geq 0\}$ taking
values in $\R \times \H_{-2}$.  
However,  in order to 
establish the Markov property, we need 
$\ops^{\hmeas_s}$ and $\ops^{\hmeas_s}(\f1)$ to be measurable 
functions of $\hmeas_s$.  As shown in Lemma 
\ref{lem-asmarkov}, the additional boundedness 
assumption on $g^\prime/(1-G)$) ensures 
that the map from $\H_{-2} \mapsto {\cal D}_{\H_{-2}}[0,\infty)$ that 
takes $\hmeas_s$ to $\ops^{\hmeas_s}$ is continuous.  
This is a reasonable assumption because, 
as noted in Remark \ref{cond-verify}, 
it is satisfied by a large class of distributions 
of interest.  
However, unfortunately, it appears that 
  measurability of the map from $\H_{-2}$ to ${\cal D}_{\R}[0,\infty)$ 
that takes $\hmeas_s$ to $\ops^{\hmeas_s} (\f1) = \hmeas_s(\Phi_s \f1)$, which would require 
that $\Phi_s \f1$ lies in $\H_2$, cannot be 
obtained without imposing too severe assumptions on the service 
distribution $G$.  Although $\Phi_s \f1 \in \acbl$ and 
$\{\hmeas_s(f), f \in \acbl\}$ 
is a well defined collection 
of random variables, 
  it is not clear whether it is possible to 
 show that  $\hmeas_s$ admits a 
version that takes values in the dual of some space 
that contains $\Phi_s \f1$  and such that the dual space 
 admits a regular conditional probability so as 
to enable the construction of the Markov kernel. 
Instead, we resolve this issue by  appending the ${\cal D}_{\R}[0,\infty)$-valued 
process $\ops^{\hmeas_s} (\f1)$ to the state descriptor. 
}
\end{remark}

\appendix

 \renewcommand{\theequation}{A.\arabic{equation}}
\renewcommand\thetheorem{A.\arabic{theorem}}

\beginsec

\section{Properties of the Martingle Measure Sequence}

\subsection{Proof of the Martingale Measure Property}
\label{sec-mms}

Recall that ${\cal B}_0[0,\endsup)$ is the algebra generated by
the intervals $[0,x]$, $x \in [0,\endsup)$. 
We now show that the collection of random variables
$\{\cmn_t(B); t \geq 0, B \in {\cal B}_0[0,\endsup)\}$ introduced 
in (\ref{def-cmn}) defines a  martingale measure. 

\begin{lemma}
\label{lem-martmeas}
For each $N \in \N$, $\cmn = \{\cmn_t(B), {\cal F}_t^{(N)};
t \geq 0, B \in {\cal B}_0[0,\endsup)\}$
is a martingale measure on $[0,\endsup)$.  Moreover, for every $B \in {\cal
  B}_0[0,\endsup)$ and $t \in [0,\infty)$,
\be
\label{l2norm}
\E\left[ \left(\cmn_t (B) \right)^2 \right] =
\E \left[ \int_0^t \left( \int_{B} h(x) \, \measn_s (dx) \, \right) \, ds \right].
\ee
\end{lemma}
\begin{proof}
In order to show that $\{\cmn_t(B); t \geq 0, B \in {\cal B}_0[0,\endsup)\}$
defines a martingale measure on $[0,\endsup)$, we verify the three properties
stated in the definition of a martingale  measure given on page 287 of Walsh 
\cite{walshbook}.   The first property in \cite{walshbook}, namely
that $\cmn_0(B) = 0$ for every $B \in {\cal B}_0$, follows trivially
from the definition.
Next, we verify the third property, which states
that  $\{\cmn_t(B), {\cal F}_t^{(N)},t \geq 0\}$
is a local martingale for each $B \in {\cal B}_0$.
For this, first observe that any $B \in {\cal B}_0[0,\endsup)$ is the
union of a finite number of disjoint intervals $I_i$, $i = 1, \ldots, k$,
where each $I_i$ is of the form $(\alpha_i, \endsup)$,
$(\alpha_i, \beta_i]$ or $[0,\beta_i]$, with $0 < \alpha_i < \beta_i < \endsup$.
For any such interval $I_i$, it is clear from the definition of the
age process given in (\ref{def-agejn}) that for every $j$, the function
$s \mapsto \ind_{I_i} (\agen_j(s))$ defines a bounded, left continuous
function on $[0,\infty)$.
In turn, since $\ind_B = \sum_{i=1}^k \ind_{I_i}$, clearly the function
$s \mapsto \ind_{B} (\agen_j(s))$ is also  bounded and
left continuous. By
Lemma 5.2  of Kang and Ramanan \cite{KanRam08}, it then follows that for every
$B \in {\cal B}_0 [0,\endsup)$,  $\{\cmn_t(B), t \geq 0\}$ is
an $\{{\cal F}_t^{(N)}\}$-martingale obtained as a compensated 
sum of jumps, where the compensator $\dcompn_{\ind_B}$ is continuous. 
Standard arguments (see, for example, 
 the proof of Lemma 5.9 in Kaspi and Ramanan \cite{KasRam07}) then show that 
$\lan \cmn(B)\ran$, the predictable quadratic variation of $\cmn(B)$, 
equals $\dcompn_{\ind_B}$.   Since $\E[\dcompn_{\ind_B}(t)]$ is dominated 
by $\E[\dn(t)]$, which is finite by Lemma 5.6 of Kaspi and Ramanan 
\cite{KasRam07}, 
the relation (\ref{l2norm})  follows.
On the other hand, because $\{\measn_t, t \geq 0\}$ 
is an $\mmf$-valued process, 
this shows that the set function $B \mapsto \E[ (\cmn_t (B) )^2 ]$
is countably additive on ${\cal B}_0[0,\endsup)$, and hence defines a finite
$\L^2 (\Omega, {\cal F}^{(N)}, \P)$-valued measure.  
This verifies the second property in \cite{walshbook}, and thus completes 
the proof of the lemma. 
\end{proof}

\subsection{Proof of Lemma \ref{lem-dep}}
\label{subs-lemdep}

As we will show below,
Lemma \ref{lem-dep} is essentially a consequence of the strong Markov
property of the state process, the continuity of the $\{{\cal F}_t\}$-compensator of
the departure process and the independence assumptions on the service
times and arrival process.

Fix  $N \in \N$ and, for conciseness, we suppress $N$ from the
notation.  We shall first prove (\ref{ineq1-dep}), namely we will show 
that almost surely, $\Delta D (t) \leq 1$ 
for every $t \in [0,\infty)$.  For $k = - \lan \f1, \nu_0 \ran + 1, \ldots$, 
let ${\cal E}_k$ denote the event that the departure time of customer $k$ 
lies in the set of the union of departure times of customers $j$, $j < k$.
To establish (\ref{ineq1-dep}), it is clearly sufficient to show that 
$\P({\cal E}_k) = 0$ for every $k$. 
Fix $k \in \N$ and let $\kinv_k$ be the $\{{\cal F}_t\}$-stopping time
\[ \kinv_k \doteq \inf\{t: K(t)=k\}. \]
Now, consider a modified system with initial data $\tilde{\nu}_0 =
\nu_{\kinv_k}$, $\tilde{X}(0) = \lan \f1, \nu_{\kinv_k} \ran$ and $\tilde{E}
\equiv 0$.  
 By  Lemma B.1 of Kang and Ramanan \cite{KanRam08}, 
$\{(R_E(t), X(t), \nu_t),t \geq 0\}$ is a strong
Markov process.  Therefore, 
conditioned on ${\cal F}_{\kinv_k}$,
the departure times of customers $j$, $j \leq k$, only depend on
$\{a_j(\kinv_k), j \leq k\}$ and are independent of  arrivals after 
$\kinv_k$. 
Consequently, the probability of the
event ${\cal E}_k$ is equal in both the  original and modified systems.
In the modified system, 
 let $\{\tilde{a}_j(s), s \in [0,\infty)\}$  
denote the age process of customer $j$ for $j \leq k$, 
let $\tilde{D}^{\kinv_k}(s)$ denote the cumulative departures in the time $[0,s]$
of all customers  other than customer $k$ and 
let $\tilde{J}^k \doteq \{s \in [0,\infty): \tilde{D}^{\kinv_k}(s) \neq
\tilde{D}^{\kinv_k} (s-) \}$ be the jump times of $\tilde{D}^{\kinv_k}$.
Also,
let $\tilde{{\cal G}}_t^k \doteq \sigma(\tilde{a}_j(s), j < k, s \in
[0,t])$ and let $\{{\cal G}_t^k,t \geq 0\}$ be the right continuous completion
(with respect to $\P)$ of $\{\tilde{{\cal G}}_t^k,t \geq 0\}$.
By the assumed independence of the service times for different
customers and the fact that  $\tilde{a}_k(0) = 0$,  the departure
time $\tilde{v}_k$ of customer $k$ in the modified system
has cumulative distribution function $G$ and is independent of $\tilde{J}^k$.
Therefore, 
\be
\label{rel-jump}
\begin{array}{rclcl} 
\ds \P({\cal E}_k) & = & \ds \P(\tilde{v}_k \in \tilde{J}^k) \\
& = & \ds \int_{[0,\endsup)} \P(t \in \tilde{J}^k|\tilde{v}_k = t) \, dG(t) 
& = & \ds \int_{[0,\endsup)} \P(t \in \tilde{J}^k) \, dG(t),
\end{array}
\ee
where the last equality follows from the independence of $\tilde{v}_k$ and
$\tilde{J}_k$.  The same logic used in Lemma 5.4 of Kaspi and 
Ramanan \cite{KasRam07} to identify
the compensator of $D$  also shows that the $\{{\cal G}^k_t\}$-compensator of
$\tilde{D}^{\kinv_k}$ equals
\[  \int_0^\cdot \left( \int_{[0,\endsup)} \dfrac{g(x+s)}{1 - G(x)}
  \nu'_0 (dx) \right) \, ds, \qquad \mbox{ where } \nu'_0 \doteq
\tilde{\nu}_{0} - \delta_0, \]
where the mass at zero is deleted from the modified age measure $\tilde{\nu}_0$ 
to remove customer $k$, which has age zero at time
$0$ in the modified system. 
By the continuity of the $\{{\cal G}_t^k\}$-compensator of
$\tilde{D}^{\kinv_k}$, $\tilde{D}^{\kinv_k}$ is
quasi-left-continuous and so $\Delta \tilde{D}^{\kinv_k}(T) = 0$ for every
$\{{\cal G}_t^k\}$-predictable time $T$
(see, for example, Theorem 4.2 and Definition 2.25 of Chapter I of Jacod and Shiryaev 
\cite{JacShiBook}).
Choosing $T$ to be the deterministic time $t$, this implies that $\P(t \in
\tilde{J}^k) = 0$ for every $t \geq 0$.  When
substituted into (\ref{rel-jump}), this shows that $\P({\cal E}_k) = 0$. For 
$k \leq 0$,   we set  
$\kinv_k = 0$ and observe  that, conditioned on ${\cal F}_0$, the
departure time $\tilde{v}_k$ of the $k$th customer has 
cumulative distribution function 
$\tilde{G}(\cdot) \doteq (G(\cdot) - G(a_k(0))/(1-G(a_k(0))$, rather than $G$, 
so that (\ref{rel-jump}) holds with $G$ replaced by $\tilde{G}$.  
The rest of the proof follows exactly as in the case $k > 0$, and thus
 (\ref{ineq1-dep}) holds.

We now turn to the proof of (\ref{ineq2-dep}).
Fix $r, s \in [0,\infty)$, recall that $D^r(s)$ is the cumulative departures 
in the interval $[r,r+s)$ of customers that entered service at or before 
time $r$,   define $J^r$ to be the jump
times of $D^r$ in $[0,\infty)$ and let ${\cal G}_t = {\cal F}_{r+t}$, $t \in [0,\infty)$.
Using the same logic as in the proof of
(\ref{ineq1-dep}), it can be shown that $\{D^r(t),t \geq 0\}$ has
a continuous $\{{\cal G}_t\}$-compensator,  given explicitly by
\[  \int_{0}^t \left(  \int_{[0,\endsup)} \dfrac{g(x+s)}{1-G(x)} \nu_{r}(dx)
\right) \,ds,   \quad t \in [0,\infty), 
\]
and hence has no fixed jump times, i.e.,
 $\P(t \in J^r|{\cal F}_r) = 0$ for every $t \in \R_+$.
Moreover, due to the assumption of independence of the arrival processes
and the service times,  $\{E(t)\}_{t \geq r}$ and
$\{D^{r}(t),t \geq 0\}$ are conditionally independent, given $\cal{F}_{r}$.
Let
\[  {\cal T} \doteq \{\bar{t} = (t_1, \ldots, t_m, \ldots) \in \R_+^\infty:
0\leq t_1 \leq
t_2 \leq ...\}, \]
and let $\overline{T}^r$ denote the random ${\cal T}$-valued sequence of  times
after $r$ at which $E$ has a jump.
Moreover, let $\mu$ denote the conditional probability distribution of
$\overline{T}^r$,  
given ${\cal F}_r$.
Then, for any $\overline{t} \in {\cal T}$, using the fact that $0 \leq \Delta E (t)
\leq 1$, we have
\[
\begin{array}{rclcl}
\ds \E\left[\sum_{s \in [0,\infty)} \Delta E(r+s)\Delta D^{r}(s)|{\cal
      F}_{r}, \overline{T}^r = \overline{t} \right] & = &
 \ds \sum_{t \in \overline{t}} \E[ \Delta D^r(t)|{\cal F}_{r}, \overline{T}^r = \overline{t}]
\\
& = &  \ds\sum_{t \in \overline{t}} \E[ \Delta D^r(t)|{\cal F}_r] \\
& = & \ds  0,
\end{array}
\]
where the second equality uses the conditional independence of $\{D^r(t),t \geq 0\}$
from $\overline{T}^r$ given ${\cal F}_r$, and the last equality follows
because as argued above, conditional on ${\cal F}_r$, $D^r$ almost surely has no fixed jumps.
In turn, integrating the left-hand side above with respect to the conditional
distribution
$\mu$ and then taking expectations, it follows that
\begin{eqnarray*}
\E\left[\sum_{s \in [0,\infty)} \Delta E(r+s) \Delta D^{r}(s)\right]
 =   0.
\end{eqnarray*}
Since the term inside the expectation is non-negative, this proves
(\ref{ineq2-dep}).

\subsection{A Consequence of Lemma \ref{lem-dep}}
\label{subs-cordep}

We now establish a consequence of Lemma \ref{lem-dep}, 
which will be used in the proof of the asymptotic 
independence property in Section \ref{subs-asind}. 

\begin{cor}
\label{cor-dep} As $N \ra \infty$, 
 \[\dfrac{1}{N} \E\left[\suli_{s\leq 
t}\Delta\en (s) \Delta\dn(s)\right] \ra 0. 
\]
\end{cor}
\begin{proof}
With the aim of computing the left-hand side 
above, using the same notation as in Lemma \ref{lem-dep}, 
for $r > 0$ and $s \geq r$, let $D^{(N),r}(s)$ denote 
the cumulative number of departures during $(r,s]$  of 
 customers that entered service at or before time $r$, 
and let $D^{(N)+,r}(s)$ be the cumulative number 
departures during $(r,s]$ of 
 customers that have  entered service  after
time $r$. Then for $\delta >0$ and $k=1,2, \ldots$, we have
\begin{eqnarray*} 
\qquad \sum_{s\in (k\delta,(k+1)\delta]}\Delta\en(s) \Delta\dn(s) &
= & \sum_{s\in(k\delta,(k+1)\delta]}\Delta\en (s) \Delta
D^{(N),k\delta} (s) \\ 
& &  + \sum_{s\in(k\delta,(k+1)\delta]}\Delta\en(s)\Delta
D^{(N)+,k\delta} (s)
\end{eqnarray*}
The first summand on the right-hand side above is almost surely 
zero by (\ref{ineq2-dep}) of
Lemma \ref{lem-dep}.  Using the
 fact that $\en$ has unit jump sizes to 
bound the second term, we  obtain 
\begin{eqnarray}
\label{sq-var2}
 \sum_{s\in (k\delta,(k+1)\delta]}\Delta\en(s) \Delta\dn(s) & \leq &
 \sum_{s\in(k\delta,(k+1)\delta]} \Delta D^{(N)+,k\delta} (s) \\
\nonumber & \leq &
\sum_{j=\kn(k\delta)+1}^{\kn((k+1)\delta)}\ind_{\{v_j\le
  \delta\}}. 
\end{eqnarray}
Summing (\ref{sq-var2}) over $k=1,\ldots, \lfloor t/\delta\rfloor$ 
and dividing by $N$, we obtain 
\[
\dfrac{1}{N} \E \left[\suli_{s\leq 
t}\Delta\en (s) \Delta\dn(s)\right]  \leq  
\E\left[\int_0^t\ind_{\{v_{\kn(s)}\leq 
  \delta\}}d\fkn(s)\right]  
 \leq  \E\left[\fbaren_{\ind_{[0,\delta]}}(t+\delta)\right]. 
\]
For each $\delta > 0$, let $f_{\delta}$ be any continuous bounded
function on $[0,\endsup)$ such that $\ind_{[0,\delta)} \leq
f_{\delta} \leq \ind_{[0,2\delta)}$. Then we have 
\[ 
\dfrac{1}{N} \E \left[\suli_{s\leq 
t}\Delta\en (s) \Delta\dn(s)\right]
  \leq 
\E\left[ \fbaren_{f_{\delta}}(t+\delta)\right]
=\E\left[\fdcompn_{f_{\delta}}(t+\delta)\right],
\]
 On both sides, taking first the limit supremum 
 as $N \ra \infty$,  and then the limit as $\delta \downarrow 0$, 
we then conclude that 
\[ \limsup_{N \ra \infty} \dfrac{1}{N} \E \left[\suli_{s\leq 
t} \Delta\en (s) \Delta\dn(s)\right] \leq 
\lim_{\delta \downarrow 0} \limsup_{N \ra \infty}
\E\left[\fdcompn_{f_{\delta}}(t+\delta)\right] = 0, 
\]
were the last equality follows from 
 Lemma 5.8(3) of Kaspi and Ramanan
\cite{KasRam07}. 
\end{proof}

 \renewcommand{\theequation}{B.\arabic{equation}}
\renewcommand\thetheorem{B.\arabic{theorem}}

\beginsec

\section{Ramifications of  Assumptions on the Service Distribution}
\label{app-holder} 

\begin{lemma}
\label{lem-asmarkov} 
Suppose $h$ is uniformly bounded.  Then 
Assumptions \ref{as-h} and \ref{as-holder} are satisfied. 
If, in addition,  $g$ is absolutely continuous and 
$g^\prime$ is either locally essentially bounded or 
$g^\prime \in \L_2^{loc}$.   Then for any $\tilde{f} \in {\cal S}_c$, 
$f = \tilde{f}(1-G)^{-1} \in \H_2$.  Moreover, 
 if  $g$ is absolutely continuous and $g^\prime/(1-G)$ is 
 bounded then  $f \in \H_2$  implies 
$\Phi_t f \in \H_2$ for every $t \geq 0$ and for every $t > 0$, 
  the mapping from $\H_{-2}$ to $\H_{-2}$ that takes 
  $\nu \mapsto S^\nu_t = \nu (\Phi_t \cdot)$ 
is Lipschitz continuous.  
\end{lemma}
\begin{proof}
If $h$ is uniformly bounded, then 
Assumption \ref{as-h} is trivially satisfied
and 
\[ \dfrac{G(x+y) - G(x+\tilde{y})}{1-G(x)} = 
\int_y^{\tilde{y}} \dfrac{g(x+u)}{1-G(x+u)} \dfrac{1-G(x+u)}{1-G(x)} \, du 
\leq \nrm{h}_{\infty} |y-\tilde{y}|, 
\]
which shows that Assumption \ref{as-holder} is satisfied 
with $C_G = \nrm{h}_{\infty}$ and $\gamma_G = 1$.

Now, suppose  that in addition,  $g$ is absolutely continuous and 
$g^\prime$ is either locally essentially bounded or 
$g^\prime \in \L_2^{loc}$.  
If $f = \tilde{f} (1-G)^{-1}$ then $f^\prime = \tilde{f}_x (1-G)^{-1} + hf$ 
and $f^{\prime \prime} = \tilde{f}_{xx} (1-G)^{-1} + 2 h (\tilde{f}_x + f_x) 
+  f h^2  + f g^\prime(1-G)^{-1}$.  Since $g$ is absolutely continuous, 
$f$ and $f^\prime$ and the first three terms in the expansion of 
$f^{\prime \prime} $ are continuous with compact support and hence in 
$\L_2^{loc}$.  In addition, because $f(1-G)^{-1}$ is  
continuous with compact support, 
the last term lies in $\L_2^{loc}$ if either $g^\prime$ 
is locally essentially bounded or, by Cauchy-Schwarz, 
 if $g^\prime$  lies in $\L_2^{loc}$. 

Now, suppose that 
 $g$ is absolutely continuous and $g^\prime/(1-G)$ is 
 bounded.  Fix $t \geq 0$ and $f \in \H_2$.  For notational conciseness, let 
\[ r(x) \doteq r_t(x) \doteq \dfrac{1-G(x+t)}{1-G(x)}, \qquad x \in [0,\endsup). \]
Then, by the definition (\ref{def-thetat}) of $\Phi_t$,  for $x \in [0,\endsup)$, 
\begin{eqnarray*}
(\Phi_t f)(x) & = & r(x) f(x+t), \\
  (\Phi_t f)^\prime(x) & = &  r^\prime (x) f(x+t) + r(x) f^\prime(x+t), \\
 (\Phi_t f)^{\prime \prime} (x) & = & r^{\prime \prime} (x) f(x+t) + 2
 r^\prime (x) f^\prime(x+t) + r(x) f^{\prime \prime}(x+t). 
\end{eqnarray*}
By the assumptions on $g$, if $f \in \H_2$ then 
$\Phi_tf$ has weak derivatives up to order two 
and elementary calculations show that 
\begin{eqnarray*}
r^\prime (x) & = & \dfrac{g(x)(1-G(x+t)) - (1-G(x))g(x+t)}{(1-G(x))^2}
 = r(x)\left( h(x) - h(x+t)\right),  \\
r^{\prime \prime}(x) & = & r(x) \left( \dfrac{g^\prime(x)}{1-G(x)} + h^2(x) -
  \dfrac{g^\prime(x+t)}{1-G(x+t)} - h^2(x+t)\right) \\
& &  + r^\prime (x) \left( h(x) - h(x+t) \right).  
\end{eqnarray*}
Clearly, $\nrm{r}_{\infty} \leq 1$ and, due to the assumed boundedness 
of $h$ and $g^\prime/(1-G)$, it follows that there exists $C \in [1,\infty)$ 
such that $\nrm{r^\prime}_{\infty} \leq C$ and 
$\nrm{r^{\prime \prime}}_\infty \leq C$. 
The above observations, when combined, show that 
\begin{eqnarray*}
\nrm{(\Phi_t f)}_{\L_2} & \leq & \nrm{f}_{\L_2} \leq \nrm{f}_{\H_2}, \\
 \nrm{(\Phi_t f)^\prime}_{\L_2} & \leq & 
\sqrt{2} C (\nrm{f}_{\L_2} + \nrm{f^\prime}_{\L_2}) \leq 
\sqrt{2} C \nrm{f}_{\H_2},  \\
\nrm{(\Phi_t f)^{\prime \prime}}_{\L_2} &\leq & 
4\sqrt{2} C \left(\nrm{f}_{\L_2} + \nrm{f^\prime}_{\L_2}
 + \nrm{f^{\prime \prime}}_{\L_2} \right)  \leq 
4\sqrt{2} C \nrm{f}_{\H_2}, 
\end{eqnarray*} 
which shows that $\nrm{\Phi_t f}_{\H_2} \leq \tilde{C} \nrm{f}_{\H_2}$ for 
some finite constant $\tilde{C}$.  This shows that 
$\Phi_t f \in \H_2$ and that, for any $t > 0$, 
 the map from $\H_2$ to $\H_2$ that takes $f$ to $\Phi_t f$ is 
Lipschitz continuous (with constant $\tilde{C}$).  This, in turn, trivially  
implies that for $\nu \in \H_{-2}$, the linear functional on $\H_2$ given by 
$S^\nu: f \mapsto \nu(\Phi_t f)$ also lies in $\H_{-2}$ and 
that the map from $\H_{-2}$ to itself that takes $\nu$  to $S^\nu$
is also Lipschitz continuous with the same constant.  This proves the 
second property and therefore the lemma. 
\end{proof}

 \renewcommand{\theequation}{C.\arabic{equation}}
\renewcommand\thetheorem{C.\arabic{theorem}}

\beginsec

\section{Proof of the Representation Formula}
\label{rep-prelim}

Fix $N \in \N$. We first show how (\ref{rep-hatmeasn}) 
can be deduced from (\ref{eq-dprelimit1}); the proof of 
how to obtain (\ref{rep-measn})  from 
(\ref{eqn-prelimit1}) is  analogous (in fact, a bit 
simpler), and is therefore omitted. 
Let $\tilde{\Omega}$ be a set of full $\P$-measure such that on 
$\tilde{\Omega}$, 
$\fdcompn(t)$, $\fdn(t)$, $\fbaren_{\f1}(t)$ and $\fkn(t)$ are finite 
for all $t \in [0,\infty)$.  
Fix $\omega \in \tilde{\Omega}$ and let 
$\gamma$ and $h \hmeasn$ be 
the linear functionals on ${\cal C}_c([0,\endsup) \times [0,\infty))$ 
defined, respectively, by 
\[ \gamma (\newf) \doteq \int_{[0,\endsup)} \newf (x,0)\, \hmeasn_0 (dx)  - 
\underset{[0,\endsup) \times [0,\infty)}{\int\int} \newf(x,s) \, \chatmn (dx,
ds)  + \int_{[0,\infty)} \newf (0,s) \, d\hkn (s) 
\]
and  
\[ 
h \hmeasn (\newf) \doteq \int_0^\infty 
\lan h(\cdot) \newf(\cdot, s), \hmeasn_s \ran \, ds 
\]
for $\newf \in {\cal C}_c([0,\endsup) \times [0,\infty))$. 
Since the total variation of 
$\hmeasn_0$ on $[0,\endsup)$ is bounded by $2 \sqrt{N}$, 
the total variation of $\chatmn_t (\f1)$ is bounded by 
$\sqrt{N}(\fdn(t) + \fdcompn_{\f1}(t))$ and 
the total variation of   $\hatkn$  on $[0,t]$ is 
bounded by $\sqrt{N}(\fkn(t) + \fk(t))$, for any 
$\newf \in {\cal C}_c([0,\endsup) \times [0,\infty))$ such that 
$\supp(\newf) \subset [0,\endsup) \times [0,t]$, we have 
\[  \left| \gamma (\newf) \right| \leq \sqrt{N} \nrm{\newf}_{\infty} 
\left( 2  + \fdn(t) + \fdcompn_{\f1} (t) + \fkn(t) + \fk(t)\right) 
\]
and, likewise, it can be argued that 
\[ \left| h \hmeasn (\newf) \right| \leq 
\sqrt{N}   \nrm{\newf}_{\infty} \left( \fdcompn_{\f1} (t) + \fdcomp_{\f1}(t)\right). 
\]
This shows that $\gamma$ and $h \hmeasn$ define Radon measures on $[0,\endsup)
\times [0,\infty)$.
Let $\tilde{{\cal C}}_c^{1,1}$ be the space of continuous functions with
compact support on $[0,\endsup) \times [0,\infty)$ such that the directional 
derivative $\newf_x + \newf_s$ exists and is continuous.  
Now, for every 
$\newf \in \tilde{{\cal C}}_c^{1,1}$, sending 
$t \ra \infty$ in (\ref{eq-dprelimit1}), the left-hand side of
(\ref{eq-dprelimit1}) vanishes because $\newf$ has compact support, 
and we obtain  
\[ -\int_0^\infty \lan \newf_x(\cdot, s) + \newf_s(\cdot, s), \hmeasn_s \ran \,
ds =  - h \hmeasn (\newf) + \gamma (\newf). 
\]
Since $\{\hmeasn_t,t \geq 0\} \in {\cal D}_{{\cal M}[0,\endsup)}[0,\infty)$, 
the last equation shows that 
$\{\hmeasn_t,t \geq 0\}$ satisfies the so-called abstract age 
equation  for $\gamma$, as introduced in 
Definition 4.9  of Kaspi and Ramanan \cite{KasRam07}. 
 Therefore, by  Corollary 4.17 and (4.24) of 
\cite{KasRam07},
 it follows  that for every $f \in {\cal C}_c[0,\endsup)$, 
 $\lan f, \hmeasn_t \ran = 
\gamma(\newf_t^f)$, $t \geq 0$,  
where 
\[  \newf_t^f (x,s) = \psi_{h}^{-1}(x,s) f(x+t-s) \psi_h(x+t-s,t),  \quad
(x,s) \in [0,\endsup) \times [0,t],
\]
where $\psi_h$ is the function  defined in (4.53) of \cite{KasRam07}, and
reproduced as equation (\ref{eq-psih}) of this  paper. 
Elementary algebra (specifically combining the  relations in (\ref{psil-ids}) with the definition 
(\ref{def-thetat}) of $\Psi_t$) then shows that 
$\newf^f_t(x,s) = \Psi_t f(x,s)$.  
For $f \in {\cal C}_c[0,\endsup)$, 
 the representation (\ref{rep-hatmeasn}) 
 is then obtained by 
expanding $\gamma(\Psi_tf)$  using the definition of 
$\gamma$ given above together with the relations  
$(\Psi_t f)(\cdot, 0) = \Phi_t f$,  
$(\Psi_t f) (0,\cdot) = f(t-\cdot) (1-G(t-\cdot))$, 
$\hathn_t(f) = \chatmn_t(\Psi_t f)$ and the definition 
(\ref{def-hcalkn}) of $\hcalkn$. 
Since the right-hand side of (\ref{rep-hatmeasn}) is well defined 
for $f \in {\cal C}_b[0,\endsup)$, a standard approximation 
argument can then be used to show that 
 the representation (\ref{rep-hatmeasn}) holds for  
$f \in {\cal C}_b[0,\endsup)$.

 \renewcommand{\theequation}{D.\arabic{equation}}
\renewcommand\thetheorem{D.\arabic{theorem}}

\beginsec

\section{Some Moment Estimates}

\label{subs-bdint}

In this section, we prove the  estimates stated in 
Lemma \ref{lem-bdint}. 

\begin{proof}[Proof of Lemma \ref{lem-bdint}]
Fix $N \in \N$ and $T < \infty$ and, for conciseness, let
$\overline{Y}^{(N)}(s) = (\ren(s), \fxn(s), \fmeasn_s)$, $s \in [0,\infty),$
represent the state process. 
Using the fact that $\fmartn_{\f1} = \fdn - \fdcompn_{\f1}$ is a martingale and taking 
expectations of both sides of the
inequality (5.30) of Kaspi and Ramanan \cite{KasRam07}, with $t$ and $\delta$
replaced by $0$ and $T$, respectively, it follows that
\be
\label{bdint-1}
 \E_{\overline{Y}^{(N)}(0)}
\left[\fdcompn_{\f1}(T) \right] = \E_{\overline{Y}^{(N)}(0)}
\left[ \fdn(T)\right]
\leq U(T),
\ee
where $U$ is the renewal function associated with $G$. 
This shows that the inequality (\ref{eq-bdint}) holds for $k = 1$.
We proceed by induction.  
Suppose that (\ref{eq-bdint}) holds with $k = j-1$ for
some integer $j \geq 2$.
Then we can write 
\begin{eqnarray*}
 \left( \fdcompn_{\f1} (T)  \right)^{j}
& =  & \int_0^T \cdots  \int_0^T \left( \lan h, \fmeasn_{s_1} \ran
\,  \lan h, \fmeasn_{s_2} \ran \,  \ldots \lan h, \fmeasn_{s_j} \ran
\right) \, ds_j \ldots ds_1  \\
& = & \int_0^T \lan h, \fmeasn_{s_1} \ran \left( \int_0^T \ldots \int_0^T
\left( \lan h, \fmeasn_{s_2} \ran \, \ldots  \lan h, \fmeasn_{s_{j}} \ran \right)
\, ds_j \ldots ds_{2} \right) ds_1 \\
& = & j \int_0^T  \lan h, \fmeasn_{s_1} \ran \left( \int_{s_1}^T \ldots
\int_{s_1}^T
\left( \lan h, \fmeasn_{s_2} \ran \,  \ldots \lan h, \fmeasn_{s_{j}} \ran
\right) \,
ds_j \ldots ds_{2} \right) \, ds_1 \\
& = & j \int_{0}^T \lan h, \fmeasn_{s_1} \ran \left( \fdcompn_{\f1} (T)  -
\fdcompn_{\f1} (s_1) \right)^{j-1} \, ds_1.
\end{eqnarray*}
Taking  expectations of both sides above and applying 
Tonelli's theorem we obtain 
\[
 \E_{\overline{Y}^{(N)}(0)}
\left[\left( \fdcompn_{\f1} (T)  \right)^{j}\right]  =  
j \int_{0}^T  \E_{\overline{Y}^{(N)}(0)}
\left[\lan h, \fmeasn_{s_1} \ran \left( \fdcompn_{\f1} (T)  -
\fdcompn_{\f1} (s_1) \right)^{j-1}\right] \, ds_1. 
\]
For each $s_1 \in [0,T]$, due to the Markov property 
of $\overline{Y}^{(N)}$ established in Lemma B.1 of Kang and 
Ramanan \cite{KanRam08} we obtain 
\[ 
\begin{array}{l}
 \E_{\overline{Y}^{(N)}(0)}
\left[\lan h, \fmeasn_{s_1} \ran \left( \fdcompn_{\f1} (T)  -
\fdcompn_{\f1} (s_1) \right)^{j-1}\right] \\
\qquad \ds =  
 \E_{\overline{Y}^{(N)}(0)} \left[ \E_{\overline{Y}^{(N)}(0)} \left[  
\lan h, \fmeasn_{s_1} \ran \left( \fdcompn_{\f1} (T)  -
\fdcompn_{\f1} (s_1) \right)^{j-1} |{\cal F}_{s_1}^{(N)} \right]
\right] \\[1.0em]
\qquad \ds =  \E_{\overline{Y}^{(N)}(0)} \left[\lan h, \fmeasn_{s_1} \ran
\E_{\overline{Y}^{(N)}(s_1)} \left[  \left( \fdcompn_{\f1} (T-s_1)
  \right)^{j-1}\right] \right]. 
\end{array}
\]
Applying the induction assumption to the last term above, it follows
 that 
\[  \E_{\overline{Y}^{(N)}(0)}
\left[\lan h, \fmeasn_{s_1} \ran \left( \fdcompn_{\f1} (T)  -
\fdcompn_{\f1} (s_1) \right)^{j-1}\right] \leq 
 (j-1)! U(T)^{j-1} \E_{\overline{Y}^{(N)}(0)} \left[\lan h,
  \fmeasn_{s_1} \ran \right]. 
\]
Combining the last three displays,  applying  
Tonelli's theorem again and using  (\ref{bdint-1}),
 we obtain 
\[ \E_{\overline{Y}^{(N)}(0)}
\left[\left( \fdcompn_{\f1} (T)\right)^j\right] \leq  
j! U(T)^{j-1}\E_{\overline{Y}^{(N)}(0)}\left[ \int_0^T \lan h,
  \fmeasn_{s_1} \ran \, ds_1 \right] 
\leq  j! U(T)^j. 
\]
This shows that (\ref{eq-bdint}) is also satisfied for $k = j$ 
and hence, by induction, for all positive integers $k$.

We now turn to the proof of the second bound. 
We can assume without loss of generality that 
 $\newf^* h$ is integrable on $[0,\endsup)$ because 
otherwise the inequality holds trivially.   
On substituting $l = \newf^* h$, $\newf = \f1$, $r = 0$ and $t = T$ in 
(5.31) of Proposition 5.7 of Kaspi and Ramanan \cite{KasRam07}, for every 
$N \in \N$, we have 
\be
\label{obs-1} \E_{\overline{Y}^{(N)} (0)} \left[
 \fdcompn_{\newf} (T) \right] \leq 
\E_{\overline{Y}^{(N)} (0)} \left[
 \fdcompn_{\newf^*} (T) \right]
\leq  C_1(T) \left( \int_{[0,\endsup)} \newf^*(x) h(x) \, dx \right),
\ee
where 
\[ C_1(T) \doteq \sup_{N} \E[ \fxn(0)  + \fen(T) ] \leq
C(T) \doteq \sup_{N} \sup_{s \in [0,T]} \E[\fxn(s) + \fen(T)], 
\]
which is finite by Theorem \ref{th-flimit}.  
Given (\ref{obs-1}), the same inductive argument used in the 
proof of the first assertion of the lemma can then be used to 
complete the proof of the second bound.

Next, note that if Assumptions \ref{as-flinit} and \ref{as-h} hold, then 
 $\fdcompn_{\f1} \Rightarrow \fdcomp_{\f1}$  by Proposition 5.17 of 
\cite{KasRam07} and $\fdcomp$ is continuous. 
Together with  
the Skorokhod representation theorem, 
Fatou's lemma and the inequality  (\ref{bdint-1}), 
this implies that 
\[ \fdcomp_{\f1} (T)  \leq \liminf_{N \ra \infty} \E\left[\fdcompn_{\f1}(T) \right]\leq \limsup_{N \ra \infty} \E\left[\fdcompn_{\f1}(T) \right] 
\leq U(T). 
\]
The inequality (\ref{eq-bdintl}) can now be deduced from this inequality 
exactly as the inequality (\ref{eq-bdint}) was deduced from 
the inequality  (\ref{bdint-1}), though the proof is in fact much 
 simpler because $\fdcomp_{\f1}$ is deterministic. 
\end{proof}

%

 \renewcommand{\theequation}{E.\arabic{equation}}
\renewcommand\thetheorem{E.\arabic{theorem}}

\beginsec

 \renewcommand{\theequation}{E.\arabic{equation}}
\renewcommand\thetheorem{E.\arabic{theorem}}

\section{Proof of Consistency} 
\label{sec-consistency}

We first start by establishing some Fubini theorems.  

\begin{lemma}
\label{lem-fub} 
Let Assumptions \ref{as-flinit}-\ref{as-holder} 
be satisfied, let $g$ be  continuous and 
let $\hath$ and $\hcalk$ be defined as in (\ref{def-hath}) and
(\ref{def-hcalk}), respectively. 
Suppose $\tilde{\newf} \in 
{\cal C}_b([0,\endsup) \times [0,\infty)$. 
Given the family of mappings $\opint_t$, $t \geq 0$,  defined in 
(\ref{def-opint}),  we have 
\be
\label{fub-1}
 \chatm_t \left(\opint_t \tilde{\newf} \right) 
= \int_0^t \chatm_r (\Psi_r (\tilde{\newf}(\cdot,r)) \, dr 
= \int_0^t \hath_r (\tilde{\newf}(\cdot,r)) \, dr. 
\ee
 If, in addition, 
for every $t \in [0,T]$,  $x \mapsto \int_0^t \tnewf(x,r) \, dr$ 
is bounded and H\"{o}lder continuous, uniformly in $t$, then 
for every $s \geq 0$, almost surely  for every $t \geq 0$,  
\be
\label{fub-0}
\chatm_s \left( \int_0^t \Psi_s\left(  \tilde{\newf} (\cdot,r)\right) \, dr \right) 
=  \int_0^t \chatm_s (\Psi_s(\tilde{\newf}(\cdot,r))) \, dr = 
\int_0^t \hath_s (\tilde{\newf}(\cdot,r)) \, dr. 
\ee
 Moreover, if either $x \mapsto \tilde{\newf}(x,r)$ is absolutely 
continuous  for every $r > 0$ 
 and $(x,r) \mapsto \tilde{\newf}_x(x,r)(1-G(x))$ is 
locally integrable on $[0,\endsup) \times [0,\infty)$, or 
$g$ is absolutely continuous and $\tilde{\newf} \in {\cal C}_b([0,\endsup)
\times [0,\infty))$, 
then almost surely, for every $s, t \geq 0$, 
\be
\label{fub-2}
  \hcalk_s \left(\int_0^t \tilde{\newf}(\cdot,r) \, dr \right) = \int_0^t 
 \hcalk_s (\tilde{\newf}(\cdot,r)) \, dr. 
\ee
\end{lemma}  
\begin{proof} 
Fix $s, t \geq 0$.  
 Then 
\be
\label{eq-fubone} 
\int_0^t \Psi_s (\tilde{\newf}(\cdot,r)) \, dr 
= \Psi_s (\int_0^t\tilde{\newf}(\cdot,r)) \, dr ) 
\ee
 and so, by  
the inequality (\ref{op-ineqs}) and the boundedness assumption on $\tilde{\newf}$, 
 $\int_0^t \Psi_s (\tilde{\newf}(\cdot,r)) \, dr$ and  
$\opint_t\tilde{\newf}$ are uniformly bounded on  
$[0,\endsup) \times [0,t]$. 
We can thus apply 
 Fubini's theorem for stochastic integrals with respect to 
martingale measures (see Theorem 2.6 of \cite{walshbook}) to 
conclude that  almost surely, (\ref{fub-1}) and (\ref{fub-0}) are satisfied.  
We now have to show that the set of measure zero on which they 
are not satisfied is independent of $t$, for which it suffices 
to show that the processes on both sides of 
(\ref{fub-1}) and (\ref{fub-0}) are continuous in $t$. 
The processes on the 
right-hand sides of (\ref{fub-1}) and (\ref{fub-0}) are clearly 
continuous in $t$, whereas the continuity in $t$ of 
the process on the left-hand side of  (\ref{fub-1}) follows from 
property 4 of Lemma \ref{lem-hathn}. 
Because of (\ref{eq-fubone}), the relation $\chatm_s(\Psi_s f)  = 
\hath_s(f)$ and the assumed boundedness and uniform H\"{o}lder continuity of 
$x \mapsto \int_0^t \tilde{\newf}(x,r) \, dr$, the continuity in $t$ 
of the left-hand side of (\ref{fub-0}) follows from property 3 of 
Lemma \ref{lem-hathn}. 
Thus, for any given $s > 0$ there exists a set of full $\P$-measure on which 
(\ref{fub-0}) and (\ref{fub-1}) hold simultaneously for 
all $t \geq 0$. 

Next,  by the definition of $\hcalk$ in (\ref{def-hcalk}), note that 
$\hcalk_s (\int_0^t \tilde{\newf}(\cdot,r) \, dr)$ is equal to 
\[ 
\ds \left(\int_0^t \tilde{\newf}(0,r) \, dr \right) \hatk(s) 
+ \int_0^s \hatk(u)\dfrac{\partial}{\partial x} \left. \left( \int_0^t ((1-G(x))
  \tilde{\newf}(x,r) \, dr \right)\right|_{x=s-u}  \, du. 
\]
By  the stated assumptions, it follows that $g$ is continuous and 
 for each $r > 0$, 
the function $x \mapsto (1-G(x))\tilde{\newf}(x,r)$ 
is absolutely continuous and  
its derivative (with respect to $x$) is locally integrable. 
Moreover, 
by Theorem \ref{th-main},  $\hatk$ is almost surely continuous and 
thus locally bounded. 
Thus, we can first  exchange the order of differentiation and integration 
and then apply Fubini's theorem for Lebesgue integrals 
in the last display to conclude that 
$\hcalk_s (\int_0^t \tilde{\newf}(\cdot,r) \, dr)$ is equal to 
\[
\begin{array}{l}
  \ds \int_0^t \tilde{\newf}(0,r) \hatk(s) \, dr 
+ \int_0^s \hatk(u)  \left( \int_0^t\dfrac{\partial}{\partial x} \left. \left((1-G(x))
  \tilde{\newf}(x,r)\right)\right|_{x=s-u} \, dr \right) \, du \\
\qquad =  \ds \int_0^t \tilde{\newf}(0,r) \hatk(s) \, dr 
+ \int_0^t  \left(\int_0^s\hatk(u)\dfrac{\partial}{\partial x} \left. \left((1-G(x)) \tilde{\newf}(x,r)\right)\right|_{x=s-u} \, du \right) \, dr \\
\qquad \ds =  \int_0^t \hcalk_s (\tilde{\newf}(\cdot,r)) \, dr, 
\end{array}
\]
which completes the proof of the lemma. 
\end{proof}

We now  prove the consistency lemma.

\begin{proof}[Proof of Lemma \ref{lem-consistency}] 
Fix  $f \in \testspace$ and $s, t \geq 0$. 
 Then, replacing $t$ by $t+s$ in (\ref{rep-hatmeasn}),
we obtain 
\be
\label{consist1} 
\ds \hmeasn_{s+t}(f)  =  \ds \ops_{s+t}^{\hmeasn_0} (f)- \hathn_{t+s}(f) +
\hcalkn_{t+s}(f).  
\ee 
Using the shift relations introduced  in 
(\ref{def-shift1})--(\ref{def-shift3}),
and recalling the definitions of $\hathn$ and $\hcalkn$  in (\ref{def-hathn})
and (\ref{eq-hcalkn}), respectively, 
 the last two terms on the right-hand side of (\ref{consist1}) 
can be decomposed as follows:  
\be
\label{consist3}
\ds \hathn_{t+s}(f)  =  \chatmn_{t+s} (\Psi_{t+s} f) = \chatmn_{s}
(\Psi_{t+s}(f)) + (\Theta_s \chatmn)(\Psi_t f), 
\ee
and, similarly, 
\begin{eqnarray}
\nonumber
\ds \hcalkn_{t+s}(f) & = & \int_{[0,s+t]} f(s+t-u) (1-G(s+t-u)) \, d\hatkn(u)\\
\label{consist2}
& = &  \int_{[0,s]} f(s+t-u)  (1-G(s+t-u)) \, d\hatkn(u) \\
\nonumber
& & \quad + \int_{[0,t]} (1-G(t-u)) f(t-u) \, d(\Theta_s \hatkn) (u).  
\end{eqnarray} 
On the other hand, since  
$\Phi_t f \in {\cal C}_b[0,\endsup)$, 
replacing $f$ and $\hmeasn_0$ in 
 (\ref{rep-hatmeasn})  by $\Phi_t f$ and $\hmeasn_s$, respectively, and 
using
the semigroup property (\ref{eq-sgroup}) and the fact that 
$\Psi_s \Phi_t = \Psi_{s+t}$  on the appropriate domain 
as specified in (\ref{rel-psiphi}),  
we obtain 
\begin{eqnarray}
\nonumber
\ds \ops_t^{\hmeasn_s} (f) = \lan \Phi_t f, \hmeasn_s \ran   & = &     
\lan \Phi_{s+t} f, \hmeasn_0 \ran - \chatmn_s(\Psi_s \Phi_t f) \\ 
\nonumber
&& \quad + \int_{[0,s]} (\Phi_t f) (s-u)   (1 - G(s-u))\, d\hatkn (u) \\
\label{consist4}
 &= &  \ds \ops^{\hmeasn_0}_{s+t}(f) - \chatmn_s(\Psi_{s+t} f) \\
\nonumber
& & \quad +  
\int_{[0,s]}  f(s+t-u) (1 - G(s+t-u)) \, d\hatkn (u).
\end{eqnarray} 
The relation (\ref{eq-consist0}) is then obtained by   
subtracting (\ref{consist4}) from  (\ref{consist1}), rearranging 
terms  and using 
the relations (\ref{consist2}) and (\ref{consist3}). 

Now, suppose that Assumptions \ref{as-flinit}--\ref{as-holder} are satisfied
 and further, 
assume that $g$ is continuous.  Then  Theorem \ref{th-fclt} shows that 
the limit $\hmeas$ of $\{\hmeasn\}$ is a continuous $\H_{-2}$-valued process 
that is given explicitly by (\ref{def-hmeas}).  
The shifted equation (\ref{eq-consist00}) for the limit $\hmeas$
 is proved in a similar 
fashion as for the corresponding quantity $\hmeasn$ in the $N$-server system, 
except that now $\hcalk$ has the slightly different 
representation (\ref{def-hcalk}).   We fill in the details for completeness.  
Applying (\ref{def-hmeas}) with $t$ replaced by $t+s$, we see that for bounded 
and absolutely continuous $f$, 
\[ \hmeas_{t+s} (f) = \ops_{t+s}^{\hmeas_0}(f) - \hath_{t+s}(f) + f(0)
\hatk(s) 
+ \int_0^{t+s} \hatk(u) \trans_{f} (t+s-u) \, du. 
\]
On the other hand, applying (\ref{def-hmeas}) with $f$ and $t$, respectively, 
replaced by $\Phi_tf$ and $s$ and using the semigroup relation
(\ref{eq-sgroup}) for $\Phi_t$ and the fact that $(\Phi_t f)(0) =
f(t)(1-G(t))$,  we obtain  
\be
\label{neweq6}
\begin{array}{rcl}
\ds \ops_t^{\hmeas_s} (f) & = &  \ds \hmeas_s \left( \Phi_t f\right) \\ 
& = & \ds \ops_{s+t}^{\hmeas_0} (f)  - \chatm_s (\Psi_t \Phi_t f) + f(t)(1-G(t))\hatk(s) 
+ \int_0^s \hatk(u) \trans_{\Phi_t f} (s-u) \, du. 
\end{array}
\ee
Simple calculations show that $\trans_{\Phi_t f} = \trans_f (\cdot +t)$. Hence,   
\[
\ds \int_0^{t+s} \hatk(u) \newf_f(t+s-u) \, du - 
\int_0^s \hatk(u) \trans_{\Phi_t f} (s-u) \, du 
  =  \ds \int_0^t \hatk(s+u) \trans_f(t-u) \, du 
\]
and, since $\trans_f = (f(1-G))'$, 
\[ \int_0^t \hatk(s) \trans_f(t-u) \, du  =  f(0)\hatk(s) - f(t)(1-G(t))\hatk(s). 
\]
Equation (\ref{eq-consist00}) can now be obtained by  
combining the last four equations with the limit analog of (\ref{consist3}), 
in which $\hathn$ and $\chatmn$, respectively, are replaced by 
$\hath$ and $\chatm$.

To show  that (\ref{eq-consist02}) is satisfied, 
 note that  by Theorem \ref{th-main}, 
 $(\hatk, \hatx, \hmeas_0(\f1)) = \Lambda (\hate, \widehat{x}_0,
\ops^{\hmeas_0}(\f1) - \hath(\f1))$.  This implies  that 
the centered many-server 
equations (\ref{ref-meas})--(\ref{ref-nonidling}) are
satisfied with $v, Z, X, K$ and $E$, respectively, replaced 
by  $\hmeas(\f1), \ops^{\hmeas_0}(\f1) - \hath(\f1), \hatx, \hatk$ 
and $\hate$.   
Fix any $s > 0$.  Subtracting  the  equation (\ref{ref-sm}) evaluated 
at $t+s$ from the same equation evaluated at $t$, it follows that 
 (\ref{ref-sm}) also  
holds when $K, E, X$ and $v$ is replaced, respectively, 
by $\Theta_s \hatk, \Theta_s \hate, \hatx_{s+\cdot}$ and $\hmeas_{s + \cdot}(\f1)$. 
It is also clear  that (\ref{ref-nonidling}) is satisfied with 
$v$ and $X$ replaced by $\hmeas_{s+t}(\f1)$  and $\hatx_{s+t}$ for all $t \geq 0$. 
Lastly, substituting $f = \f1$ in (\ref{eq-consist00}), using  
the definition (\ref{def-shift4}) of $\Theta_s \hcalk$ and the fact that 
$\newf_{\f1} = -g$, it follows  that 
 (\ref{ref-meas}) holds with $v, Z$ and $K$,  respectively, replaced, 
 by $\hmeas_{s+\cdot}(\f1), \ops^{\hmeas_s}(\f1)- \Theta_s \hath(\f1)$ and 
$\Theta_{s} \hatk$.  This proves (\ref{eq-consist02}).

Fix $s > 0$. We first need to show that Assumption \ref{as-hate} 
is satisfied when $\haten$ and $\hate$, respectively, are replaced by 
$\Theta_s\haten$ and $\Theta_s \hate$.   This is easily deduced 
using basic properties of renewal processes and Poisson processes 
and is thus left to the reader.    Next, we show that 
 Assumption \ref{as-diffinit}' is satisfied when $\widehat{x}^{(N)}_0$, 
$\widehat{x}_0, \hmeasn_0$,
 and $\hmeas_0$, respectively,  
are replaced by $\hatxn(s), \hatx(s), 
\hmeasn_s$ and $\hmeas_s$.  
By definition (\ref{def-hmeas}), $\hmeas_s(f)$ is a well defined 
random variable for all $f \in \acbl$ and, since $\Phi_t f \in 
\acbl$ if $f \in \acbl$ it follows that 
$\ops_t^{\hmeas_s}(f) = \hmeas_s(\Phi_t f)$ is also a well 
defined random variable for every $f \in \acbl$.  
Now, due to (\ref{def-hmeas}), 
 the assumed measurability of $f \mapsto \ops^{\hmeas_0}_s(f)$ 
and the joint measurability of the maps
 $(s,f) \mapsto \chatm_s(f)$ and $(s,f) \mapsto \hcalk_s(f)$ for 
$f \in {\cal C}_b[0,\endsup)$, 
which is a consequence of the definition of these stochastic integrals, 
it follows that almost surely the map $f \mapsto \hmeas_s (f)$ 
from $\acbl \subset {\cal C}_b[0,\endsup)$ to $\R$ (both 
equipped with their respective Borel $\sigma$-algebras) is also 
measurable.    
Moreover, for a given $\varphi \in {\cal C}_b([0,\endsup) \times [0,T])$ and 
$t > 0$, 
the maps  
$r \mapsto \Phi_r \varphi(\cdot,r)$
from  
$[0,T]$ to ${\cal C}_b[0,\endsup)$ 
and   $f \mapsto \Phi_t f$ 
from ${\cal C}_b[0,\endsup)$ to itself 
are clearly also measurable (they are in fact continuous). 
Since the composition of measurable maps is measurable,
 it follows that almost surely, 
both  $r \mapsto 
\hmeas_s (\Phi_r(\varphi(\cdot,r)))$ and 
$f \mapsto \ops^{\hmeas_s}_t(f) = \hmeas_s(\Phi_tf)$ are measurable.

In addition, due to the continuity of the limit in the convergence 
(\ref{conv-fclt}) established in Proposition \ref{prop-fclt}, it 
follows  that  
\[ (\hatxn(s), \hmeasn_{s+\cdot}, \Theta_s \hcalkn, \Theta_s \hathn) \Rightarrow 
( \hatx(s), \hmeas_{s+\cdot}, \Theta_s \hcalk,\Theta_s \hath) \]
 in the space $\R \times {\cal D}_{\H_{-2}}^3[0,\infty)$. In particular, this implies
that $\hmeas_s$ has an $\H_{-2}$-valued version, 
and so property (a) of Assumption \ref{as-diffinit} is satisfied. 
Next, by (\ref{eq-consist0}) and (\ref{eq-consist00}), 
$\ops^{\hmeasn_s}$  can be expressed as a linear combination of the 
$\H_{-2}$-valued processes 
 $(\hmeasn_{s+\cdot}, \Theta_s \hcalkn, \Theta_s \hathn)$ and, 
 likewise, $\ops^{\hmeas_s}$  is the same linear combination of 
$(\hmeas_{s+\cdot}, \Theta_s \hcalk, \Theta_s \hath)$. 
Therefore, the continuity of $\hmeas_{s+\cdot}, \Theta_s \hcalk$ and 
$\Theta_s \hath$ show that $\ops^{\hmeas_s}$ is a continuous 
$\H_{-2}$-valued process. The same logic used above then shows 
that  the real-valued process $\ops^{\hmeas_s} (\f1)$ 
is continuous and that the limits in property (c) of Assumption 
\ref{as-diffinit} holds.  Thus, we have established 
that Assumption \ref{as-diffinit} continues to hold 
at a shifted time.

It only remains to show that
 Assumption \ref{as-diffinit}' is satisfied when 
$\hmeas_0$ is replaced by $\hmeas_s$.  Fix  
$\newf \in {\cal C}_b([0,\endsup) \times [0,\infty))$ such that 
$x \mapsto \newf(x,r)$ is absolutely continuous and  H\"{o}lder continuous and 
$\newf_x$ is integrable on $[0,\endsup) \times [0,T]$ for any 
$T < \infty$.  
We will make repeated use of the  semigroup property  
$\Phi_{s}\circ \Phi_{r} = \Phi_{s+r}$, the 
relation $\Psi_{s+r} = \Psi_s \circ \Phi_{r}$
 on the appropriate domain 
as specified in (\ref{rel-psiphi}),  
 and the 
form (\ref{def-hcalk}) of $\hcalk$, without explicit mention.  
Then,  by  the other assumptions on $h$, 
for any $r > 0$, $\Phi_{r} \newf(\cdot,r)$ and 
$\int_0^t \Phi_{r} \newf (\cdot,r)\, dr$ are both bounded, H\"{o}lder
continuous and absolutely 
continuous functions on $[0,\endsup)$. 
Therefore, substituting $f = \Phi_{r} \newf(\cdot, r)$ into
(\ref{def-hmeas}) 
 we see that
\be
\label{eq-phirh}
\hmeas_s (\Phi_{r} \newf(\cdot,r))  =  
 \hmeas_0 (\Phi_{s+r} \newf(\cdot,r)) 
- \chatm_s (\Psi_{s+r} \newf(\cdot,r)) + 
\hcalk_s ( \Phi_r \newf(\cdot,r)). 
\ee  
By (\ref{rel-ops}), it follows that 
$\hmeas_0(\Phi_{s+r} \newf(\cdot,r)) = \ops_s^{\hmeas_0} (\Phi_r \newf
(\cdot,r))$.   
 Furthermore, substituting $f = \int_0^t \Phi_{r} \newf(\cdot,r) \, dr$
into (\ref{def-hmeas}),  invoking Assumption \ref{as-diffinit}'(d) with 
$\newf(\cdot,r)$ replaced by $\Phi_s \newf(\cdot,r)$, and applying the Fubini theorems 
(\ref{fub-0}) and (\ref{fub-2}) with the 
absolutely continuous and 
uniformly bounded function 
$\tilde{\newf}(x,r) = (\Phi_r
\newf(\cdot,r))(x)$, it follows that 
$\hmeas_s \left(\int_0^t \Phi_{r} \newf(\cdot,r) \, dr\right)$ is equal to  
\begin{eqnarray*}
 \quad \hmeas_0 \left( \int_0^t\Phi_{s+r} \newf(\cdot,r) \, dr \right) 
- \chatm_s\left( \int_0^t \Psi_{s+r}(\newf(\cdot,r)) \, dr \right) 
 + \hcalk_s \left(\int_0^t \Phi_r \newf (\cdot,r) \, dr \right) \\
 =  \int_0^t \hmeas_0 (\Phi_{s+r} \newf(\cdot,r)) \, dr -  
\int_0^t \chatm_s (\Psi_{s+r} \newf(\cdot,r)) \, dr  + 
\int_0^t \hcalk_s (\Phi_r \newf (\cdot,r)) \, dr. 
\end{eqnarray*}
A comparison with (\ref{eq-phirh}) shows that the 
right-hand side above equals $\int_0^t \hmeas_s (\Phi_r \newf(\cdot,r)) \, dr$. 
Thus, Assumption \ref{as-diffinit}'(d) holds with $\hmeas_0$ replaced 
by $\hmeas_s$. 
\end{proof}

 \noi
 {\bf Acknowledgments. }  We are grateful
to Leonid Mytnik for some useful discussions and to
Avi Mandelbaum for bringing this open problem to our attention.

\end{document}